\documentclass[12pt]{amsart}
\usepackage{amsmath,amsthm,amsfonts,amssymb,enumerate,color,url,hyperref}
\usepackage[all]{xy}

\usepackage[top=25mm,right=30mm,bottom=25mm,left=30mm]{geometry}

\newcommand{\C}{\mathbb{C}}
\newcommand{\F}{\mathbb{F}}
\newcommand{\N}{\mathbb{N}}
\newcommand{\Z}{\mathbb{Z}}

\newcommand{\Q}{\mathbb{Q}}

\newcommand{\GL}{\mathrm{GL}}
\newcommand{\SL}{\mathrm{SL}}
\newcommand{\SU}{\mathrm{SU}}
\newcommand{\PGU}{\mathrm{PGU}}
\newcommand{\Sp}{\mathrm{Sp}}

\newcommand{\PSL}{\mathrm{PSL}}
\newcommand{\PSU}{\mathrm{PSU}}
\newcommand{\GO}{\mathrm{GO}}
\newcommand{\SO}{\mathrm{SO}}
\newcommand{\PCO}{\mathrm{PCO}}
\newcommand{\Spin}{\mathrm{Spin}}
\newcommand{\Stab}{\mathrm{Stab}}
\newcommand{\Spec}{\mathrm{Spec}\,}
\newcommand{\Frob}{\mathrm{Frob}}
\newcommand{\Tr}{\mathrm{Tr}}
\newcommand{\Cl}{\mathrm{Cl}}

\newcommand{\Aut}{\mathrm{Aut}}

\newcommand{\vG}{\underline{G}}
\newcommand{\vS}{\underline{S}}

\newcommand{\vW}{\underline{W}}
\newcommand{\vX}{\underline{X}}
\newcommand{\vY}{\underline{Y}}
\newcommand{\vZ}{\underline{Z}}
\newcommand{\schG}{\mathcal{G}}
\newcommand{\schX}{\mathcal{X}}
\newcommand{\schS}{\mathcal{S}}
\newcommand{\schY}{\mathcal{Y}}
\newcommand{\schZ}{\mathcal{Z}}
\newcommand{\cD}{\mathcal{D}}
\newcommand{\cL}{\mathcal{L}}

\newcommand{\St}{\mathsf {St}}
\newcommand{\ZB}{\mathbf{Z}}
\newcommand{\Irr}{\mathrm{Irr}}

\newcommand{\e}{\epsilon}
\newcommand{\gam}{\gamma}
\newcommand{\supp}{\mathsf {supp}}
\newcommand{\diag}{\mathrm {diag}}
\newcommand{\AAA}{\mathsf{A}}
\newcommand{\SSS}{\mathsf{S}}
\renewcommand{\Pr}{\mathbf{P}}
\newcommand{\bfU}{\mathbf{U}}
\newcommand{\sgn}{\mathsf{sgn}}
\newcommand{\tw}[1]{{}^#1\!}
\renewcommand{\mod}{\bmod \,}

\newtheorem{theor}{Theorem}
\newtheorem{thm}{Theorem}[section]
\newtheorem{quest}[theor]{Question}

\newtheorem{cor}[thm]{Corollary}

\newtheorem{prop}[thm]{Proposition}

\newtheorem{defn}[thm]{Definition}

\newtheorem*{thmA}{Theorem A}
\newtheorem*{thmB}{Theorem B}

\newtheorem{lem}[thm]{Lemma}

\newtheorem*{claim*}{Claim}

\numberwithin{equation}{section}

\begin{document}

\author{Michael Larsen}
\address{Department of Mathematics\\
    Indiana University\\
    Bloomington, IN 47405\\
    U.S.A.}
\email{mjlarsen@indiana.edu}

\author{Aner Shalev}
\address{Einstein Institute of Mathematics\\
    Hebrew University\\
    Givat Ram, Jerusalem 91904\\
    Israel}
\email{shalev@math.huji.ac.il}

\author{Pham Huu Tiep}
\address{Department of Mathematics\\ Rutgers University\\ Piscataway, NJ 08854\\ U.S.A.}
\email{tiep@math.rutgers.edu}

\title{Products of normal subsets and derangements}

\subjclass[2010]{Primary 20D06; Secondary 20F69, 20G40, 20P05, 20B15, 20C33}

\thanks{ML was partially supported by NSF grant DMS-1702152.
AS was partially supported by ISF grant 686/17 and the Vinik Chair of mathematics which he holds.
PT was partially supported by NSF grant DMS-1840702 and the Joshua Barlaz Chair in Mathematics.
All three authors were partially supported by BSF grant 2016072.}

\maketitle

\centerline{{ {\it Dedicated to Bob Guralnick on the occasion of his seventieth birthday}}}

\begin{abstract}
In recent years there has been significant progress in the study of products of subsets of finite groups and of finite simple
groups in particular.
In this paper we consider which families of finite simple groups $G$ have the property that for each $\epsilon > 0$ there exists
$N > 0$ such that, if $|G| \ge N$ and $S, T$ are normal subsets of $G$ with at least $\epsilon |G|$ elements each, then every non-trivial element of $G$ is the product of an element of $S$ and an element of $T$.

We show that this holds in a strong sense for finite simple groups of Lie type of bounded rank, while it does not hold for alternating groups
or groups of the form $\PSL_n(q)$ where $q$ is fixed and $n\to \infty$.

Our second main result is that any element in a transitive permutation representation of a sufficiently large finite simple group
is a product of two derangements. 
\end{abstract}


\tableofcontents

\section{Introduction}

In the past two decades there has been considerable interest in the products of subsets of finite groups, especially
(nonabelian) finite simple groups. The so-called Gowers trick (see \cite{G} and \cite{NP}), which is part of the theory
of quasi-random groups, shows that the product of three large subsets of a finite group $G$ is the whole of $G$
(where large is defined in terms of $|G|$ and the minimal degree of a non-trivial irreducible representation of $G$).
See Section 7 below for details and consequences.

The celebrated Product Theorem of \cite{BGT} and \cite{PS}, which is part of the deep theory of approximate subgroups
originating in \cite{H1} and \cite{Hr}, shows that for finite simple groups $G$ of Lie type and bounded rank there exists $\e > 0$
(depending only on the rank of $G$) such that for every subset $A \subseteq G$ which generates $G$, either
$|A^3| \ge |A|^{1+\e}$ or $A^3 = G$.

Note that both the Gowers trick and the Product Theorem deal with products of three (or more) subsets. Much less is known
about products of two subsets, which is the main topic of this paper.

As for normal subsets, a longstanding related conjecture of Thompson asserts that every
finite simple group $G$ has a conjugacy class $C$ such that $C^2 = G$. In spite of considerable efforts
(see \cite{EG} and the references therein) and the proof of the related Ore Conjecture (see \cite{LOST}), Thompson's
Conjecture is still open for groups of Lie type over fields with $q \le 8$ elements. A weaker result, that all
sufficiently large finite simple groups $G$ have conjugacy classes $C_1, C_2$ such that $C_1C_2 \supseteq G \smallsetminus \{ e \}$
is obtained in \cite{LST1}; this was improved in \cite{GM}, where the same conclusion is proved for all finite
simple groups. See also \cite{Sh2}, where it is shown that, for finite simple groups $G$ and random elements
$x, y \in G$, the sizes of $x^G y^G$ and of $(x^G)^2$ are $(1-o(1))|G|$. This may be viewed as a probabilistic
approximation to Thompson's Conjecture.

For normal subsets $S$ (not equal to $\emptyset, \{ e \}$) of arbitrary finite simple groups $G$, the minimal $k>0$
such that $S^k = G$ is determined in \cite{LiSh1} up to an absolute multiplicative constant. In \cite{LSSh} it is shown
that the product of two small normal subsets of finite simple groups has size which is close to the product of their sizes.
However, this says nothing about products of two large normal subsets.

An interesting context in which the products of normal subsets of finite simple groups play a role is the Waring problem
for finite simple groups; see for instance \cite{Sh3, LS1, LS2, LOST, LST1, GT2, GLOST, LST2}, the references therein, and the monograph \cite{S}
on word width.

By a \emph{word} we mean an element $w$ of some free group $F_d$. A word $w$ and a group $G$ give rise to a word map $w: G^d \to G$
induced by substitution; its image, denoted by $w(G)$, is a normal subset of $G$ (hence a union of conjugacy classes).
The main result of \cite{LST1} is that, for non-trivial words $w_1, w_2 \in F_d$, and all sufficiently large finite simple groups $G$
we have
\begin{equation}
\label{Waring}
w_1(G)w_2(G)=G.
\end{equation}

There are various results showing that word maps $w \ne 1$ on finite simple groups $G$ have large image, see \cite{L, LS1, LS2, NP}.
In particular, it is shown in \cite{L} that $|w(G)| \ge |G|^{1-\e}$ for any $\e > 0$ provided $|G| \ge N(\e)$, and that for $G$ of Lie type
and bounded rank, there exists $\e > 0$ (depending only on the rank of $G$) such that for all words $w \ne 1$ we have $|w(G)| \ge \e|G|$.
We would like to understand to what extent \eqref{Waring} can be extended to products of arbitrary large normal subsets of finite simple groups.

Let $\epsilon > 0$ be a constant.
Let $G$ be a finite simple group and $S$ and $T$ normal subsets of $G$ such that
$|S|,\,|T|>\epsilon |G|$.  We are particularly interested in the following  questions:

\begin{quest}
\label{Weak}
Does every element in $G\smallsetminus\{e\}$ lie in $ST$ if $|G|$ is sufficiently large?
\end{quest}

\begin{quest}
\label{Strong}
Does the ratio between the number of representations of each $g\in G\smallsetminus \{e\}$ and
$\frac{|S|\,|T|}{|G|}$ tend uniformly to $1$ as $|G|\to \infty$?
\end{quest}

\begin{quest}
\label{Equal}
What happens in the special case $S=T$?
\end{quest}

We exclude the identity in Questions \ref{Weak} and \ref{Strong} because  every conjugacy class $C$ in a non-trivial finite group $G$ satisfies $|C| = \frac{|G|}{n}$ for some $n\ge 2$, and therefore
each such group has a normal subset $S$ with $\frac{|G|}3 \le |S| \le \frac{2|G|}3$.
Setting $T = G \smallsetminus S^{-1}$, we have $|T| \ge \frac{|G|}3$, and $e\not\in ST$.

If $G$ is non-trivial and we do not assume that $S, T \subseteq G$ are normal subsets, then we may choose $S, T \subseteq G$ of size at least
$\bigl\lfloor \frac{|G|}2\bigr\rfloor$ such that $ST \not\supseteq G \smallsetminus \{ e \}$; indeed, fix $g\in G\smallsetminus\{e\}$, choose $S$ of the specified size, and let $T = G \smallsetminus S^{-1}g$.

Our results about these questions are summarized below.  An affirmative answer to Question~\ref{Strong} implies an affirmative answer to Question~\ref{Weak} (and, of course, the same holds in the special case $S=T$).

\begin{thmA}\label{main}
\begin{enumerate}[\rm(i)]
\item The answers to Questions~\ref{Weak} and~\ref{Strong} are negative if $G$ is allowed to range over all finite simple groups,
or even just over the alternating groups, or just over all projective special linear groups.

\item In the $S=T$ case, the answer to Question~\ref{Strong} is still negative for alternating groups.

\item In the $S=T$ case, the answer to Question~\ref{Weak} is positive for alternating groups.

\item  If $G$ is a group of Lie type of bounded rank, then the answers to Questions~\ref{Weak} and \ref{Strong} are both positive.
\end{enumerate}
\end{thmA}

Our proof of part (iv) depends on a result which may be of independent interest, concerning the number of points in a finite product set
inside a product variety which lie on a subvariety of the product variety.  See Theorem~\ref{Induction} below.

Applications of Theorem A to word maps and to permutation groups are given in Sections 8 and 10.
Our main application concerns derangements (namely fixed-point-free permutations) in finite simple transitive permutation groups.
Let $G$ be a permutation group on a finite set $X$ of size $n$. Denote by $\cD(G) = \cD(G,X)$ the set of derangements in $G$, and let
$\delta(G) = \delta(G,X) = |\cD(G,X)|/|G|$ be the proportion of derangements in $G$. If $G$ is transitive, and $H < G$ is a
point stabilizer, we identify $X$ with the set of left cosets $G/H$ and write $\cD(G,H)$ and $\delta(G,H)$ for $\cD(G,G/H)$
and $\delta(G,G/H)$ respectively. Note that $\cD(G,H) = G \setminus \cup_{g \in G} H^g$.

The study of derangements goes back three centuries to 1708, when Monmort showed that the symmetric group $\SSS_n$ (in its natural action on
$\{ 1, 2,\ldots , n \})$ satisfies $\delta(\SSS_n) \to 1/e$ as $n \to \infty$.
In the 1870s Jordan proved that every finite transitive permutation group of degree $n > 1$ contains a derangement. Since then
derangements have been studied extensively and have proved useful in various areas of mathematics, including group theory, graph theory,
probability, number theory and topology. See the book \cite{BG} for background and new results.

The Classification of Finite Simple Groups has revolutionized the study of derangements, and various powerful results
have been obtained. These include the well-known result of Fein, Kantor and Schacher \cite{FKS}, strengthening Jordan's theorem,
that every finite transitive permutation group of degree $n>1$ has a derangement of prime power order. The question of the
existence of derangements of prime order is discussed extensively in \cite{BG}.

In recent years there has been considerable interest in invariable generation of groups, which has sparked renewed interest in derangements.
Recall that a group $G$ (finite or infinite) is said to be invariably generated by a subset $S \subseteq G$ if, whenever we replace
each $s \in S$ by any conjugate $s^g$ of $s$ (where $g \in G$ depends on $s$), we obtain a generating set for $G$. It is easy
to see that $G$ is invariably generated by $G$ if and only if whenever $G$ acts transitively on some set $X$ with $|X| > 1$ we have
$\cD(G,X) \ne \emptyset$. This in turn is equivalent to $\cup_{g \in G} H^g \subsetneq G$ for every proper subgroup $H < G$.
Thus finite groups are invariably generated by themselves, but some infinite groups are not.

For a finite group $G$ and a positive integer $k$, let $P_I(G,k)$ denote the probability that $k$ randomly chosen elements
of $G$ invariably generate $G$. The study of these probabilities is motivated by computational Galois theory, see e.g. 
\cite{D}, \cite{LP}, \cite{KLSh}, \cite{PPR},  and \cite{EFG}. The latter two papers show that $P_I(\SSS_n,4)$ is bounded away from zero, while
$P_I(\SSS_n,3)$ is not.

It is easy to see (see for instance \cite[2.3]{KLSh}) that $1 - P_I(G,k) \le \sum_H (1-\delta(G,H))^k$, where $H$ ranges over a set
of representatives of the conjugacy classes of the maximal subgroups of $G$. Thus the study of derangements and their proportions 
has applications to invariable generation and related topics.

A lower bound of the form $1/n$ on the proportion of derangements in arbitrary transitive permutation groups $G$ of degree $n$ was provided
in \cite{CC}. This bound is sharp.  It is attained if and only if $G$ is a Frobenius group of degree
$n(n-1)$. If $n \ge 7$ and $G$ is not a Frobenius group of size $n(n-1)$ or $n(n-1)/2$ than a better lower bound of the form $\delta(G) > 2/n$
was subsequently provided in \cite{GW}, with a number-theoretic application.

The case where the transitive permutation group $G$ is simple has been studied thoroughly in the past two decades by Fulman and Guralnick
\cite{FG1, FG2, FG3}, proving a conjecture of Boston and Shalev that $\delta(G) \ge \e$ for some fixed $\e > 0$.
Thus the set of derangements in such a group is a large normal subset, and our results on the square of such subsets may be applied.

Our main result concerning derangements is the following.

\begin{thmB}\label{main2}
Let $G$ be a finite simple transitive permutation group of sufficiently large order. Then every element of $G$ is a product of two derangements.
\end{thmB}

Theorem A in itself does not imply Theorem B, since the answer to Question 1 in the case $S=T$ is positive only for certain families of finite simple groups. However, it does imply Theorem B for alternating groups and for groups of Lie type of bounded rank (note that we always have
$e \in \cD(G)^2$, since $\cD(G) = \cD(G)^{-1}$).
The remaining cases of classical groups of unbounded rank are rather challenging and require additional tools; these include
some results from \cite{FG1, FG2, FG3}, the theory of maximal subgroups of classical groups (see \cite{KL}),
as well as new results in Representation theory of classical groups (see Section 9 below).

Our paper is organized as follows. Sections 2 and 3 are devoted to algebro-geometric results that are needed
in the proof of part (iv) of Theorem A, which is carried out in Section 4. In Section 5 we prove part (i) of Theorem A for
special linear groups. Section 6 is devoted to alternating groups and contains proofs of parts (i), (ii) and (iii) of Theorem A. 
In Section 7 we discuss products of three normal subsets. An application to word maps is presented in Section 8. Section 9 is 
devoted to representation-theoretic results which are
required in the proof of Theorem B and may have some independent interest. Theorem B is then proved in Section 10.
Finally, in Section 11 we show that the conclusion of Theorem B holds for {\it all} simple alternating groups.

\section{The Lang-Weil estimate}
By a \emph{variety} $\vX$ over a field $k$, we mean a separated geometrically irreducible scheme of finite type over $k$.  By the Lang-Weil theorem, if $k=\F_q$, then
\begin{equation}
\label{Lang-Weil}
\bigm||\vX(\F_{q^m})| - q ^{m\dim\vX}\bigm| \le Bq^{m(\dim\vX-1/2)}
\end{equation}
for some constant $B$ depending on $\vX$ but not on $m$.
We will need a number of variants of this statement; the reader who is willing to accept them can skip the remainder of this section.

For any separated scheme of finite type,
the left hand side can be computed using the Lefschetz trace formula \cite[Rapport, Th\'eor\`eme~3.2]{SGA 4.5}:
\begin{equation}
\label{LTF}
|\vX(\F_{q^m})| = \sum_{i=0}^{2\dim \vX} (-1)^i \Tr(\Frob_{q^m}\vert H^i_c(\bar\vX,\Q_\ell)).
\end{equation}

Let $d := \dim \vX$.
We fix an embedding $\iota\colon \Q_\ell\to \C$.  A well-known theorem of Deligne \cite[Th\'eor\`eme~3.3.4]{Weil II} asserts that each eigenvalue of $\Frob_q$ acting on
$H^i_c(\bar\vX,\Q_\ell)$ has absolute value $q^{w/2}$ for some non-negative integer $w\le i$.
In particular, the only $i$ for which $H^i(\bar\vX,\Q_\ell)$ has an eigenvalue
of absolute value $\ge q^d$  is $i=2d$.  If these eigenvalues are $\alpha_1q^d,\ldots,\alpha_kq^d$ (with repetitions allowed), then each $\alpha_i$ has absolute value $1$, and \eqref{Lang-Weil} implies
$$\lim_{m\to \infty} \bigl(\alpha_1^m+\cdots+\alpha_k^m\bigr)  = 1,$$
which implies $k=1$ and $\alpha_1 = 1$.
(In fact, geometric irreducibility implies that
$H^{\dim \vX}(\bar\vX,\Q_\ell)$ is $1$-dimensional and the trace map is an isomorphism.)
Thus, in \eqref{Lang-Weil}, the $q^{m\dim\vX}$ term cancels the contribution of $i=2\dim\vX$ in \eqref{LTF}, and $B$ can be taken to be  the sum of the compactly supported  Betti numbers of $\bar{\vX}$.
Note that $B$ depends only on $\bar\vX$, so this estimate holds uniformly for all Galois twists of $\vX$.

If $\vX$ ranges over the (geometrically irreducible) fibers of a morphism $\pi\colon \schX\to \schS$ between schemes of finite type over $\Z$, then $B$ is bounded uniformly among all such fibers.  This is a consequence of the proper base change theorem \cite[Arcata, IV, Th\'eor\`eme~5.4]{SGA 4.5} (which identifies the $i$th \'etale cohomology group with compact support of a geometric fiber with the corresponding fiber of $R^i\pi_!\Q_\ell$), Nagata's compactification theorem (\cite[Arcata, IV, (5.3)]{SGA 4.5}), and the constructibility \cite[Finitude, Th\'eor\`eme~1.1]{SGA 4.5} of the sheaves $R^i\pi'_*j_!\Q_\ell = R^i \pi_! \Q_\ell$ for a compactification
$$\xymatrix{\schX\ar[rr]^j\ar[dr]_\pi&& \schX'\ar[dl]^{\pi'} \\ &\schS&}$$
As a consequence, there exists $B$ such that for all $q$, all points $s\in \schS$ with finite residue field $k(s) = \F_q$, all varieties $\vX$ of the form
$\vX = \schX\times_{\schS} k(s)$, and all positive integers $m$,
\begin{equation}
\label{Uniform}
\bigm| |\vX(\F_{q^m})| - q^{m\dim \vX} \bigm| \le Bq^{m(\dim \vX-1/2)}.
\end{equation}

Given any integer $r$, there are only finitely many root systems of rank $r$, and for each root system $\Phi$, there exists a Chevalley group scheme $\schG$ over $\Z$, that is, a smooth group scheme over $\Spec \Z$, whose fiber over each field $F$ is the connected, simply connected, split  semisimple algebraic groups over $F$ with root system $\Phi$. Thus, we can uniformly bound the sum of compactly supported Betti numbers for all semisimple
groups of rank $r$ over all algebraically closed fields.

Suppose $\vX$ is a variety over $\F_q$ and $F\colon \vX\to \vX$ is an endomorphism of varieties over $\F_q$ such that $F^2 = \Frob_q$.  Then for $f\in \N$ sufficiently large,
\begin{equation}
\label{Suzuki-Ree estimate}
\bigm| |\vX(\bar\F_q)^{F^{2f+1}}| - q^{(f+1/2){\dim \vX}}\bigm| < B q^{(f+1/2)(\dim \vX-1/2)}.
\end{equation}
This follows from Fujiwara's extension of the Lefschetz trace formula \cite{Varshavsky}.
This allows us to treat Suzuki and Ree groups on the same footing as the other finite simple groups of Lie type.

If $\vZ$ is a variety and $\vW$ is a proper closed subvariety, then $\dim\vW \le \dim \vZ-1$, so
$$|\vW(\F_q)| \le Bq^{\dim \vZ-1},$$
where $B$ is the sum of Betti numbers of $\vW$.  As $\vZ$ and $\vW$ range over the fibers of a morphism of finite type over $\Z$, the constant $B$ can be bounded uniformly as before.

If $\pi\colon \vZ\to \vS$ is a dominant morphism of $\F_q$-varieties whose generic fiber is geometrically irreducible,
then there exists a proper closed subscheme $\vW$ of $\vZ$ such that the restriction of $\pi$ to the complement of $\vW$ is geometrically irreducible \cite[Corollaire~9.7.9]{EGA IV 3}.
If $B$ denotes the maximum sum of Betti numbers of any fiber of $\pi|_{\vZ\smallsetminus \vW}$, $B'$ denotes the sum of Betti numbers of $\vW$, and $B''$ denotes the sum of Betti numbers of $\vS$, then for all $S\subset \vS(\F_{q^m})$,
\begin{align*}
\bigm||\pi^{-1}(S)| - |S| q^{m(\dim \vZ-\dim \vS)}\bigm|
		&\le B |S| q^{m(\dim \vZ-\dim \vS)-1/2} + B' q^{m(\dim\vZ-1)}\\
		&\le (B+BB''+B')q^{m(\dim \vZ)-1/2}.
\end{align*}

\bigskip

\section{Morphisms which respect products}
If $\pi\colon \vZ\to \vS$ is a morphism of varieties over $\F_q$, we denote by $\pi_m$ the function $\vZ(\F_{q^m})\to \vS(\F_{q^m})$ that it determines.
Let $S_m \subset \vS(\F_{q^m})$.  We have seen that if $\pi$ has geometrically irreducible generic fiber, then
$$|\pi_m^{-1}(S_m)| = q^{m(\dim \vZ-\dim \vS)} |S_m| + O(q^{m((\dim \vZ)-1/2)}).$$
Applying Lang-Weil for $\vS$ and $\vZ$, this estimate can be expressed equivalently as
\begin{equation}
\label{uniformity}\frac{|\pi_m^{-1}(S_m)|}{|\vZ(\F_{q^m})|} = \frac{|S_m|}{|\vS(\F_{q^m})|} + O(q^{-m/2}).
\end{equation}

All we actually need from the estimate \eqref{uniformity} is the weaker version
\begin{equation}
\label{weak-uniformity}
\frac{|\pi_m^{-1}(S_m)|}{|\vZ(\F_{q^m})|} = \frac{|S_m|}{|\vS(\F_{q^m})|} + o(1),
\end{equation}
or, equivalently,
\begin{equation}
\label{weak-uniformity 2}
\frac{|\pi_m^{-1}(S_m)|}{q^{m\dim\vZ} } = \frac{|S_m|}{q^{m\dim\vS}} + o(1).
\end{equation}

Conversely, if \eqref{weak-uniformity} holds for all $S_m$, then $\pi$ is generically geometrically irreducible \cite[Proposition~2.1]{LST2}.

Now, let $\vX$, $\vY$, and $\vZ$ denote varieties over $\F_q$ and $\pi\colon \vZ\to \vX\times \vY$
a morphism of $\F_q$-varieties.  By Lang-Weil for $\vX$, $\vY$, or $\vZ$, we mean the $o(1)$ form of the error term rather than the $O(q^{-m/2})$ form.

\begin{defn}
We say $\pi$ \emph{respects products} if, as $m\to \infty$, for all $X_m\subset \vX(\F_{q^m})$ and $Y_m\subset \vY(\F_{q^m})$, we have
\begin{equation}
\label{alpha-beta}
\frac{|\pi_m^{-1}(X_m\times Y_m)|}{q^{m\dim \vZ}} = \frac{|X_m\times Y_m|}{q^{m\dim \vX\times \vY}}
+ o(1).
\end{equation}
\end{defn}

In particular, $\pi$ respects products if it has geometrically irreducible generic fiber.
The converse is not true, but we have the following weaker statement.
Let $\pi_{\vX}$ and $\pi_{\vY}$ denote the compositions of $\pi$ with the projection morphisms from $\vX\times\vY$ to $\vX$ and $\vY$ respectively.

\begin{lem}
\label{each}
If $\pi$ respects products, then $\pi_{\vX}$ and $\pi_{\vY}$ are generically geometrically irreducible.
\end{lem}

\begin{proof}
By specializing to the case $X_m = \vX(\F_{q^m})$, \eqref{alpha-beta} becomes \eqref{weak-uniformity 2}, which implies
that $\pi_{\vY}$ is generically geometrically irreducible.
By symmetry, the same is true for $\pi_{\vX}$ as well.
\end{proof}

Note that just because $\pi_{\vX}$ and $\pi_{\vY}$ are generically geometrically irreducible, it is not necessarily the case that $\pi$ respects products. For example, if $\vX=\Spec \F_q[x]$, $\vY = \Spec \F_q[y]$, and
$$\vZ = \Spec \F_q[x,y,z]/(z^2-xy),$$
$\pi$ corresponds to the obvious homomorphism
$$\F_q[x]\otimes\F_q[y] \to \F_q[x,y,z]/(z^2-xy),$$
and $X_m = Y_m$ is the set of squares of elements of $\F_{q^m}^\times$, then the left hand side of \eqref{alpha-beta} approaches $1/2$, while the right hand side is $1/4+o(1)$.

However,  in many cases, the converse of Lemma~\ref{each} does hold.  Suppose that $\pi_{\vY}$ is flat
with geometrically irreducible generic fiber.  As flatness is preserved by base change and the composition of flat morphisms is flat,
 $\vZ\times_{\vY}\vZ$ is flat over
$\vY$, and this remains true after base change from $\F_q$ to $\bar\F_q$.  By \cite[2.4.6]{EGA IV 2}, therefore, every geometric component of  $\vZ\times_{\vY}\vZ$ dominates $\vY\times_{\Spec\F_q}\Spec\bar\F_q$.
However, the generic fiber of $\vZ\times_{\vY}\vZ$ is geometrically irreducible \cite[Corollaire~4.5.8]{EGA IV 2}, so there is only one geometric component, and $\vZ\times_{\vY}\vZ$ is therefore a variety.

\begin{thm}
\label{Induction}
Assume $\vZ$ is flat over $\vY$.
Let
$\psi\colon \vZ\times_{\vY}\vZ\to \vX\times \vX$ denote the morphism of varieties given by
$\psi(z_1,z_2) = (\pi_{\vX}(z_1),\pi_{\vX}(z_2))$.
If  $\psi$ respects products and $\pi_{\vX}$ and $\pi_{\vY}$
have geometrically irreducible generic fiber,
then $\pi$ respects products.
\end{thm}

\begin{proof}
Let $X_m\subset \vX(\F_{q^m})$ and $Y_m\subset \vY(\F_{q^m})$ be subsets, $Y_m^c$ the complement of $Y_m$ in $\vY(\F_{q^m})$,
and $Z_m := \pi_{\vX\,m}^{-1}(X_m) = \pi^{-1}(X_m\times \vY(\F_{  q^m}))$.  As $\pi_{\vX}$ has geometrically irreducible generic fiber,
\begin{equation}
\label{Zm bound}
\frac{|Z_m|}{q^{m\,\dim\vZ}} = \frac{|X_m|}{q^{m\,\dim \vX}} + o(1).
\end{equation}
Since $X_m\subset \vX(\F_{q^m})$,
\begin{equation}
\label{Z squared}
\frac{|Z_m|^2}{q^{2m\,\dim\vZ}} = \frac{|X_m|^2}{q^{2m\,\dim\vX}} + o(1).
\end{equation}

Let
%
%
\begin{align*}
\Delta_m& := |\pi_m^{-1}(X_m\times Y_m)|\,|Y_m^c| - |\pi_m^{-1}(X_m\times Y_m^c)|\,|Y_m|\\
&=|\pi_m^{-1}(X_m\times Y_m)|\,|\vY(\F_{q^m})| - |\pi_m^{-1}(X_m\times \vY(\F_{q^m})|\,|Y_m|\\
&=|\pi_m^{-1}(X_m\times Y_m)|\,|\vY(\F_{q^m})| - |Z_m|\,|Y_m|.
\end{align*}
We aim to prove an $o(1)$ bound for
\begin{equation}
\label{XY error}
\begin{split}
\frac{|\pi_m^{-1}(X_m\times Y_m)|}{|\vZ(\F_{q^m})|} &- \frac{|X_m\times Y_m|}{|(\vX\times \vY)(\F_{q^m})|} \\
& \hskip -30pt= \frac{|\pi_m^{-1}(X_m\times Y_m)|\,|(\vX\times \vY)(\F_{q^m})| - |X_m\times Y_m|\,|\vZ(\F_{q^m})|}{| (\vX\times \vY\times \vZ)(\F_{q^m})|} \\
&\hskip -30pt= \frac{\Delta_m|\vX(\F_{q^m})| + |Y_m|(|\vX(\F_{q^m})|\,|Z_m| - |X_m|\,|\vZ(\F_{q^m})|)} {|(\vX\times \vY\times \vZ)(\F_{q^m})|} \\
& \hskip -30pt=\frac{\Delta_m}{|(\vY\times \vZ)(\F_{q^m})|} + \frac{|Y_m|}{|\vY(\F_{q^m})|}\Bigl(\frac{|Z_m|}{\vZ(\F_{q^m})|} - \frac{|X_m|}{|\vX(\F_{q^m})|}\Bigr).
\end{split}
\end{equation}
By \eqref{Zm bound} and Lang-Weil for $\vY$ and $\vZ$, this expression can be written
$$\frac{\Delta_m}{q^{m(\dim \vY+\dim \vZ)}} + o(1).$$
It suffices, therefore, to prove that
\begin{equation}
\label{Delta bound}
\Delta_m = o\bigl(q^{m(\dim \vY+\dim \vZ)}\bigr).
\end{equation}

We have
\begin{align*}
\psi_m^{-1}&(X_m\times X_m) \\
                             &= \{(z_1,z_2,y)\in Z_m\times Z_m\times \vY(\F_{q^m})\mid
                             \pi_{\vY}(z_1) = \pi_{\vY}(z_2) = y\},
\end{align*}
so the cardinality of the left hand side is
\begin{equation}
\label{long-split}
\begin{split}
\sum_{y\in\vY(\F_{q^m})} &|\pi_m^{-1}(X_m\times \{y\})|^2 \\
&= \sum_{y\in Y_m} |\pi_m^{-1}(X_m\times \{y\})|^2 + \sum_{y\in Y_m^c}  |\pi_m^{-1}(X_m\times \{y\})|^2 \\
&\ge \frac{\Bigl(\sum_{y\in Y_m} |\pi_m^{-1}(X_m\times \{y\})| \Bigr)^2}{|Y_m|} + \frac{\Bigl(\sum_{y\in Y_m^c} |\pi_m^{-1}(X_m\times \{y\})| \Bigr)^2}{|Y_m^c|} \\
&= \frac{|\pi_m^{-1}(X_m\times Y_m)|^2}{|Y_m|}
  + \frac {|\pi_m^{-1}(X_m\times Y_m^c)|^2}{|Y_m^c|} \\
&= \frac{(|\pi_m^{-1}(X_m\times Y_m)|+|\pi_m^{-1}(X_m\times Y_m^c)|)^2+\frac{\Delta_m^2}{|Y_m|\,|Y_m^c|}}{|Y_m|+|Y_m^c|} \\
&= \frac{|Z_m|^2+\frac{\Delta_m^2}{|Y_m|\,|Y_m^c|}}{|\vY(\F_{q^m})|}  \\
&= \frac{q^{2m(\dim \vZ-\dim \vX)}|X_m|^2+\frac{\Delta_m^2}{|Y_m|\,|Y_m^c|}}{q^{m\dim \vY}} + o(q^{m(2\dim \vZ-\dim\vY)}),
\end{split}
\end{equation}
by Cauchy-Schwartz, \eqref{Z squared}, and Lang-Weil for $\vY$.   As $\psi$ respects products,
\begin{equation}
\label{XX bound}
\frac{|\psi_m^{-1}(X_m\times X_m)|}{q^{m(2\dim\vZ-\dim\vY)}} = \frac{|X_m|^2}{q^{2m\dim\vX}} + o(1).
\end{equation}
Thus \eqref{long-split} implies
$$\frac{\Delta_m^2}{|Y_m|\,|Y_m^c|} = o(q^{2m\,{\dim \vZ}}),$$
which, by Lang-Weil for $\vY$, gives \eqref{Delta bound}.
\end{proof}

Note that the implicit bound of \eqref{XY error} can be expressed in terms of the implicit bounds in the Lang-Weil estimates of $\vX$, $\vY$, and $\vZ$ and those in
\eqref{Zm bound} and \eqref{XX bound}.  The uniformity \eqref{Uniform} in Lang-Weil estimates for families over a scheme of finite type over $\Z$ implies the following.
Let
$$\xymatrix{\schZ\ar[rr]^\pi\ar[dr]&&\schX\times\schY\ar[dl] \\ &\schS&}$$
be a morphism of schemes of finite type over $\Z$ for which the corresponding morphism
$\pi_{\schY}\colon \schZ\to \schY$ is flat.
For each point $s$ with finite residue field $k(s) = \F_q$, we consider the specialization $\vZ\to \vX\times \vY$ of $\pi$.  Assuming that for some family of such morphisms we have a uniform $o(1)$ error bound for \eqref{XX bound}, then we have a uniform $o(1)$ error bound in \eqref{alpha-beta} for all members of the family of morphisms.  As Betti numbers depend only on cohomology after base change to $\bar\F_q$, we also have a uniform $o(1)$ error bound in  \eqref{alpha-beta} for morphisms obtained from members of the family by Galois twisting.

The estimate \eqref{Suzuki-Ree estimate} gives a uniform $o(1)$ bound of type \eqref{alpha-beta} in the setting of Suzuki and Ree groups.
Explicitly, let  $\pi\colon \vZ\to \vX\times \vY$ be a morphism of $\F_q$-varieties,
and let $\psi\colon \vZ\times_{\vY}\vZ\to \vX\times \vX$ be defined as before.
Suppose $F_X$, $F_Y$ and $F_Z$ are endomorphisms
of $\vX$, $\vY$, and $\vZ$  as $\F_q$-varieties such that $F_X^2$, $F_Y^2$, and $F_Z^2$ are the $q$-Frobenius morphisms on $\vX$, $\vY$, and $\vZ$ respectively.
Suppose further that the diagram
$$\xymatrix{\vZ\ar[r]\ar[d]_{F_Z}&\vX\times \vY\ar[d]^{F_X\times F_Y}\\
\vZ\ar[r]&\vX\times \vY}$$
commutes.  For $f$ a non-negative integer, let
$$\pi_f\colon \vZ(\bar\F_q)^{F^{2f+1}} \to \vX(\bar\F_q)^{F^{2f+1}}\times \vY(\bar\F_q)^{F^{2f+1}},$$
denote the obvious restriction of $\pi$, and likewise for
$$\psi_f\colon (\vZ\times_{\vY}\vZ)(\bar\F_q)^{F^{2f+1}} \to \vX(\bar\F_q)^{F^{2f+1}}\times \vX(\bar\F_q)^{F^{2f+1}}.$$
Let $X$ and $Y$ denote subsets of $\vX(\bar\F_q)^{F^{2f+1}}$ and $\vY(\bar\F_q)^{F^{2f+1}}$.  Then
$$\frac{|\psi_f^{-1}(X\times X)|}{q^{(f+1/2)\dim \vZ\times_{\vY}\vZ}} = \frac{|X\times X|}{q^{(f+1/2)\dim \vX\times \vX}} + o(1)$$
implies
\begin{equation}
\label{Suzuki-Ree product}
\frac{|\pi_f^{-1}(X\times Y)|}{q^{(f+1/2)\dim \vZ}} = \frac{|X\times Y|}{q^{(f+1/2)\dim \vX\times \vY}} + o(1).
\end{equation}

In applying Theorem~\ref{Induction} and its variants, we are always in the situation that $\pi_{\vY}$ is a projection map from a product variety to one of its factors.  It is therefore flat (since every morphism to the spectrum of a field is flat, and flatness respects base change.)

\bigskip

\section{Equidistribution for bounded rank groups of Lie type}

In this section, we show that Questions~\ref{Weak} and \ref{Strong} have an affirmative answer if one restricts to finite simple groups of bounded rank.
Throughout the section, $\vG$ denotes a simply connected simple algebraic group over $\F_q$.

\begin{thm}
If   $c\in \vG(\F_{q^m})$ is not central then for every integer $n\ge 2\dim \vG$, the morphism
$$\phi\colon\vG^{2n}\to \vG$$
given by
$$\phi(x_1,y_1,\ldots,x_n,y_n) = x_1cx_1^{-1} y_1c^{-1}y_1^{-1} \cdots x_ncx_n^{-1}y_nc^{-1} y_n^{-1}$$
has geometrically irreducible generic fiber.
\end{thm}

\begin{proof}
It suffices to prove that, fixing $n$,
\begin{equation}
\label{fiber-size}
|\phi_m^{-1}(g)| = (1+o(1)) q^{m(2n-1)\dim\vG}
\end{equation}
for all $g\in \vG(\F_{q^m})$ as $m\to\infty$.
A well-known theorem of Frobenius,
asserts that if $C_1,\ldots,C_k$ are conjugacy classes in a finite group $G$ and $g\in G$, then the number
of elements in the  set
$$\{(g_1,\ldots,g_k)\in C_1\times \cdots\times C_k\mid g_1\cdots g_k = g\}$$
is
\begin{equation}
\label{Frobenius}
\frac{|C_1|\cdots |C_k|}{|G|}\sum_\chi \frac{\chi(C_1)\cdots\chi(C_k)\bar\chi(g)}{\chi(1)^{k-1}},
\end{equation}
where the sum is taken over irreducible characters $\chi$ of $G$.  Thus,
if $C$ is a conjugacy class in $\vG(\F_{q^m})$,  the number of representations
$$|\{(x_1,y_1,\ldots,x_n,y_n)\in C^{2n}\mid x_1 y_1^{-1}\cdots x_n y_n^{-1} = g\}|$$
is given by
$$\frac{|C|^{2n}}{|\vG(\F_{q^m})|}\sum_\chi \frac{|\chi(C)|^{2n}\bar\chi(g)}{\chi(1)^{2n-1}},$$
Therefore,
$$|\phi_m^{-1}(g)| = |\vG(\F_{q^m})|^{2n-1}\biggl(1+\sum_{\chi\neq 1}\frac{|\chi(C)|^{2n}\bar\chi(g)}{\chi(1)^{2n-1}}\biggr).$$

By a theorem of David Gluck \cite{Gluck}, for every non-central element $x\in \vG(\F_{q^m})$
and every non-trivial irreducible character $\chi$, we have
$$\frac{|\chi(x)|}{\chi(1)} \le aq^{-m/2},$$
where $a$ is an absolute constant.  As
$$|\chi(1)\bar\chi(g)|\le \chi(1)^2 \le |\vG(\F_{q^m})| = (1+o(1))q^{m\dim \vG},$$
we have
$$\frac{|\chi(C)|^{2n}\bar\chi(g)}{\chi(1)^{2n-1}}= (1+o(1))a^{2n}q^{m\dim \vG-mn}.$$
The total number of irreducible characters is $o(|\vG(\F_{q^m})|) = o(q^{m\dim\vG})$,
so $n\ge 2\dim\vG$ implies \eqref{fiber-size}.

\end{proof}

\begin{cor}
\label{Key-cor}
With notations as above,
If $\theta^n\colon  \vG^{2n}\times \vG\to \vG\times \vG$ is defined by
$$\theta^n(x_1,y_1,\ldots,x_n,y_n,g) = (\phi(x_1,y_1,\ldots,x_n,y_n)g,g),$$
then $\theta^n$ is generically geometrically irreducible.
\end{cor}

\begin{proof}
We have
$$|(\theta^n_m)^{-1}(g_1,g_2)| = |\phi_m^{-1}(g_1 g_2^{-1})|.$$
By \eqref{fiber-size}, the right hand side is always
$$(1+o(1))q^{m(2n-1)\dim \vG} =  (1+o(1))q^{m(\dim \vZ - \dim \vX\times \vY)}.$$
The corollary follows from \eqref{uniformity}.
\end{proof}

\begin{thm}
\label{Products}
Let $\vX=\vY = \vG$ and $\vZ =  \vG\times\vG$.  Let $\pi\colon \vZ\to \vX\times \vY$
be defined by $\pi(x,g) = (xcx^{-1}g,g)$.  Then $\pi$ respects products.
\end{thm}

\begin{proof}

The isomorphism $\omega\colon  \vZ\times_{\vG}\vZ\to  \vG^2\times\vG$ defined by
$$\omega((x_1,g),(x_2,g)) = (x_1,x_2,x_2cx_2^{-1}g)$$
makes the diagram
$$\xymatrix{\vZ\times_{\vG}\vZ \ar[rr]^\omega\ar[dr]^\psi&&\vG^2\times\vG\ar[dl]_{\theta^1} \\
&\vG\times \vG}$$
commute.  By Theorem~\ref{Induction}, if $\pi$ does not respect products, then $\theta^1$ does not respect them either.

For $n\ge 1$, we define
$$\xi^n\colon (\vG^{2n}\times\vG)\times_{\vG} (\vG^{2n}\times\vG)\to (\vG^{4n}\times\vG)$$
by
\begin{align*}
\xi^n((x_1,y_1,&\ldots,x_n,y_n,g),(x_{n+1},y_{n+1}\ldots,x_{2n},y_{2n},g)) \\
&= (x_1,y_1,\ldots,x_{2n},y_{2n}, \phi(x_{n+1},y_{n+1},\ldots,y_{2n})^{-1}g)
\end{align*}
and
$$\eta^n\colon (\vG^{2n}\times\vG)\times_{\vG} (\vG^{2n}\times\vG)\to \vG\times \vG$$
by
\begin{align*}
\eta^n((x_1,y_1,\ldots,x_n,y_n,g),&(x_{n+1},y_{n+1}\ldots,x_{2n},y_{2n},g)) \\
&= (\phi(x_1,y_1,\ldots, y_n)g,\phi(x_{n+1},y_{n+1},\ldots, y_{2n})g),
\end{align*}
the diagram
$$\xymatrix{(\vG^{2n}\times\vG)\times_{\vG} (\vG^{2n}\times\vG) \ar[rr]^(.6){\xi^n}\ar[dr]^{\eta^n}&&\vG^{4n}\times\vG\ar[dl]_{\theta^{2n}} \\
&\vG\times \vG}$$
commutes.  Applying Theorem~\ref{Induction} in the case $\vX = \vY = \vG$, $\vZ := \vG^{2n}\times \vG$, and $\pi=\theta^n$,
so $\pi_{\vX}$ and $\pi_{\vY}$ are both given by
composing $\theta^n$ with projection to the first coordinate, and therefore $\psi$ is $\eta^n$,
we deduce that if $\theta^n$ does not respect products, $\theta^{2n}$ does not respect them either.
Thus if $\theta^1$ does not respect products, by induction $\theta^{2^i}$ does not respect them
either.

By Corollary~\ref{Key-cor}, for $i$ sufficiently large, $\theta^{2^i}$ is generically geometrically irreducible and therefore does respect products.  The theorem follows.
\end{proof}

\begin{thm}
\label{Main-General}
Given a  simply connected simple algebraic group $\vG$ over $\F_q$ and $\epsilon > 0$, there exists $M$ such that if $m > M$, $S$ and $T$ are subsets of $\vG(\F_{q^m})$ with at least $\epsilon q^{m\dim \vG}$ elements, and $C$ is a non-central conjugacy class of $\vG(\F_{q^m})$, then
the number of pairs $(s,t)\in S\times T$ with $st^{-1}\in C$ satisfies
\begin{equation}
\label{total-reps-SR}
1-\epsilon < \frac{|\{(s,t)\in S\times T\mid st^{-1}\in C\}|\, |\vG(\F_{q^m})|}
{|S|\, |T|\, |C|} < 1+\epsilon.
\end{equation}
\end{thm}

\begin{proof}
If $c\in C$, the number of such pairs is $|\vG(\F_{q^m})|^{-1}|C|$ times the number of solutions of $st^{-1} = x c x^{-1}$, $s\in S$, $t\in T$, $x\in \vG(\F_{q^m})$.
Theorem~\ref{Products} implies the the number of such solutions is asymptotic to
$|S|\,|T|$
as $m\to\infty$, which gives the theorem.
\end{proof}

Note that $T^{-1}$ is normal, and $|T| = |T^{-1}|$, so the theorem gives equivalently
$$1-\epsilon < \frac{|\{(s,t)\in S\times T\mid st\in C\}|\, |\vG(\F_{q^m})|}
{|S|\, |T|\, |C|} < 1+\epsilon.
$$
Note also that as the error $o(1)$ in Theorem~\ref{Induction} is uniform over all finite simple groups of bounded rank and all choices of $c$,
the same is true for Theorem~\ref{Main-General}.

By the comments following the proof of Theorem~\ref{Products}, we have the following ``Suzuki-Ree'' version of Theorem~\ref{Main-General}:
\begin{thm}
\label{Main-Suzuki-Ree}
Given a  simply connected simple algebraic group $\vG$ over $\F_q$ and an endomorphism $F$ of $\vG$ such that
$F^2 = \Frob_q$, for all $\epsilon > 0$, there exists $M$ such that if $f > M$, $S$ and $T$ are subsets of $G := \vG(\bar\F_q)^{F^{2f+1}}$ with at least $\epsilon q^{(f+1/2)\dim \vG}$ elements, and $C$ is a non-central conjugacy class of $G$, then
the number of pairs $(s,t)\in S\times T$ with $st^{-1}\in C$ satisfies
\begin{equation}
\label{total-reps}
1-\epsilon < \frac{|\{(s,t)\in S\times T\mid st^{-1}\in C\}|\, |G|}
{|S|\, |T|\, |C|} < 1+\epsilon.
\end{equation}
\end{thm}

\begin{thm}
Let $r$ and $\epsilon > 0$ be fixed.  If $G$ is the universal central extension of a finite simple group of Lie type of rank $r$
and $S$ and $T$ are normal subsets with at least $\epsilon |G|$ elements each, the number of representations of any non-central element $c$ as $st$, $s\in S$ and $t\in T$, is
$$(1+o(1))\frac{|S|\, |T|}{|G|}.$$
\end{thm}

\begin{proof}
With finitely many exceptions, the universal central extension $G$ of a finite simple group of Lie type is either of the form $\vG(\F_{q^m})$, where $\vG$ is a simply connected simple algebraic group over $\F_q$,
or is a Ree or Suzuki group.  In the former case, the theorem is just Theorem~\ref{Main-General}; in the latter case, it is Theorem~\ref{Main-Suzuki-Ree}.
\end{proof}

\begin{thm}\label{Q12bounded}
Questions~\ref{Weak} and \ref{Strong} have an affirmative answer for finite simple groups $G$ of Lie type of bounded rank.

\end{thm}

\begin{proof}
Let $\tilde G$ denote the universal central extension of $G$, so we may assume either $\tilde G = \vG(\F_q)$ for some
simply connected simple algebraic group of bounded rank, or $\tilde G = \vG(\bar\F_q)^{F^{2f+1}}$.
Let $\pi\colon \tilde G\to G$ be the quotient map by the center of $\tilde G$.
Let $z$ denote the order of $\ker\pi$.
If $S$ and $T$ are normal subsets of $G$, $\tilde S=\pi^{-1}(S)$ and $\tilde T=\pi^{-1}(T)$
are normal subsets of $\tilde G$ of cardinality $z|S|$ and $z|T|$ respectively.
For any $c\in G$, the total number of representations of $c$ as $st$, $s\in S$ and $t\in T$
is $z^{-2}$ times the sum over the elements $\tilde c\in \pi^{-1}(c)$ of the number of representations of $\tilde c$
as $\tilde s\tilde t$ with $\tilde s\in \tilde S$, $\tilde t\in \tilde T$.  For each of these $z$ elements, the number of such representations is
$$(1+o(1))\frac{|\tilde S|\, |\tilde T|}{|\tilde G|} = (1+o(1))z \frac{|S|\, |T|}{|G|},$$
which gives the theorem.
\end{proof}

\bigskip

\section{Behavior of $\PSL_n(q)$ for fixed $q$}

In this section we prove that for $q$ fixed and $n\to\infty$, the answer to Question~\ref{Weak} (and therefore also Question~\ref{Strong}) is negative
for the set of groups $\{\PSL_n(q)\mid n\ge 2\}$.

For $0\le m\le n$, let $\SL_n(\F_q)_{\ge m}$ denote the set of elements $g\in \SL_n(\F_q)$ such that the dimension of the space $(\F_q^n)^{\langle g\rangle}$
of $g$-invariants is at least $m$, and let $\SL_n(\F_q)_m$ denote the set for which the dimension of invariants is exactly $m$.  Let $G_{k,m}$ denote the Grassmannian of $m$-dimensional $\F_q$-subspaces of a $k$-dimensional $\F_q$-vector space $W$.  Its cardinality is the number of ordered linearly independent $m$-tuples in $W$ divided by the number of ordered
bases for a given $m$-dimensional subspace $V$, i.e.,
\begin{equation}
\label{Grass estimate}
\frac{(q^k-1)(q^k-q)\cdots(q^k-q^{m-1})}{(q^m-1)(q^m - q)\cdots (q^m-q^{m-1})} < \frac{q^{m(k-m)}}{(1-q^{-1})\cdots (1-q^{-n})} < 4q^{m(k-m)}
\end{equation}
since
$$\prod_{i=1}^\infty \frac 1{1-q^{-i}} \le \prod_{i=1}^\infty \frac 1{1-2^{-i}} < 4.$$
On the other hand, there is an obvious lower bound, $|G_{k,m}| \ge q^{m(k-m)}$.

\begin{lem}
For $1\le m\le n-1$, the cardinality of $\SL_n(\F_q)_{\ge m}$ is less than
$$16 q^{-m^2} |\SL_n(\F_q)|.$$
\end{lem}

\begin{proof}
As $\SL_n(\F_q)$ acts transitively on linearly independent $m$-tuples in $\F_q^n$, the index of the stabilizer
of an ordered linearly independent $m$-tuple is
$$(q^n-1)(q^n-q)\cdots (q^n-q^{m-1}) > \frac{q^{nm}}4,$$
so the number of elements of $\SL_n(\F_q)$ in the pointwise stabilizer $\Stab(V)$ of a given $m$-dimensional subspace $V$ satisfies
\begin{equation}
\label{Stab}
\frac{|\SL_n(\F_q)|}{q^{mn}}\le
 |\Stab(V)| < \frac{4|\SL_n(\F_q)|}{q^{mn}}.
\end{equation}
The lemma follows by combining the upper bound with \eqref{Grass estimate}.
\end{proof}

Note that this lemma does not cover the case $m=n$, but the bound $16q^{1-m^2}|\SL_n(\F_q)|$
works also for $m=n$ since it is greater than $4>1$ in this case.

\begin{lem}
The number of elements in $\SL_n(\F_q)_m$ is at least
\begin{equation}
\label{Stab lower bound}
(1-128q^{-m})q^{-m^2} |\SL_n(\F_q))|.
\end{equation}
\end{lem}

\begin{proof}
Let $\Stab(V)$ denotes the pointwise stabilizer in $\SL_n(\F_q)$ of $V\in G_{n,m}$.
Then,
\begin{align*}
\sum_{V\in G_{n,m}} |\Stab(V)| &= \sum_{k=m}^{n} |\SL_n(\F_q)_k|\,|G_{k,m}| \\
&= |\SL_n(\F_q)_m| + \sum_{k=m+1}^{n} |\SL_n(\F_q)_k|\,|G_{k,m}| \\
&\le |\SL_n(\F_q)_m| + 4\sum_{k=m+1}^{n} |\SL_n(\F_q)_k| q^{m(k-m)}\\
&\le |\SL_n(\F_q)_m| + 64 |\SL_n(\F_q)|\sum_{k=m+1}^{n} q^{1-k^2} q^{m(k-m)} \\
&=  |\SL_n(\F_q)_m| + 64q^{1-m^2} |\SL_n(\F_q)|\sum_{k=m+1}^{n} q^{k(m-k)} \\
&\le  |\SL_n(\F_q)_m| + 128q^{-m^2} q^{-m}|\SL_n(\F_q)|.
\end{align*}
By the lower bound in \eqref{Stab} and the trivial lower bound for the cardinality of a Grassmannian,
$$q^{-m^2}|\SL_n(\F_q)| \le\sum_{V\in G_{n,m}} |\Stab(V)|.$$
Combining these inequalities, we get \eqref{Stab lower bound}.
\end{proof}

We can now answer Question~\ref{Weak} for fixed $q$.

\begin{thm}
\label{Fixed-q}
If $q$ is fixed, there exist normal subsets $S_n,T_n\subset \SL_n(\F_q)$ such that
$S_n T_n$ does not contain  any transvection, and
\begin{equation}
\label{liminf}
\liminf_n \frac{| S_n|}{|\SL_n(\F_q)|},\,\liminf_n \frac{| T_n|}{|\SL_n(\F_q)|} > 0.
\end{equation}
\end{thm}

\begin{proof}
For small $n$, we can take $S_n=T_n=\{e\}$, so without loss of generality, we may assume $n\ge 10$.
Let $S_n = \SL_n(\F_q)_8$ and $T_n = \SL_n(\F_q)_{10}$.  By \eqref{Stab lower bound},
$$\liminf_n \frac{| S_n|}{|\SL_n(\F_q)|},\,\liminf_n \frac{| T_n|}{|\SL_n(\F_q)|} > 0.$$

Let $\sigma\in S_n$ and $\tau\in T_n$.   If $\rho:=\sigma\tau$ were a transvection, then
it would fix a codimension $1$ subspace $V'\subset \F_q^n$ pointwise, while $\tau$ fixes a $10$-dimensional subspace $V\subset \F_q^n$ pointwise.  This implies that $\sigma$ fixes $V\cap V'$,
which is of dimension $\ge 9$ pointwise, contrary to the definition of $S_n$.
\end{proof}

\begin{cor}
For $n$ relatively prime to $q-1$, for each fixed prime power $q$,
Question~\ref{Weak} has a negative answer for the set of groups $\{\PSL_n(q)\mid n\ge 2\}$.
\end{cor}

\begin{proof}
For $n$ relatively prime to $q-1$, we have an isomorphism $\SL_n(\F_q)\to \PSL_n(q)$,
so the corollary follows.
\end{proof}

\bigskip

\section{Alternating groups}

For alternating groups, we can prove an even stronger negative result.

\begin{thm}\label{alt}
If $0\le s,t\le 1$ then there exists an infinite sequence of pairs of normal subsets $S_n,T_n\subset \AAA_n$, $n\ge 3$, such that
\begin{equation}
\label{limits}
\lim_{n\to \infty} \frac{|S_n|}{|\AAA_n|} = s,\ \lim_{n\to \infty} \frac{|T_n|}{|\AAA_n|} = t,
\end{equation}
and $S_nT_n$ contains no $3$-cycle if and only if $s+t\le 1$.  In particular, Question~\ref{Weak} has a negative answer for alternating groups.
\end{thm}

We begin with two lemmas.  For $\sigma\in \SSS_n$, let $p(\sigma)$ denote the total number of cycles of $\sigma$, i.e., the number of orbits of $\langle \sigma\rangle$ on $\{1,2,\ldots,n\}$.

\begin{lem}
If $\sigma,\tau\in \AAA_n$ and $\sigma\tau$ is a $3$-cycle, then
\begin{equation}
\label{difference}
p(\tau) - p(\sigma)\in \{-2,0,2\}.
\end{equation}
\end{lem}

\begin{proof}
For all elements $\sigma\in \AAA_n$, $n-p(\sigma)$ is even.  Thus, it suffices to prove that $|p(\tau) - p(\sigma)| \le 3$.
Letting $\Stab(x)$ denote the set of fixed points of $x\in \AAA_n$ acting on $\{1,\ldots,n\}$,
$$\Stab(\sigma) \supset \Stab(\sigma\tau) \cap \Stab(\tau),$$
so $p(\tau) \ge p(\sigma)-3$, and by the same argument $p(\sigma)\ge p(\tau)-3$.
\end{proof}

\begin{lem}
If $m$ is a positive integer and $a$ is any integer, the number of elements $\sigma\in \SSS_n$ such that $p(\sigma)\equiv a\pmod m$ is 
$(m^{-1}+o(1))n!$.
\end{lem}

\begin{proof}
Let $P_{n,m,a}$ denote the number of such elements, and let $\zeta\in \C$ satisfy $\zeta^m=1$.  Then, by \cite[Corollary~5.1.8]{Stanley},
$$Q_{n,m,\zeta} := \sum_{a=0}^{m-1} \zeta^a P_{n,m,a}$$
is $n!$ times the $x^n$ coefficient of $e^{-\zeta \log(1-x)}$.  By the binomial theorem,
$$1+\sum_{n=1}^\infty \frac{Q_{n,m,\zeta}}{n!} x^n = (1-x)^{-\zeta} = \sum_{n=0}^\infty \frac{\zeta(\zeta+1)\cdots(\zeta+n-1)}{n!} x^n,$$
so
$$Q_{n,m,\zeta} = \zeta(\zeta+1)\cdots(\zeta+n-1) = \frac{\Gamma(\zeta+n)}{\Gamma(\zeta)}.$$
Stirling's approximation \cite[12.33]{WW} gives
$$\log \Gamma(z) = (z-\frac12)\log z-z+\frac{\log 2\pi}{2}+O(|z|^{-1})$$
for $\arg(z) \in [\epsilon-\pi/2,\pi/2-\epsilon]$ for each fixed $\epsilon>0$.  In particular,
taking $\epsilon < \pi/3$, this estimate holds for $\zeta+n$ for all $\zeta$ on the unit circle and all $n\ge 2$.
As
$$\log (\zeta+n) = \log n + O(n^{-1}),$$
$$\log \Gamma(\zeta+n) = (n+\Re(\zeta)-\frac12)\log n  - \log n+ O(1),$$
so
$$|\Gamma(\zeta+n)| = O(n^{\Re(\zeta) - 1} \Gamma(n+1)).$$

Together with the functional equation $\Gamma(z+1) = z\Gamma(z)$, Stirling's approximation implies that $\Gamma$ has no zeroes, so
$$Q_{n,m,\zeta} = O(\Gamma(\zeta+n)) = O(n^{\Re(\zeta) - 1} \Gamma(n+1)).$$
In particular, for $\zeta\neq 1$, we have
$$Q_{n,m,\zeta} = o(Q_{n,m,1}),$$
so
\begin{equation}
\label{asymptotic}
P_{n,m,a} = \frac 1m\sum_{\{\zeta\mid \zeta^m=1\}} \zeta^{-a} Q_{n,m,\zeta} = (m^{-1}+o(1)) Q_{n,m,1} = (m^{-1}+o(1))n!.
\end{equation}
\end{proof}

We can now prove Theorem~\ref{alt}.

\begin{proof}

A permutation $\sigma\in \SSS_n$ is even if and only if $p(\sigma)\equiv n\pmod 2$.  Therefore, if $m$ is odd,
$$|\{\sigma\in \AAA_n\mid p(\sigma)\equiv a\pmod m\}| = (m^{-1}o(1)) |\AAA_n|.$$
If $s+t\le 1$, by \eqref{asymptotic}, we can choose for each $n$, an odd integer $m_n$ in such a way that $m_n\to \infty$ as $n\to \infty$ and
\begin{equation}
\label{m-sub-n}
\sup_a\frac{|m_nP_{n,m_n,a}-n!|}{n!} \to 0.
\end{equation}
If $0 < k_n < l_n \le m_n$, $S_n\subset \AAA_n$ consists of all even permutations which are congruent to any element of $\{2,4,\ldots,2k_n-2\}$ (mod $m_n$), and $T_n$ consists of even permutations which
are congruent to any element of $\{2k_n+2,2k_n+4,\ldots, 2l_n-2\}$ (mod $m_n$), then by \eqref{difference}, $S_nT_n$ does not contain any $3$-cycle.
By construction, \eqref{m-sub-n} implies \eqref{limits}.

If $s+t > 1$, then $|S_n| + |T_n| > \frac{n!}2$ for all $n\gg 0$, so $S_n T_n = \AAA_n$ follows immediately.
\end{proof}

In the case $S=T$, Question~\ref{Weak} has a positive answer for alternating groups.  We give a stronger result in Theorem~\ref{S equals T} below.

Let $\sigma \in \SSS_n$.
Following \cite{LS2}, for every positive integer $k$,
we define $\Sigma_k(\sigma)$ to be the set of elements of $\{1,2,\ldots,n\}$ whose $\sigma$-orbit has cardinality at most $k$.  We define non-negative integers $e_1,e_2,\ldots,e_n$ so that for $1\le k\le n$,
$$n^{e_1+\cdots+e_k} = |\Sigma_k(\sigma)|.$$
Finally, we define
$$E(\sigma) = \sum_{k=1}^n \frac{e_k}k.$$



Choose $\alpha,\e >0$ so that $\alpha + 2 \e = 1/4$.
Let $W \subseteq \AAA_n$ be a subset satisfying $|\AAA_n |/|W| \le e^{n^{\alpha}}$.
By \cite[Corollary~6.5]{LS2},
there exists $N_2$ depending only on $\alpha$ such that, if $n \ge N_2$ and $\sigma \in W$ is randomly
chosen, the probability that $E( \sigma ) \le \alpha + \e = 1/4 - \e$ is at least $1 - e^{-n^{\alpha}}$.

Now, by  \cite[Corollary 1.11]{LS2}, there exists $N \ge N_2$ depending on $\e$ such that,
if $n \ge N$ then $E(\sigma) < 1/4-\e$ implies $ ( \sigma^{\SSS_n} )^2 = \AAA_n$.
It follows that, for random $\sigma \in W$, $ ( \sigma^{\SSS_n} )^2 = \AAA_n$ holds with probability
at least $1 - e^{-n^{\alpha}}$.

It is well known that $\sigma^{\AAA_n} = \sigma^{\SSS_n}$ unless $\sigma$ is a product of cycles of distinct odd lengths.
By Theorem VI of \cite{ET}, the probability that $\sigma \in \SSS_n$ does not have cycles of lengths $a_1, \ldots , a_k$ is
at most $(\sum_{i=1}^k 1/a_i)^{-1}$. Applying this with $k = \lfloor n/2 \rfloor$ and $a_i = 2i$, we conclude that
the probability that $\sigma^{\AAA_n} \ne \sigma^{\SSS_n}$ is at most $\frac2{\log n/2}$.

It now follows that, for random $\sigma \in W$, the probability that $( \sigma^{\AAA_n} )^2 = \AAA_n$ is at least
$1 - e^{-n^{\alpha}} - \frac2{\log n/2} \ge 1 - \frac3{\log{n}}$ for large $n$. Now, if we assume also that $W$ is a
normal subset of $\AAA_n$, we have $\sigma^{\AAA_n} \subseteq W$.
In summary, we have proved the following:

\begin{thm}
\label{S equals T}
For every $0 < \alpha < 1/4$ there exists $N>0$ such that, if $n \ge N$ and $W \subseteq \AAA_n$ is a normal subset satisfying
$$|W| \ge e^{-n^{\alpha}} \cdot |\AAA_n|,$$ then $W^2 = \AAA_n$.
\end{thm}

On the other hand, we have the following theorem.

\begin{thm}
\label{Strong Alt}
Even in the case $S=T$, Question~\ref{Strong} has a negative answer for alternating groups.
\end{thm}

\begin{proof}
We prove that if, for each $n$, $S_n = T_n$ denotes the set of derangements in $\AAA_n$,
then $|S_n| = |T_n| \sim \frac{n!}{2e}$ and the number of representations of any $3$-cycle as
$st$, $s\in S_n$ and $t\in T_n$ is also asymptotic to $\frac{n!}{2e}$.

The first claim is an analogue of a well-known fact about derangements in $\SSS_n$, and the argument is the same.  As $\AAA_n$ acts $n-2$-tuples transitively on $X_n=\{1,2,\ldots,n\}$, for each subset $\Sigma$
of $X_n$ with $\le n-2$ elements, the number of elements in $\AAA_n$ which fix $\Sigma$ pointwise is
$$\frac{n!}{2(n-|\Sigma|)!}.$$
Therefore,
$$\sum_{|\Sigma| = r \le n-2} |\Stab_{\AAA_n} \Sigma| = \frac{n!}{2r!}.$$
By the Bonferroni inequalities, the number of derangements in $\AAA_n$ lies between any two consecutive values of the sequence
$\sum_{r=0}^{n-3} \frac{(-1)^rn!}{2\,r!}$, where $r=1,2,\ldots,n-2$, implying the first claim.

For the second claim, it suffices to prove that in the limit $n\to\infty$, the probability approaches $1$ that the product of a given $3$-cycle in $\AAA_n$ and a uniformly distributed random should again be a derangement approaches $1$.  Without loss of generality, we take our fixed $3$-cycle to be $\sigma =(123)$
and let $\tau$ denote a random derangement in $\AAA_n$.  Then $\tau\sigma$ can fix only $1$, $2$, or $3$.  It fixes $1$ if and only if $\tau(2) = 1$, and likewise for $2$ and $3$.  By symmetry, the probability that $\tau(2)=1$ is the same as the probability that $\tau(2)$ takes any other value in $X_n\smallsetminus\{2\}$, i.e., $\frac1{n-1}$.  Thus, the probability that $\tau\sigma$ is a derangement is at least $1-\frac3{n-1}$.
\end{proof}

\bigskip

\section{Products of three normal subsets}

While Questions \ref{Weak} and \ref{Strong} have negative answers for general finite simple groups, the analogous questions for products
of three normal subsets of arbitrary finite simple groups $G$ have a positive answer. This follows easily and effectively from existing results,
even without assuming the normality of the subsets.

By the so-called Gowers trick (see Gowers \cite{G} and Nikolov-Pyber \cite{NP}), if $G$ is a finite group, $m(G)$ is the minimal degree of a non-trivial character of $G$,
and $A, B, C \subseteq G$ satisfy
$$|A|\,|B|\,|C| \ge \frac{|G|^3}{m(G)},$$
then $ABC = G$. Thus Question~\ref{Weak} for three arbitrary subsets has a positive answer, with $\e = m(G)^{-1/3}$;
this holds also for general \emph{quasi-random} families of non-simple groups, that is, provided $m(G) \to \infty$ as $|G| \to \infty$.

Question~\ref{Strong} for $t \ge 3$ subsets is solved in\cite[2.8]{BNP}, which we quote below.

\begin{thm}\label{bnp}
Let $G$ be a finite group, $t \ge 3$ an integer, and $\alpha > 0$. Let $C_1, \ldots , C_t$ be subsets of $G$ which satisfy
\[
\prod_{i=1}^t |C_i| \ge \alpha \frac{|G|^t}{m(G)^{t-2}}.
\]
For $g \in G$ let $N_g$ denote the number of solutions to the equation $x_1 \cdots x_t = g$ with $x_i \in C_i$ ($i=1, \ldots , t$).
Set
\[
E = \frac{\prod_{i=1}^t |C_i|}{|G|}.
\]
Then, for every $g \in G$ we have
\[
|N_g - E| \le \alpha^{-1/2}E.
\]
\end{thm}

For a group $G$ and subsets $C_1, \ldots , C_t$ of $G$, denote by $\Pr_{C_1, \ldots ,C_t}$ the probability distribution on $G$
such that, for $g \in G$, $\Pr_{C_1, \ldots ,C_t}(g)$ is the probability that $x_1 \cdots x_t = g$ where $x_i \in C_i$ ($i=1, \ldots ,t$)
are randomly chosen, uniformly and independently.

We also denote by $\bfU_G$ the uniform distribution on $G$.

Theorem \ref{bnp} for $t=3$ yields the following.

\begin{cor}\label{ABC} For finite groups $G$, and subsets $A, B, C \subseteq G$ satisfying
$$m(G)|A|\,|B|\,|C|/|G|^3 \to \infty$$
as $|G| \to \infty$, we have
\[
\Vert \Pr_{A,B,C} - \bfU_G\Vert _{L^{\infty}} \to 0 {\rm \; as \;} |G| \to \infty.
\]
In particular we have $ABC = G$ for $|G| \gg 0$.

These two conclusions hold when $G$ is a finite simple group and $A, B, C \subseteq G$ are
subsets of sizes $\ge \e |G| > 0$ for any fixed $\e > 0$.
\end{cor}


For finite simple classical groups $G$ and normal subsets $R, S, T \subseteq G$ we can obtain
$RST = G$ under asymptotically weaker assumptions. The proof uses character methods.

For a real number $s$ let
\[
\zeta^G(s) = \sum_{\chi \in \Irr(G)} \chi(1)^{-s}.
\]
Then $\zeta^G$ is the Witten zeta function of $G$, studied in \cite{LiSh1, LiSh2}.

Suppose $C_i$ above are conjugacy classes of $G$. Then \eqref{Frobenius} implies that
\[
\Pr_{C_1, C_2, C_3}(g) = |G|^{-1} \sum_{\chi \in \Irr(G)} \frac{\chi(C_1)\chi(C_2)\chi(C_3)\chi(g^{-1})}{\chi(1)^2},
\]
where $\chi(C_i)$ is the common value of $\chi$ on elements of $C_i$.

Since $|\chi(g^{-1})|/\chi(1) \le 1$, this yields
\begin{equation}\label{uniform}
|\Pr_{C_1,C_2,C_3}(g)-|G|^{-1}| \leq \sum_{1 \ne \chi \in \Irr(G)} \frac{|\chi(C_1)|\,|\chi(C_2)|\,|\chi(C_3)|}{\chi(1)}.
\end{equation}
Denote by $\Cl_n(q)$ the set of finite simple classical groups over $\F_q$ with an $n$-dimensional natural module.
We need the following slight extension of \cite[7.5]{GLT} and its proof.

\begin{prop}\label{Linf} There exists an absolute constant $0<\gamma <1$ such that the following holds.
Suppose $n \ge 9$, $G \in \Cl_n(q)$, and for $i = 1, 2, 3$ let $g_i \in G$ satisfy $|C_G(g_i)| \le |G|^{\gamma}$. Set $C_i = g_i^G$
($i = 1,2,3$). Then we have
\begin{enumerate}[\rm(i)]
\item $\lim_{|G| \to \infty} \Vert  \Pr_{C_1,C_2,C_3} - \bfU_G \Vert _{\infty} = 0$.

\item There exists an absolute constant $N$ such that, if $|G| \ge N$, then $C_1C_2C_3=G$.
\end{enumerate}
\end{prop}

\begin{proof}
By Theorem 1.3 of \cite{GLT} we may choose $0 < \gamma < 1$ such that, if $g \in G$ satisfies $|C_G(g)| \le |G|^{\gamma}$,
then $|\chi(g)| \le \chi(1)^{1/4}$ for all $\chi \in \Irr(G)$.

Let $g_i, C_i$ be as in the statement of the proposition.
Then $|\chi(g_i)| \le \chi(1)^{1/4}$, and therefore inequality (\ref{uniform}) above shows that
\[
|\Pr_{C_1,C_2,C_3}(g)-|G|^{-1}| \le |G|^{-1} \sum_{1 \ne \chi \in \Irr(G)} \chi(1)^{-1/4} = |G|^{-1}(\zeta^G(1/4)-1).
\]
By \cite[1.1]{LiSh2} and our assumptions on $G$, it follows that $\zeta^G(1/4) - 1 \to 0$ as $|G| \to \infty$.
This completes the proof of part (i).

Part (ii) follows from part (i) and the effective nature of the proof of \cite[1.1]{LiSh2}.
\end{proof}

We note that the results \cite[2.4, 2.5]{Sh3} provide a weaker version of Proposition \ref{Linf}. More specifically, these results
show that the conclusions of Proposition \ref{Linf} hold if we assume
$$|C_G(g_i)| \le q^{(4/3-\delta)r},\ i=1,2,3$$
for any fixed $\delta > 0$ and $N = N(\delta)$.

Proposition \ref{Linf} easily implies the following.

\begin{thm} There exist an absolute constant $\delta > 0$ and an integer $N$ such that the following holds.
Suppose $n \ge N$, $G \in \Cl_n(q)$, and $R, S, T \subseteq G$ are normal subsets satisfying
$|R|, |S|, |T| \ge |G|^{1 - \delta}$. Then $RST = G$.
\end{thm}

\begin{proof}
Let $\gamma$ be as in Proposition \ref{Linf}, and define, say, $\delta = \gamma/2$.

Suppose $G$ above has rank $r$.
Then, by \cite{FG1}, we have $k(G) \le cq^r$, for a small absolute constant $c > 0$.
Clearly, $R, S, T$ contain conjugacy classes $C_1, C_2, C_3$ respectively satisfying
\[
|C_i| \ge \frac{|G|^{1-\delta}}{k(G)} \ge c^{-1} q^{-r} |G|^{1-\delta} \ge |G|^{1-\gamma/2-o_r(1)} \ge |G|^{1-\gamma},
\]
provided $N$ is large enough and $r \ge N$.

It follows from Proposition \ref{Linf} that (enlarging $N$ if needed) $C_1 C_2 C_3 = G$
and hence $RST = G$.
\end{proof}

\bigskip

\section{An application to word maps}

Probabilistic Waring problems for finite simple groups are studied  \cite{LST2}.
For a word $w \in F_d$ and a finite group $G$, let $\Pr_{w,G}$ denote the probability induced
by the corresponding word map $w:G^d \to G$, namely $\Pr_{w,G}(g) = |w^{-1}(g)|/|G|^d$ for $g \in G$.

It is shown in \cite{LST2} that for every $l \in \N$ there exists $N=N(l)$ such that, if $w_1, \ldots , w_N \in F_d$ are non-trivial
words in pairwise disjoint sets of variables, then
\[
\Vert  \Pr_{w_1 \cdots w_N,G} - \bfU_G \Vert _{\infty} \to 0 \; {\rm as} \; |G| \to \infty,
\]
where $G$ ranges over the finite simple groups. The dependence of $N$ on $l$ is genuine.
It turns out that, if we change the probabilistic model, let $G$ be a finite simple group of Lie type,
choose random elements $g_i \in w_i(G)$ and study the distribution of $g_1 \cdots g_N$, we obtain
an almost uniform distribution in $L^{\infty}$ much faster, namely in two or three steps.

\begin{thm} Let $w_1, w_2, w_3 \in F_d$ be non-trivial words and let $G$ be a finite simple group.
\begin{enumerate}[\rm(i)]
\item Suppose $G$ is of Lie type of bounded rank. Then
\[
\Vert  \Pr_{w_1(G),w_2(G)} - \bfU_G \Vert_{L^{\infty}} \to 0 \; {\rm as} \; |G| \to \infty.
\]

\item Suppose $G$ is a classical group. Then
\[
\Vert  \Pr_{w_1(G),w_2(G),w_3(G)} - \bfU_G \Vert _{L^{\infty}} \to 0 \; {\rm as} \; |G| \to \infty.
\]
\end{enumerate}
\end{thm}

\begin{proof}
Let $G$ be as in part (i). By \cite{L} there exists $N, \e > 0$ such that, if $|G| \ge N$ then $|w_i(G)| \ge \e |G|$ for $i = 1,2$.
The conclusion now follows from part (iv) of Theorem A.

To prove part (ii), we may assume, applying part (i), that the rank $r$ of $G$ tends to infinity.
Theorem 1.12 of \cite{LS1} shows that, if $G$ is symplectic or orthogonal, then $|w_i(G)| \ge cr^{-1} |G|$ ($i=1,2,3$),
where $c > 0$ is an absolute constant. Since $m(G) \ge b q^r$ for fixed $b> 0$
(see \cite{FG1}) we have
\begin{equation}\label{inf}
\frac{m(G)|w_1(G)|\,|w_2(G)|\,|w_3(G)|}{|G|^3} \to \infty \; {\rm as} \; |G| \to \infty.
\end{equation}
In the case where $G$ is $\PSL_n(q)$ or $\PSU_n(q)$, Propositions 1.7 and 1.8 of \cite{NP}
show that $|w_i(G)| \ge q^{-n/4+o_n(1)}|G|$ ($i=1,2,3$), which implies (\ref{inf}) for $n \gg 0$.

The desired conclusion now follows from Theorem \ref{ABC}.
\end{proof}

\bigskip

\section{Character estimates and product results}\label{products}
In this section, we prove several results concerning character values and products of conjugacy classes in finite simple groups of Lie type, which
will be needed in the next section and which may be of independent interest.

\subsection{Groups of type $A_n$ and $\tw2 A_n$}
\begin{prop}\label{A-degree}
For all integers $L$ there exists a constant $A = A(L)>0$ such that for all integers $n\ge L$ and all prime powers $q$,
the degree of the unipotent character of $\GL_n(\F_q)$ associated to a partition whose largest piece is $n-L$ is at least $q^{\frac{n^2-n}2-A}$.
\end{prop}

\begin{proof}
Choosing $A$ large enough, without loss of generality, we may assume $n>2L$.
The partition $\lambda = \lambda_1\ge \lambda_2\ge \cdots$ of $n$ associated to the character has $\lambda_1 = n-L$.
It is well known (see, for instance, \cite[(21)]{O} or \cite{M1}) that the
unipotent characters of $\GL_n(\F_q)$  have degree
$$\chi_\lambda(1) = q^{\sum_i \binom{\lambda_i}{2}} \frac{\prod_{j=1}^n (q^j-1)}{\prod_{k=1}^n (q^{h_k}-1)},$$
where $h_k$ denotes the hook of the $k$th box in the Ferrers diagram of $\lambda$.  Now, the last $n-2L$ boxes in the first row of the Ferrers diagram belong to one-box columns.
Therefore, their hooks have lengths $n-2L,\ldots,3,2,1$.
All hooks of boxes not in the first row have lengths $\le L$, and the hooks
of the first $L$ boxes in the first row have length $\le n$.  We conclude that
$$\frac{\prod_{j=1}^n (q^j-1)}{\prod_{k=1}^n (q^{h_k}-1)} \ge \frac{\prod_{j=n-2L+1}^n (q^j-1)}{q^{L^2+Ln}}.$$
As
$$\prod_{i=1}^\infty (1-q^{-i}) > 1/4 \ge q^{-2},$$
we have
$$\dim \chi_\lambda(1) > q^{\binom{\lambda_1}2}q^{-2+L(n+(n-2L+1)) - L^2-Ln} = q^{\frac{n^2-n-5L^2+3L-4}2}.$$
\end{proof}

Up to conjugacy, $\F_q$-rational maximal tori in the algebraic groups $\SL_n$ and $\SU_n$ over a finite field $\F_q$ are both indexed by partitions of $n$.
We do not distinguish between the maximal torus as an algebraic group and the finite subgroup of $G$ obtained by taking $\F_q$-points.
If $G$ is either $\SL_n(q)$ or $\SU_n(q)$, and $a_1,\ldots,a_k$ are positive integers summing to $n$
(not necessarily in order), then we denote by $T_{a_1,\ldots,a_k}<G$ a maximal torus in the class belonging to the partition with parts $a_1,\ldots,a_k$.

\begin{thm}\label{A-product1}
Let $a \geq 3$ be a fixed positive integer.  Then there exists an integer $N = N(a) \geq 2a^2+6$
such that the following statements hold whenever $n>N$, $q$ any prime power, and $G = \SL_n(q)$ or $\SU_n(q)$.
\begin{enumerate}[\rm(i)]
\item If $t_1$ and $t'_1$ are regular semisimple elements of $G$ belonging to tori $T$ and $T'$ of type
$T_n$ and $T_{1,a,n-a-1}$ respectively, then $t_1^G\cdot (t'_1)^G \supseteq G \smallsetminus \ZB(G)$.
\item If $t_2$ and $t'_2$
are regular semisimple elements of $G$ belonging to tori $T$ and $T'$ of type
$T_{1,n-1}$ and $T_{a,n-a}$ respectively, then $t_2^G\cdot (t'_2)^G \supseteq G \smallsetminus \ZB(G)$.
\end{enumerate}
\end{thm}

\begin{proof}
(i) Consider any $g \in G \smallsetminus \ZB(G)$, and any $\chi \in \Irr(G)$ such that
$\chi(t_1)\chi(t'_1) \neq 0$. By \cite[Proposition 3.1.5]{LST1} and its proof, then $\chi = \chi^{(n-k,1^k)}$, the unipotent
character labeled by $(n-k,1^k)$ with $k = 0$ (the principal character $1_G$), $k=a$, $k = n-a-1$, or
$k= n-1$ (the Steinberg character $\St$); moreover, $|\chi(t_1)\chi(t'_1)| = 1$, and the last two characters both have degree
$\geq C|G|/q^n$ for a universal constant $C > 0$. The character $\chi_2:=\chi^{(n-a,1^a)}$ has level
$$a \leq \min\{\sqrt{n-3/4}-1/2,\sqrt{(8n-17)/12}-1/2\}$$
by \cite[Theorem 3.9]{GLT1}, and so $\chi_2(1) > q^{a(n-a)-3}$ by \cite[Theorem 1.3]{GLT1} and
$$|\chi_2(g)| \leq (2.43)\chi_2(1)^{1-1/n}$$
by \cite[Theorem 1.6]{GLT1}. In particular,
$$|\chi_2(g)|/\chi_2(1) \leq 2.43/\chi_2(1)^{1/n} \leq 2.43/q^{a-1/2} \leq 2.43/2^{2.5} < 0.43.$$
On the other hand, for the latter two (large degree) characters, by \cite[Proposition 6.2.1]{LST1} we have
$|\chi(g)|/\chi(1) < 0.25$ if we take $N(a)$ large enough. It follows that
$$\biggl{|}\sum_{\chi \in \Irr(G)}\frac{\chi(t_1)\chi(t'_1)\overline{\chi(g)}}{\chi(1)}\biggr{|} \geq 1-0.43 - 2(0.25) = 0.07 > 0,$$
and so $g \in t_1^G \cdot (t'_1)^G$.

\smallskip
(ii) Suppose $\chi \in \Irr(G)$ is such that
$\chi(t_2)\chi(t'_2) \neq 0$. By \cite[Proposition 3.1.5]{LST1} and its proof, we again have $\chi = 1_G$,
$\chi_2:=\chi^{(n-a,2,1^{a-2})}$, $\chi^{(a,2,1^{n-a-2})}$, or $\St$; moreover, $|\chi(t_2)\chi(t'_2)| = 1$, and the last two characters both have degree $\geq C|G|/q^n$ for a universal constant $C > 0$. Now we can repeat the arguments in (i) verbatim.
\end{proof}

We also need a similar result, using \cite[Proposition 8.4]{GLOST} and its notation.

\begin{thm}\label{A-product2}
There exists an integer $N \geq 32$ such that if $t$ and $t'$ are regular semisimple elements of $G$ belonging to tori $T$ and $T'$ of type
$T_{n-2,2}$ and $T_{n-3,3}$ respectively, then $t^G\cdot (t')^G \supseteq G \smallsetminus \ZB(G)$ in each of the following cases:
\begin{enumerate}[\rm(i)]
\item $G = \SL_n(q),\,n\ge N$,
\item $G = \SL_n(q),\,n\ge 7,\,q>7^{481}$,
\item $G = \SU_n(q),\,n\ge N,\,q\ge 3$,
\item $G = \SU_n(q),\,n\ge 7,\,q>7^{481}$.
\end{enumerate}

\end{thm}

\begin{proof}
Suppose $\chi \in \Irr(G)$ is such that
\begin{equation}\label{a21}
  \chi(t)\chi(t') \neq 0.
\end{equation}
By \cite[Proposition 8.4]{GLOST}, the two tori are weakly orthogonal, hence $\chi=\chi^\lambda$ is a
unipotent character labeled by a partition $\lambda \vdash n$. Now, as in the proof of \cite[Proposition 3.1.5]{LST1},
the condition \eqref{a21} implies that the irreducible character $\psi^\lambda$ of $\SSS_n$ labeled by $\lambda$ takes nonzero
values at permutations $\sigma_1 = (1,2)(3,4,\ldots,n)$ and $\sigma_2 = (1,2,3)(4,5,\ldots,n)$. By the Murnaghan-Nakayama
rule \cite[Proposition 3.1.1]{LST1} and by \cite[Corollary 3.1.2]{LST1}, it follows that we can remove a rim $(n-2)$-hook from the Young diagram $Y(\lambda)$ of $\lambda$ and likewise we can remove a rim $(n-3)$-hook from $Y(\lambda)$ (so that the remainder
is a proper diagram). The list of $\lambda$ that a rim $(n-2)$-hook can be removed from $Y(\lambda)$ is given
in \cite[Corollary 3.1.4]{LST1}. Checking through them for a removal of a rim $(n-3)$-hook, we see that $\lambda$ is one of
the following $8$ partitions
$$(n),~(1^n),~\lambda_2:=(n-1,1),~(2,1^{n-2}),~\lambda_3:=(n-3,3),~(2^3,1^{n-6}),~\lambda_4:=(n-4,2^2),~(3^2,1^{n-6}).$$
Moreover, \cite[Proposition 3.1.1]{LST1} implies that
\begin{equation}\label{a22}
  \chi^\lambda(t)\chi^\lambda(t') = \pm 1
\end{equation}
in all these cases. Let $\e=1$ if $G = \SL_n(q)$ and $\e=-1$ if $G = \SU_n(q)$. Using \cite[\S13.8]{Ca}, we can write down the degrees of
these $8$ characters:
\begin{equation}\label{a23}
  \begin{array}{ll}\chi^{(n)}(1)& =1,\\
     \chi^{(1^n)}(1)& =q^{n(n-1)/2},\\
     \chi^{(n-1,1)}(1)& =q\frac{q^{n-1}+\e^n}{q-\e},\\
     \chi^{(2,1^{n-2})}(1)& =q^{n(n-1)/2-(n-1)}\frac{q^{n-1}+\e^n}{q-\e},\\
     \chi^{(n-3,3)}(1)& =q^3\frac{(q^n-\e^n)(q^{n-1}-\e^{n-1})(q^{n-5}-\e^{n-5})}{(q^3-\e^3)(q^2-\e^2)(q-\e)},\\
     \chi^{(2^3,1^{n-6})}(1)& =q^{n(n-1)/2-(3n-9)}\frac{(q^n-\e^n)(q^{n-1}-\e^{n-1})(q^{n-5}-\e^{n-5})}{(q^3-\e^3)(q^2-\e^2)(q-\e)},\\
     \chi^{(n-4,2^2)}(1)& =q^6\frac{(q^n-\e^n)(q^{n-1}-\e^{n-1})(q^{n-4}-\e^{n-4})(q^{n-5}-\e^{n-5})}{(q^3-\e^3)(q^2-\e^2)^2(q-\e)},\\
     \chi^{(3^2,1^{n-6})}(1)& =q^{n(n-1)/2-(4n-12)}
        \frac{(q^n-\e^n)(q^{n-1}-\e^{n-1})(q^{n-4}-\e^{n-4})(q^{n-5}-\e^{n-5})}{(q^3-\e^3)(q^2-\e^2)^2(q-\e)}.\end{array}
\end{equation}
The first two characters in this list are the principal character $1_G$ and the Steinberg character $\St$ of $G$.

Next, consider any $g \in G \smallsetminus \ZB(G)$. If $q > 7^{481}$, then using \eqref{a22} and \cite[Theorem 1.2.1]{LST1} we get
$$\biggl{|}\sum_{\chi \in \Irr(G)}\frac{\chi(t)\chi(t')\overline{\chi(g)}}{\chi(1)}\biggr{|} \geq 1-\frac{7}{q^{1/481}} > 0,$$
and so $g \in t^G \cdot (t')^G$.

Now we may assume $n > N$.
Since $N \geq 32$, $\chi_i :=\chi^{\lambda_i}$ with $i = 3,4$ has level
$$i \leq \min\{\sqrt{n-3/4}-1/2,\sqrt{(8n-17)/12}-1/2\}$$
by \cite[Theorem 3.9]{GLT1}, and so
\begin{equation}\label{a24}
 \frac{|\chi_i(g)|}{\chi_i(1)} \leq \frac{2.43}{\chi_i(1)^{1/n}}
\end{equation}
by \cite[Theorem 1.6]{GLT1}; furthermore,
\begin{equation}\label{a25}
  \chi_3(1) > q^{3n-12},~~\chi_4(1) > q^{4n-15}.
\end{equation}
On the other hand, $\chi_2:=\chi^{\lambda_2}$ is a unipotent Weil character, and using the character formula \cite[Lemma 4.1]{TZ1}, one
can show that
\begin{equation}\label{a26}
 \frac{|\chi_2(g)|}{\chi_2(1)} \leq \frac{q^{n-1}+q^2}{q^n-q}.
\end{equation}
Note that the second, fourth, sixth, and eighth characters in \eqref{a23} have degree $>q^{n(n-1)/2-9}$.
Using \cite[Proposition 6.2.1]{LST1} as in the proof of Theorem \ref{A-product1}, we have
$$\frac{|\chi(g)|}{\chi(1)} < 0.01$$
for all four of them, if we take $N$ large enough. We also note that
$$\lim_{n \to \infty} \biggl( \frac{q^{n-1}+q^2}{q^n-q} + \frac{2.43}{q^{(3n-12)/n}} + \frac{2.43}{q^{(4n-15)/n}}\biggr) =
    \frac{1}{q} + \frac{2.43}{q^3} + \frac{2.43}{q^4} < 0.956$$
which implies by \eqref{a24}, \eqref{a25}, \eqref{a26} that
$$\sum^{4}_{i=2}\frac{|\chi_i(g)|}{\chi_i(1)} < 0.957$$
when $N$ is large enough.  It now follows from \eqref{a22} that
$$\biggl{|}\sum_{\chi \in \Irr(G)}\frac{\chi(t)\chi(t')\overline{\chi(g)}}{\chi(1)}\biggr{|} \geq 1-0.957 - 0.04 = 0.003,$$
and so $g \in t^G \cdot (t')^G$.
\end{proof}

In fact, for $\SU_n(2)$ we will need an analogue of Theorem \ref{A-product2} for tori of types $T_{3,n-3}$ and $T_{4,n-4}$.  We begin by classifying
characters $\SSS_n$ which vanish on neither of the corresponding permutations.

\begin{prop}
\label{Twelve}
Let $n\ge 10$, and let
$$\sigma_1 = (1,2,3)(4,\ldots,n),\,\sigma_2 = (1,2,3,4)(5,\ldots,n)\in \SSS_n.$$
There are exactly twelve characters $\psi = \psi^\lambda$ of $\SSS_n$ such that $\psi(\sigma_1)\psi(\sigma_2)\neq 0$, for each of these characters, the product is $\pm 1$,
and for each such $\lambda$, either $\lambda$ or its transpose belongs to the following set:
$$\{(n),(n-1,1),(n-2,1^2),(n-4,4),(n-5,3,2),(n-6,2^3)\}.$$
\end{prop}

\begin{proof}
As $\lambda \vdash n\ge 10$, transposing if necessary, we may assume $\lambda_1 \ge 4$.
As $\psi(\sigma_1)\neq 0$, by the Murnaghan-Nakayama rule, removal of a rim $n-3$-hook leaves a Young diagram $\mu$ with $3$
boxes, and it follows that this rim hook must include the last box in the first row (which implies, in particular, that there is no other rim $n-3$-hook, so the character value at $\sigma_1$ is $\pm 1$).  There are three cases to consider.

\smallskip
(i) $\mu = (3)$.  In this case $\lambda$ must be $(n)$ or $(n-k-4,4,1^k)$ for $0\le k\le n-8$.

\smallskip
(ii) $\mu = (2,1)$.  In this case $\lambda$ must be $(n-1,1)$, $(n-3,3)$, or $(n-k-5,3,2,1^k)$ for $0\le k\le n-8$.

\smallskip
(iii) $\mu = (1^3)$.  In this case, $\lambda$ must be $(n-2,1^2)$, $(n-3,2,1)$, $(n-4,2^2)$, or
$(n-6-k,2^3,1^k)$, where $0\le k\le n-8$.

\smallskip
As $\psi(\sigma_2)\neq 0$, $\lambda$ must have a rim $n-4$-hook whose removal leaves a Young diagram which is a $4$-hook.  In case (i), this is possible for $(n)$ and possible for $(n-k-4,4,1^k)$ if and only if $k=0$.  In case (ii), this is possible for $(n-1,1)$, impossible for $(n-3,3)$, and possible for $(n-5-k,3,2,1^k)$ if and only if $k=0$.  In case (iii), this is possible only for $(n-2,1^2)$ and $(n-6,2^3)$.
In every case where it is possible, the rim hook contains the last box in the first row and is therefore unique, implying that $\psi(\sigma_2)$ is $\pm 1$.
\end{proof}

\begin{thm}\label{A-product3}
There exists an integer $N \geq 43$ such that the following statement holds for $G = \SU_n(2)$ with $n > N$.
If $t$ and $t'$ are regular semisimple elements of $G$ belonging to tori $T$ and $T'$ of type
$T_{n-3,3}$ and $T_{n-4,4}$ respectively, and $g \in G$ has $\supp(g) \geq 2$, then $g \in t^G\cdot (t')^G$.
\end{thm}

\begin{proof}
Suppose $\chi \in \Irr(G)$ is such that
\begin{equation}\label{a31}
  \chi(t)\chi(t') \neq 0.
\end{equation}
By \cite[Proposition 8.4]{GLOST}, the two tori are weakly orthogonal, hence $\chi=\chi^\lambda$ is a
unipotent character labeled by a partition $\lambda \vdash n$.
Then, by Proposition~\ref{Twelve},
$\lambda$ is one of the following $6$ partitions
$$(n),~\lambda_1:=(n-1,1),~\lambda_2:=(n-2,1^2),~\lambda_4:=(n-4,4),~\lambda_5:=(n-5,3,2),~\lambda_6:=(n-6,2^3)$$
or their dual partitions $\lambda_i$, $7 \leq i \leq 12$; moreover,
\begin{equation}\label{a32}
  \chi^\lambda(t)\chi^\lambda(t') = \pm 1
\end{equation}
in all these cases. Let $\chi_i :=\chi^{\lambda_i}$ for $i \geq 2$.
Since $N \geq 43$, $\chi_i$ with $i = 4,5,6$ has level
$i \leq \sqrt{n-3/4}-1/2$ by \cite[Theorem 3.9]{GLT1}, and so
\begin{equation}\label{a34}
 \frac{|\chi_i(g)|}{\chi_i(1)} \leq \frac{2.43}{\chi_i(1)^{1/n}}
\end{equation}
by \cite[Theorem 1.6]{GLT1}; furthermore, with $q:=2$ we have
\begin{equation}\label{a35}
  \chi_i(1) > q^{in-i^2-3}
\end{equation}
by \cite[Theorem 1.2]{GLT1}.
On the other hand, $\chi_1$ is a unipotent Weil character, and using the character formula \cite[Lemma 4.1]{TZ1} and the
assumption $\supp(g) \geq 2$, one can show that
\begin{equation}\label{a36}
 |\chi_1(g)| \leq \frac{q^{n-2}+q^2}{q+1} < q^{n-3},~\frac{|\chi_1(g)|}{\chi_1(1)} \leq \frac{q^{n-2}+q^3}{q^n-q}.
\end{equation}
Next, as shown in \cite[Table 7.1]{M}, $\chi_2 = \chi_1\overline\chi_1-1_G$ with $\chi_2(1) > q^{2n-4}$.
Together with \eqref{a36}, this implies that
\begin{equation}\label{a37}
 \frac{|\chi_2(g)|}{\chi_2(1)} < \frac{q^{2n-6}}{q^{2n-4}} = \frac{1}{q^2}.
\end{equation}
By explicitly writing down the degrees of $\chi_j$ with $7 \leq j \leq 12$ using \cite[\S13.8]{Ca}, or by (applying Ennola's duality to)
Proposition \ref{A-degree}, we can show that there is some universal constant $A > 0$ such that
$\chi_j(1) > q^{n(n-1)/2-A}$. Using \cite[Proposition 6.2.1]{LST1} as in the proof of Theorem \ref{A-product1}, we have
$$|\chi(g)|/\chi(1) < 0.01$$
for all six of them, if we take $N$ large enough. We also note that
$$\lim_{n \to \infty} \biggl( \frac{q^{n-2}+q^3}{q^n-q} + \frac{1}{q^2}+
    \sum_{i=4,5,6}\frac{2.43}{q^{(in-i^2-3)/n}}\biggr)  < 0.77$$
which implies by \eqref{a34}--\eqref{a37} that
$$\sum_{i=1,2,4,5,6}\frac{|\chi_i(g)|}{\chi_i(1)} < 0.78$$
when $N$ is large enough.  It now follows from \eqref{a32} that
$$\biggl{|}\sum_{\chi \in \Irr(G)}\frac{\chi(t)\chi(t')\overline{\chi(g)}}{\chi(1)}\biggr{|} \geq 1-0.78 - 0.06 = 0.16,$$
and so $g \in t^G \cdot (t')^G$.
\end{proof}

\subsection{Other classical types: symbols, hooks, and cohooks}
To treat the unipotent characters of finite simple groups of orthogonal and symplectic types, we use Lusztig's theory of symbols \cite{Lu}.
If $X\subset \N$ is a set of natural numbers, we define the shift $\schS(X) = \{0\} \cup \{x+1\mid x\in X\}$.
If $X$ is finite, we  define the \emph{inefficiency} of $X$ to be
$$i(X) = -\binom {|X|}2 + \sum_{x\in X} x.$$
Thus, $i(\schS(X)) = i(X)$.  Every finite $X$ is uniquely of the form $\schS^m(X')$ for some $X'$ which does not contain $0$, and since $i(X') \ge |X'|$, there are only finitely many possibilities for $X'$ given $i(X)$.

A \emph{$d$-hook} in $X$ is an element $x\in X$ such that $x-d\in \N\smallsetminus X$;
in what follows we also label this hook by $(x-d,x)$.
If $x$ is a $d$-hook of $X$, then  \emph{removing the $d$-hook $x$}
means replacing $x$ by $x-d$ in $X$.  The resulting set $X'$ satisfies $i(X') = i(X)-d$.  In particular, if $X$ contains a $d$-hook, then $i(X)\ge d$.

We say a $d$-hook $x$ and a $d'$-hook $x'$ are \emph{disjoint} if $x-d\neq x'-d'$.  If, in addition, $x\neq x'$, it is possible to remove both the $d$-hook $x$ and the $d'$-hook $x'$, so $i(X)\ge d+d'$.
Even if $x=x'$, we still have
\begin{equation}
\label{inefficiency}
i(X)\ge d+d'-1.
\end{equation}

We recall that a \emph{symbol} is an ordered pair $(X,Y)$ of finite subsets of
$\N$.  We define equivalence of symbols by imposing the relations
$(X,Y)\sim (Y,X)$ and $(X,Y)\sim (\schS(X),\schS(Y))$ and taking transitive closure.
If $X=Y$, the symbol is \emph{degenerate}.
We will say a symbol is \emph{minimal} if $0\not\in X\cap Y$; in particular, every symbol is equivalent to at least one minimal symbol.
The \emph{rank} of a symbol is given by
\begin{equation}
\label{s-rank}
r = -\Bigl\lfloor\frac{(|X|+|Y|-1)^2}4 \Bigr\rfloor+\sum_{x\in X} x + \sum_{y\in Y} y
= i(X) + i(Y) + \Bigl\lfloor\frac{(|X|-|Y|)^2}4 \Bigr\rfloor.
\end{equation}
For any $q$, the unipotent representations of orthogonal and symplectic groups of Lie type of rank $r$ for specified $q$ are
given by symbols of rank $r$; equivalence classes of symbols with $|X|-|Y|$ odd correspond to representations of groups of type $B_r$ and $C_r$, and those with $|X|-|Y|$ divisible by $2$ but not $4$
correspond to representations of groups of type $^2D_r$.  Those with $|X|-|Y|$ divisible by $4$ correspond to representations of type $D_r$ , with the additional proviso that each {\it degenerate} symbol class, that is where $X=Y$, corresponds to a
pair of unipotent representations for groups of type $D_r$.

By a $d$-hook of a symbol $(X,Y)$, we mean either a $d$-hook of $X$ or a $d$-hook of $Y$.  Any hook of $X$ is considered to be disjoint to any hook of $Y$.
If $d+d'-1 > r$, by \eqref{inefficiency} and \eqref{s-rank}, a symbol $(X,Y)$ cannot have a disjoint $d$-hook and $d'$-hook.
By a $d$-\emph{cohook} of $(X,Y)$ we mean either an element $x\in X$ such that $x-d\in \Z_{\geq 0}\smallsetminus Y$ or $y\in Y$ such that $y-d\in \Z_{\geq 0}\smallsetminus X$; again, we will sometimes label this cohook by $(x-d,x)$.
A $d$-cohook $x\in X$ and a $d'$-cohook $x'\in X$ are disjoint if and only if $x-d\neq x'-d'$, and likewise for two cohooks in $Y$; every cohook in $X$ is disjoint from every cohook in $Y$.
Removing a $d$-cohook $x\in X$ means removing $x$ from $X$ adding $x-d$ to $Y$, and likewise for removing a cohook $y\in Y$; either way, the effect is to reduce the rank of the symbol by $d$.
Again, if $d+d'-1>r$ it is impossible for a symbol of rank $r$ to have a $d$-cohook and a $d'$-cohook which are disjoint.

We also recall that the degree of the unipotent representation labeled by $S=(X,Y)$ is given by
\begin{equation}\label{deg11}
  q^{a(S)}\frac{|G|_{q'}}{2^{b(S)}\prod_{(b,c)\mbox{ {\tiny hook}}}(q^{c-b}-1)
            \prod_{(b,c)\mbox{ {\tiny cohook}}}(q^{c-b}+1)}
\end{equation}
for some integers $a(S),b(S) \geq 0$, (see \cite[Bem.~3.12 and~6.8]{M1}).

\begin{prop}
\label{Hooks and cohooks}
\begin{enumerate}[\rm(i)]
\item If $k < k'$ are fixed, there exists a bound $B=B_1(k,k')$ such that for each $r$ there are at most $B$ symbols of rank $r$ which contain both an $(r-k)$-hook and an $(r-k')$-hook.
\item If $k < k'$ are fixed, there exists a bound $B=B_2(k,k')$ such that for each $r$ there are at most $B$ symbols of rank $r$ which contain both an $(r-k)$-cohook and an $(r-k')$-cohook.
\item If $k$ and $k'$ are fixed (and possibly equal), there exists a bound $B=B_3(k,k')$ such that for each $r$ there are at most $B$ symbols of rank $r$ which contain both an $(r-k)$-hook and an $(r-k')$-cohook.
\end{enumerate}
\end{prop}

\begin{proof}
First we consider the case of two hooks.
Let $d = r-k$ and $d' = r-k'$.  If $r$ is sufficiently large, $d+d'-1>r$, so $(X,Y)$ cannot have a disjoint $d$-hook and $d'$-hook.  Without loss of generality, we may assume that the two hooks belong to $X$, so there exists $z\in \Z_{\geq 0}\smallsetminus X$ such that $z+d,z+d'\in X$.
Moreover, we may assume $(X,Y)$ is minimal, so $0\not\in X\cap Y$.
Let $(X',Y)$ denote the symbol obtained by removing the $d$-hook from $(X,Y)$.
By \eqref{s-rank},
$$\biggl\lfloor\frac{(|X'|-|Y|)^2}4\biggr\rfloor \le k,$$
so $|X|-|Y|$ is bounded as $r\to\infty$, so $|Y|$ grows without bound.  Moreover, $i(Y)$ is also bounded.
If $0\not\in Y$, then $i(Y) \ge |Y|$, so it follows that if $r$ is sufficiently large, $0\in Y$, which means $0\not\in X$.  If $z\neq 0$, then $X'$ contains the $z+d'$-hook $z+d'$.  Thus $z=0$, and moreover,
$X'$ is of the form $\schS^m(X'')$ for some non-negative integer $m$ and some $X''$ not containing $0$.
As $i(X'')$ is bounded above, there are only finitely many possibilities for $X''$.  However, $X'$ contains $d'$, so $d'-m$ is bounded, and therefore $n-m$ is bounded.  It follows that the number of possibilities for $X'$ and therefore $X$ is bounded as $r\to \infty$.  As $|X|$ determines $|Y|$ up to a bounded number of possibilities, and $|Y|$ and $i(Y)$ determine $Y$ up to a bounded number of possibilities, it follows that the number of possibilities for $(X,Y)$ is bounded.

Next we consider two cohooks and a minimal symbol $(X,Y)$.  Without loss of generality, we may assume that $x,x'\in X$,
and $x-d = x'-d' = z\in \Z_{\geq 0}\smallsetminus Y$.  Let $(X',Y')$ denote the symbol obtained by removing the cohook $x$ from $(X,Y)$.
As before $d'\in X'$, while  $|X'|-|Y'|$, $i(Y')$, and $r-d'$ are bounded independent of $r$.  This implies successively that $0\in X'\cap Y'$, $0\in X$,  $0\not\in Y$,  $z=0$, and $r-x$ bounded.
As $i(X')\le k$, we can write $X' = \schS^m(X'')$, $0\not\in X''$, where there is a bounded set of possibilities for $X''$, and $n-m$ is bounded.  Proceeding as before, the number of possibilities for $Y'$ given $X'$ is bounded independent of $r$, so the total number of possibilities for $(X',Y')$ and therefore for $(X,Y)$ is bounded.

Finally, we consider the case that $x\in X$ is a $d$-hook and $y'\in Y$ is a $d'$-cohook, where
$x-d = y'-d' = z\in \Z_{\geq 0}\smallsetminus X$.  Removing the $d'$-cohook from $(X,Y)$, we obtain
$(X',Y')$, so as in the previous case, $z=0$.
Writing $X' = \schS^m(X'')$ as before, again $n-m$ is bounded, so the number of possibilities for $X$ and  for $Y$ given $X$ is bounded.
\end{proof}

\begin{prop}
\label{bounded}
Let $k$ and $k'$ be fixed integers.
Let $$T = T_{d_1,\ldots,d_p}^{\epsilon_1,\ldots,\epsilon_p},~~T' = T_{d'_1,\ldots,d'_{p'}}^{\epsilon'_1,\ldots,\epsilon'_{p'}},$$
with $\e_i,\e'_i = \pm 1$, be a pair of weakly orthogonal maximal tori of a classical group of Lie type $G$ of rank $r$ defined over $\F_q$, and let $t,t'\in G$ regular elements of $T,T'$ respectively. Suppose that
$$r-d_1 = k,~~r-d'_1=k',~~(\epsilon_1,k) \neq (\epsilon'_1,k').$$
Then the number of irreducible characters $\chi$ of $G$ for which $\chi(t)\chi(t')\neq 0$ is bounded by a constant depending only
on $k$ and $k'$. Next, if $G$ is of type $D_n$ assume in addition that
\begin{enumerate}[\rm(a)]
\item either at least one of $\{\e_1, \ldots,\e_p\}$ is
$-1$ or at least one of $\{d_1, \ldots,d_p\}$ is odd, and
\item either at least one of $\{\e'_1, \ldots,\e'_{p'}\}$ is
$-1$ or at least one of $\{d'_1, \ldots, d'_{p'}\}$ is odd.
\end{enumerate}
\noindent
Then the values $|\chi(t)\chi(t')|$ are also bounded independently of anything but $k$ and $k'$.
\end{prop}

\begin{proof}
As $T$ and $T'$ are weakly orthogonal, by \cite[Proposition 2.2.2]{LST1} we need only consider unipotent characters $\chi$.  Any such character is associated with an equivalence class of symbols of rank $r$.
Let $(X,Y)$ represent such a class.
By \cite[Theorem~3.3]{LM}, the values $\chi(t)$ and $\chi(t')$ are independent of the choices of $t$ and $t'$; moreover $\chi(t) = 0$ unless $(X,Y)$ has a $d_1$-hook assuming $\e_1=1$, respectively a $d_1$-cohook assuming $\epsilon_1=-1$.  Similarly $\chi(t')=0$ unless $(X,Y)$ has a $d'_1$-hook assuming $\e'_1=1$, respecitvely a $d'_1$-cohook assuming $\epsilon'_1=-1$.
By Proposition~\ref{Hooks and cohooks}, the number of possibilities for $(X,Y)$ is bounded by $B=B(k,k')$; in particular,
the number of possibilities for $\chi$ is bounded by $2B$.
Removing a $d_1$-hook or cohook or a $d'_1$-hook or cohook from a bounded set of $(X,Y)$,
the set of possible resulting symbol classes is also bounded independently of $r$, and likewise for the number of possible removals.
Hence, \cite[Theorem 3.3]{LM} implies that the character values $\chi(t)$ and $\chi(t')$ also belong to finite sets independent of $r$, if
none of $(X,Y)$ is degenerate. In the case some $(X,Y)$ is degenerate, which can happen only when $G$ is of type $D_n$, then our extra
assumption ensures that both $t$ and $t'$ are non-degenerate. As mentioned in \cite[\S3.4]{LM}, the two unipotent characters corresponding
to a degenerate symbol take the same values at non-degenerate regular semisimple elements, and their sum is still governed by
\cite[Theorem 3.3]{LM}, whence our statement follows in this case as well.
\end{proof}

\subsection{Groups of type $D_n$ and $\tw2 D_n$}

\begin{lem}\label{spin-tori}
Let $q$ be an odd prime power and let $G = \Omega^\e_{2n}(q)$ with $n \geq 4$ and $\e = \pm$. Let
$T < \SO^\alpha_{2a}(q) \times \SO^\beta_{2b}(q)$ be a maximal torus of
type $T^{\alpha,\beta}_{a,b}$ in $G$ with $1 \leq a < b$. Then we can find a regular semisimple element $g = \diag(u,v) \in T$
with $u \in \SO^\alpha_{2a}(q)$ having order $q^a -\alpha$ and $v \in \SO^\beta_{2b}(q)$ having order $q^b -\beta$.
\end{lem}

\begin{proof}
First we consider the maximal torus $T^\alpha_a = \langle x \rangle \cong C_{q^a-\alpha}$ in $\SO^\alpha_{2a}(q)$. If $\alpha = +$, or
if $\alpha = -$ but $2 \nmid a$, then, as shown in \cite[Lemma 8.14]{TZ2}, $T^\alpha_a \cap \Omega^\alpha_{2a}(q) = \langle x^2 \rangle$.
On the other hand, if $\alpha=-$ and $2|a$, then as $1=(-1)^{a(q-1)/2}$, by
\cite[Proposition 2.5.13]{KL} we have
$\SO^\alpha_{2a}(q) = \langle z \rangle \times \Omega^\alpha_{2a}(q)$ for a central involution $z$ which is contained in $T^\alpha_a$.
Since $C_{q^a-\alpha} \cong C_{(q^a-\alpha)/2} \times C_2$ with $2 \nmid (q^a-\alpha)/2$, we again see that
$T^\alpha_a \cap \Omega^\alpha_{2a}(q) = \langle x^2 \rangle$.

Let $T^\beta_b = \langle y \rangle \cong C_{q^b-\beta}$. By the above, $x^2,y^2 \in G$, but
$x \in \SO^\alpha_{2a}(q) \smallsetminus \Omega^\alpha_{2a}(q)$ and $y \in \SO^\beta_{2b}(q) \smallsetminus \Omega^\beta_{2b}(q)$.
We can now choose $g = xy$. As $q \geq 3$ and $a<b$, $g$ has simple spectrum acting on the natural module
$V=\F_q^{2n}$ of $G$ and so it is regular,
unless $(q,\alpha,a) = (3,+,1)$. But even in this exceptional case, ${\mathbf C}_{\SO(V \otimes \overline{\F}_q)}(g)^\circ$ is still
a torus of type $T^{+,\beta}_{1,n-1}$ and so $g$ is again regular.
\end{proof}

\begin{prop}\label{D-product1}
Let $G = \Spin^\e_{2n}(q)$ with $n \geq 4$ and $\e=\pm$. Then the following statements hold.
\begin{enumerate}[\rm(i)]
\item If $2|n$ and $\e=-$, then the pair of maximal tori $T^-_n$ and $T^{+,-}_{n-1,1}$ is weakly orthogonal.
\item If $a \in \N$ and $n \geq 2a+2$, then the pair of maximal tori $T^{-,-\e}_{n-a,a}$ and $T^{-,-\e}_{n-a-1,a+1}$ is weakly orthogonal.
\end{enumerate}
\end{prop}

\begin{proof}
We follow the proof of \cite[Proposition 2.6.1]{LST1}.
In this case, the dual group $G^{*}$ is $\PCO(V)^{\circ}$, where
$V = \F_{q}^{2n}$ is endowed with a quadratic form $Q$ of type $\e$ and $G^* = H/\ZB(H)$ with
$H := \mathrm{CO}(V)^{\circ}$. Consider the
complete inverse images in $H$ of the tori dual to the given two tori, and assume $g$ is an element
belonging to both of them. We need to show that $g \in \ZB(H)$.
We will consider the spectrum $S$ of the semisimple element $g$ on $V$ as a
multiset. Let $\gamma \in \F_{q}^{\times}$ be the {\it conformal coefficient} of
$g$, i.e. $Q(g(v)) = \gamma Q(v)$ for all $v \in V$.

\smallskip
In the case of (i), $S$ can be represented
as the joins of multisets $X$ and $Z \sqcup T$, where
$$\begin{array}{l}
  X := \{x,x^{q}, \ldots,x^{q^{n-1}},\gam x^{-1},\gam x^{-q}, \ldots,
             \gam x^{-q^{n-1}}\},\\
  Z := \{z,z^{q}, \ldots,z^{q^{n-2}},\gam z^{-1},\gam z^{-q}, \ldots,
             \gam z^{-q^{n-2}}\},~~
  T := \{t,\gam t^{-1}\},\end{array}$$
for some $x,z,t \in \bar{\F}_{q}^{\times}$ with
$x^{q^{n}+1} = \gam = t^{q+1}$ and $z^{q^{n-1}-1} = 1$. Since $|X| =2n > |Z|$, we may assume that
$x \in X \cap T$, whence $x^{q^n+1}=x^{q+1}=\gam$. As $2|n$, it follows that
$$x^{q^n-1} = (\gam^{q-1})^{(q^n-1)/(q^2-1)} = 1,$$
whence $\gam=x^2$. In turn, this implies that $x^{q+1}=x^2$, i.e. $x \in \F_q^\times$.
Since we now have $S = X = \{\underbrace{x,x, \ldots,x}_{2n}\}$, $g \in \ZB(H)$.

\smallskip
In the case of (ii), $S$ can be represented
as the joins $X \sqcup Y$ and $Z \sqcup T$, where
$$\begin{array}{l}
  X := \{x,x^{q}, \ldots,x^{q^{n-a-1}},\gam x^{-1},\gam x^{-q}, \ldots,
             \gam x^{-q^{n-a-1}}\},\\
  Y := \{y,y^{q}, \ldots,y^{q^{a-1}},\gam y^{-1},\gam y^{-q},
             \ldots, \gam y^{-q^{a-1}}\},\\
  Z := \{z,z^{q}, \ldots,z^{q^{n-a-2}},\gam z^{-1},\gam z^{-q}, \ldots,
             \gam z^{-q^{n-a-2}}\},\\
  T := \{t,t^{q}, \ldots,t^{q^{a}},\gam t^{-1},\gam t^{-q},
             \ldots, \gam t^{-q^{a}}\},\end{array}$$
for some $x,y,z,t \in \bar{\F}_{q}^{\times}$ with
$x^{q^{n-a}+1} = \gam = z^{q^{n-a-1}+1}$, and $y^{q^{a}+\e} = \gam =t^{q^{a+1}+\e}$ if $\e=+$ and
$y^{q^{a}+\e} = 1=t^{q^{a+1}+\e}$ if $\e=-$. Since $|X| =2(n-a) > |T|=2(a+1)$, we may assume that
$x \in X \cap Z$, whence $x^{q^{n-a}+1}=x^{q^{n-a-1}+1}=\gam$. It follows that
$$x^{q^{n-a-1}(q-1)} = 1,$$
whence $x \in \F_q^\times$, $\gam=x^2$, and
$X = \{\underbrace{x,x, \ldots,x}_{2(n-a)}\}$, $Z = \{\underbrace{x,x, \ldots,x}_{2(n-a-1)}\}$.
This also implies that $x \in T$, whence $T = \{\underbrace{x,x, \ldots,x}_{2a+2}\}$ and $g \in \ZB(H)$.
\end{proof}

\begin{prop}\label{D-product2}
Let $G = \Spin^\e_{2n}(q)$ with $n \geq 4$ and $\e=\pm$. Then the following statements hold.
\begin{enumerate}[\rm(i)]
\item Suppose $2|n$ and $\e=-$. Then there exist regular semisimple elements $x \in T^-_n$ and $y \in T^{+,-}_{n-1,1}$ such that
$x^G \cdot y^G \supseteq G \smallsetminus \ZB(G)$.
\item Suppose $a \in \N$, $a \geq 3$, and $n \geq 2a+2$. Then there exist regular semisimple elements
$x \in T^{-,-\e}_{n-a,a}$, $y \in T^{-,-\e}_{n-a-1,a+1}$ and a constant $C = C(a)$, such that if $g \in G$ has $\supp(g) \geq C$ then
$g \in x^G \cdot y^G$.
\end{enumerate}
\end{prop}

\begin{proof}
(i) As $2|n \geq 4$, by \cite{Zs} we can find a primitive prime divisor $\ell_{2n}$ of $q^{2n}-1$ and a primitive prime divisor $\ell_{n-1}$ of $q^{n-1}-1$. It is straightforward to check that $T^-_n$ contains a regular semisimple element $x$ of order divisible by $\ell_{2n}$, and
likewise $T^{+,-}_{n-1,1}$ contains a regular semisimple element $y$ of order divisible by $\ell_{n-1}$ (with
the projection onto $T^-_1 \cong \SO^-_2(q)$ having order $q+1$, which is possible by Lemma \ref{spin-tori}).

Suppose $\chi \in \Irr(G)$ is such that $\chi(x)\chi(y) \neq 0$. By Proposition \ref{D-product1}(ii), the pair of tori in question is
weakly orthogonal, hence $\chi$ is unipotent, labeled by a minimal symbol
$$S = (X,Y), ~~X = (x_1 < x_2 < \ldots <x_k),~Y = (y_1 < y_2< \ldots <y_l).$$
Now, if the denominator
of the degree formula \eqref{deg11} is not divisible by $\ell_{2n}$ then $\chi$ has $\ell_{2n}$-defect $0$ and so $\chi(x) = 0$. Similarly,
if the denominator of \eqref{deg11} is not divisible by $\ell_{n-1}$ then $\chi$ has $\ell_{n-1}$-defect $0$ and $\chi(y) = 0$. Thus the denominator
in \eqref{deg11} is divisible by both $\ell_{2n}$ and $\ell_{n-1}$.

Observe that if $x_1=0$, then by \eqref{s-rank} and the minimality of $S$ we have
$$n \geq x_k + \sum^{k-1}_{i=1}(i-1) + \sum^l_{j=1}j- \frac{(k+l)(k+l-2)}{4} = x_k + \frac{(k-l-2)^2}{4},$$
and so $x_k \leq n$, with equality precisely when
\begin{equation}\label{deg21}
  X = (0,1, \ldots,k-2,n),~Y=(1,2, \ldots, l),~k=l+2.
\end{equation}
On the other hand, if $x_1 \geq 1$, then
$$n \geq x_k + \sum^{k-1}_{i=1}i + \sum^l_{j=1}(j-1)- \frac{(k+l)(k+l-2)}{4} = x_k + \frac{(k-l)^2}{4} \geq x_k+1,$$
and so $x_k \leq n-1$.
Thus we always have $x_i \leq n$, and similarly $y_j \leq n$. Hence, the condition that the denominator of \eqref{deg11} is divisible by
$\ell_{2n}$ implies that there is an $n$-cohook $n$, where we may assume that $n \in X$ and $0 \notin Y$; in particular, \eqref{deg21} holds.
Now, if $l=0$, then $k=2$ and $\chi=1_G$. Assume $l \geq 1$.
Since $2|n$, we must also have an $(n-1)$-hook $c$ with $0 \leq c-(n-1) \leq 1$. As $k \geq 3$, we have $0,1 \in X$ by \eqref{deg21}, so
$c \notin X$, i.e. $c \in Y$ and $c-(n-1) \notin Y$. But $1 \in Y$, so $c=n-1 \in Y$. Furthermore, $k-2 \leq n-1$, hence $k \leq n+1$ and $l \leq n-1$
by \eqref{deg21}. It follows that $l=n-1$ and so $\chi = \St$, the Steinberg character.

We have shown that $1_G$ and $\St$ are the only two characters $\Irr(G)$ that are nonzero at both $x$ and $y$. Now, if $g \in G$ is semisimple, then $g \in x^G \cdot y^G$ by \cite[Lemma 5.1]{GT1}. If $g$ is not semisimple, then $\St(g)=0$, whence
$$\sum_{\chi \in \Irr(G)}\frac{\chi(x)\chi(y)\overline\chi(g)}{\chi(1)} = 1,$$
and so $g \in x^G \cdot y^G$ as well.

\smallskip
(ii) The assumption $a \geq 3$ ensures that regular semisimple elements $x \in T^{-,-\e}_{n-a,a}$ and $y \in T^{-,-\e}_{n-a-1,a+1}$ exist.
Suppose $\chi \in \Irr(G)$ is such that $\chi(x)\chi(y) \neq 0$. By Proposition \ref{D-product1}(ii), the pair of tori in question is
weakly orthogonal. Hence, by Proposition \ref{bounded}, the number of such characters $\chi$ is at most $C_1=C_1(a)$, and
for any such character, $|\chi(x)\chi(y)| \leq C_2$ for some $C_2 =C_2(a)$. Now choosing $C = \bigl(481\log_2(C_1C_2)\bigr)^2$, for
any $g \in G$ with $\supp(g) \geq C$ we have by \cite[Theorem 1.2.1]{LST1} that
$$\biggl{|} \sum_{\chi \in \Irr(G)}\frac{\chi(x)\chi(y)\overline\chi(g)}{\chi(1)} \biggr{|} > 1-\frac{C_1C_2}{q^{\sqrt{C}/481}} \geq 0,$$
and so $g \in x^G \cdot y^G$.
\end{proof}

\subsection{Groups of type $B_n$}
We will need a slight generalization of the notion of weakly orthogonal tori \cite{MSW}, \cite[Definition 2.2.1]{LST1}:

\begin{defn}
{\em We say that two $\F$-rational maximal tori $T$ and $T'$ in a connected reductive group $G/\F$ are
\emph{centrally orthogonal} if
$$T^*(\F) \cap {T'}^*(\F) = \ZB(G^*(\F))$$
for every choice of dual tori $T^*$ and ${T'}^*$ in the dual group $G^*$. This depends only on types of $T$ and ${T'}$.}
\end{defn}

The following is an analogue of \cite[Proposition 2.2.2]{LST1}:

\begin{prop}
\label{disjoint}
Let $T$ and ${T'}$ be centrally orthogonal maximai tori in a connected reductive group $G(\F)$, and let
$t \in T$ and $t' \in {T'}$ be regular semisimple elements of $G(\F)$. If
$\chi$ is an irreducible character of $G(\F)$ such that $\chi(t)\chi(t')\neq 0$, then there is a (degree $1$) character
$\alpha \in \Irr(G(\F))$ such that $\chi\alpha$ is unipotent.
\end{prop}

\begin{proof}
By \cite[5.1]{MM}, if $s\in G(\F)$ is semisimple, and $\chi(s)\neq 0$, then there exist
a maximal torus $T$ and a character $\theta \in \Irr(T(\F))$ such that $R_{T,\theta}(s)\neq 0$, and $\theta^*$ belongs to
the conjugacy class $C_\chi$.  By \cite[7.2]{DL}, this implies that $s$ lies
in the $G(\F)$-conjugacy class of some element of $T(\F)$.
If $\chi(t)\chi(t') \neq 0$, then there exist $G^*(\F)$-conjugate elements $\theta_1^*$
and $\theta_2^*$ belonging to tori $T^*$ and ${T'}^*$ which are dual to tori
$T$ and ${T'}$ containing $t$ and $t'$ respectively.  As $T^*$ and ${T'}^*$
intersect in $\ZB(G^*(\F))$, this means $\theta_1^*,\theta_2^*\in \ZB(G^*(\F))$, and the statement follows from
\cite[Proposition 13.30]{DM}.
\end{proof}

\begin{prop}\label{B-product1}
The following statements hold for $G = \SO_{2n+1}(q)$ with $n \geq 3$.
\begin{enumerate}[\rm(i)]
\item Define $\kappa:=(-1)^n$. Then the pair of maximal tori $T^{-\kappa}_n$ and $T^{\kappa,-}_{n-1,1}$ is weakly orthogonal when
$2|q$ and centrally orthogonal if $2 \nmid q$.
\item If $2 \nmid n \geq 5$, then the pair of maximal tori $T^{-}_{n}$ and $T^{+,-}_{n-2,2}$ is weakly orthogonal when
$2|q$ and centrally orthogonal if $2 \nmid q$.
\item If $2|n \geq 8$, then the pair of maximal tori $T^{-,-}_{n-2,2}$ and $T^{+,+}_{n-3,3}$ is weakly orthogonal when
$2|q$ and centrally orthogonal if $2 \nmid q$.
\end{enumerate}
\end{prop}

\begin{proof}
 In this case, the dual group $G^{*}$ is $\Sp(V)$, where
$V = \F_{q}^{2n}$ is endowed with a symplectic form. Consider any $g$ in the intersection of
dual tori, and let $S$ denote the spectrum of $g$ on $V$ as a
multiset.

\smallskip
In the case of (i), $S$ can be represented
as the joins of multisets $X$ and $Z \sqcup T$, where
$$\begin{array}{l}
  X := \{x,x^{q}, \ldots,x^{q^{n-1}},x^{-1},x^{-q}, \ldots,
             x^{-q^{n-1}}\},\\
  Z := \{z,z^{q}, \ldots,z^{q^{n-2}},z^{-1},z^{-q}, \ldots,
             z^{-q^{n-2}}\},~~
  T := \{t,t^{-1}\},\end{array}$$
for some $x,z,t \in \bar{\F}_{q}^{\times}$ with
$x^{q^{n}+\kappa} = z^{q^{n-1}-\kappa} = t^{q+1}=1$. Since $|X| =2n > |Z|$, we may assume that
$x \in X \cap T$, whence $x^{q^n+\kappa}=x^{q+1}=1$. As $(q+1)|(q^n-\kappa)$, it follows that
$x^ 2 = 1= x^{q-1}$, i.e. $x \in \F_q^\times$.
Since we now have $S = X = \{\underbrace{x,x, \ldots,x}_{2n}\}$, $g \in \ZB(G^*)$.

\smallskip
In the case of (ii), $S$ can be represented
as the multisets $X \sqcup Y$ and $Z \sqcup T$, where
$$\begin{array}{l}
  X := \{x,x^{q}, \ldots,x^{q^{n-1}},x^{-1},x^{-q}, \ldots,
             x^{-q^{n-1}}\},\\
  Z := \{z,z^{q}, \ldots,z^{q^{n-3}},\gam z^{-1},z^{-q}, \ldots,
             \gam z^{-q^{n-3}}\},~
  T := \{t,t^{q},t^{-1},t^{-q}\},\end{array}$$
for some $x,z,t \in \bar{\F}_{q}^{\times}$ with
$x^{q^{n}+1} = z^{q^{n-2}-1}=t^{q^2+1}=1$. Since $|X| =2n > |Z|$, we may assume that
$x \in X \cap T$, whence $x^{q^{n}+1}=x^{q^{2}+1}=1$. As $2\nmid n$, it follows that
$x^{q+1} = 1=x^2$,
whence $x \in \F_q^\times$, $X = \{\underbrace{x,x, \ldots,x}_{2n}\}$, and $g \in \ZB(G^*)$.

\smallskip
In the case of (iii), $S$ can be represented
as the joins $X \sqcup Y$ and $Z \sqcup T$, where
$$\begin{array}{l}
  X := \{x,x^{q}, \ldots,x^{q^{n-3}},x^{-1},x^{-q}, \ldots,
             x^{-q^{n-3}}\},~
  Y := \{y,y^{q},y^{-1},y^{-q}\},\\
  Z := \{z,z^{q}, \ldots,z^{q^{n-4}},z^{-1},z^{-q}, \ldots,
             z^{-q^{n-4}}\},~
  T := \{t,t^{q},t^{q^{2}},t^{-1}, t^{-q},t^{-q^{2}}\},\end{array}$$
for some $x,y,z,t \in \bar{\F}_{q}^{\times}$ with
$x^{q^{n-2}+1} = y^{q^{2}+1} = z^{q^{n-3}-1}=t^{q^{3}-1}=1$. Since $|X| =2n-4 > |T|=6$, we may assume that
$x \in X \cap Z$, whence $x^{q^{n-2}+1}=x^{q^{n-3}-1}=1$. As $2|n$, it follows that
$x^{q+1} = 1=x^2$,
whence $x \in \F_q^\times$, and
$X = \{\underbrace{x,x, \ldots,x}_{2n-4}\}$, $Z = \{\underbrace{x,x, \ldots,x}_{2n-6}\}$.
This also implies that $x \in T$, whence $T = \{x,x,x,x,x,x\}$ and $g \in \ZB(G^*)$.
\end{proof}

In what follows, we note that for $n \geq 3$ that if $2|q$ then $\SO_{2n+1}(q) \cong \Sp_{2n}(q)$ is simple, whereas if $2 \nmid q$ then
$[G,G] = \Omega_{2n+1}(q)$ is simple and has index $2$ in $G = \SO_{2n+1}(q)$; let $\sgn$ denote the linear character of order
$2$ of $G$ in the latter case.

\begin{prop}\label{B-product2}
There is an explicit constant $C \in \N$ such that the following statements hold for $G = \SO_{2n+1}(q)$ with $2|n \geq C$.
There exist regular semisimple elements $x \in T^-_n \cap [G,G]$, $y \in T^{+,-}_{n-1,1} \cap [G,G]$ such that
$x^G \cdot y^G = [G,G] \smallsetminus \{e\}$. If $2 \nmid q$, then there is a regular semisimple element
$y'  \in T^{+,-}_{n-1,1} \smallsetminus [G,G]$ such that
$x^G \cdot (y')^G = G \smallsetminus [G,G]$.
\end{prop}

\begin{proof}
(i) As $2|n \geq 4$, by \cite{Zs} we can find a primitive prime divisor $\ell_{2n}$ of $q^{2n}-1$ and a primitive prime divisor $\ell_{n-1}$ of $q^{n-1}-1$. It is straightforward to check that $T^-_n$ contains a regular semisimple element $x \in [G,G]$ of order $\ell_{2n}$, and
likewise $T^{+,-}_{n-1,1}$ contains a regular semisimple element $y \in [G,G] \cap \Omega^-_{2n}(q)$ of order divisible by $\ell_{n-1}$ (with
the projection onto $T^-_1 \cong \SO^-_2(q)$ having order $q+1$, which is possible by Lemma \ref{spin-tori}). If $2 \nmid q$, then by
changing $y$ to have the first projection onto $\SO^+_{2n-2}(q)$ of order $\ell_{n-1}$, we obtain a regular semisimple element
$y' \in T^{+,-}_{n-1,1} \smallsetminus [G,G]$.

\smallskip
(ii) Suppose $\chi \in \Irr(G)$ is such that $\chi(x)\chi(y) \neq 0$ or $\chi(x)\chi(y') \neq 0$ if $2 \nmid q$. By Proposition \ref{B-product1}(i), the
pair of tori in question is centrally orthogonal, hence either $\chi$ or $\chi\cdot\sgn$ is unipotent by Proposition \ref{disjoint}.
Without loss we may assume that $\chi$ is unipotent, labeled by a minimal symbol
$$S = (X,Y), ~~X = (x_1 < x_2 < \ldots <x_k),~Y = (y_1 < y_2< \ldots <y_l),$$
where $k,l \in \Z_{\geq 0}$ and $2 \nmid (k-l)$.
Now, if the denominator
of the degree formula \eqref{deg11} is not divisible by $\ell_{2n}$ then $\chi$ has $\ell_{2n}$-defect $0$ and so $\chi(x) = 0$. Similarly,
if the denominator of \eqref{deg11} is not divisible by $\ell_{n-1}$ then $\chi$ has $\ell_{n-1}$-defect $0$ and $\chi(y) = 0$, as
well as $\chi(y')=0$ when $2 \nmid q$. Thus the denominator
in \eqref{deg11} is divisible by both $\ell_{2n}$ and $\ell_{n-1}$.

Observe that if $x_1=0$, then by \eqref{s-rank} and the minimality of $S$ we have
$$n \geq x_k + \sum^{k-1}_{i=1}(i-1) + \sum^l_{j=1}j- \frac{(k+l-1)^2}{4} = x_k + \frac{(k-l-1)(k-l-3)}{4},$$
and so $x_k \leq n$, with equality precisely when
\begin{equation}\label{deg31}
  X = (0,1, \ldots,k-2,n),~Y=(1,2, \ldots, l),~k-l = 1 \mbox{ or }3.
\end{equation}
On the other hand, if $x_1 \geq 1$, then
$$n \geq x_k + \sum^{k-1}_{i=1}i + \sum^l_{j=1}(j-1)- \frac{(k+l-1)^2}{4} = x_k + \frac{(k-l)^2-1}{4} \geq x_k,$$
and so $x_k \leq n$, with equality precisely when
\begin{equation}\label{deg31a}
  X = (1,2, \ldots,k-1,n),~Y=(0,1, \ldots, l-1),~k-l = \pm 1.
\end{equation}
Thus we always have $x_i \leq n$, and similarly $y_j \leq n$. Hence, the condition that the denominator of \eqref{deg11} is divisible by
$\ell_{2n}$ implies that there is an $n$-cohook $n$, whence we may assume that $n=x_k \in X$ and $0 \notin Y$.
This rules out the case $x_1 \geq 1$, whence \eqref{deg31} holds.
Now, if $k=1$, then $l=0$ and $\chi=1_G$. If $k=2$, then $l=1$, $S = \binom{0,n}{1}$, and $\chi(1)=(q^n-1)(q^n+q)/2(q-1)$; denote this
unipotent character by $\chi_1$.

Assume $k \geq 3$.  Since $2|n$, we must also have an $(n-1)$-hook $c$ with $0 \leq c-(n-1) \leq 1$. As $k \geq 3$, we have $0,1 \in X$ by \eqref{deg31}, so $c \notin X$, i.e. $c \in Y$ and $c-(n-1) \notin Y$. In particular, $l \geq 1$, hence $1 \in Y$ and $c=n-1 \in Y$. Furthermore,
$k-2 \leq n-1$, hence $k \leq n+1$ but $l \leq n-1$. By \eqref{deg31}, we have

$\bullet$ either $(k,l)=(n+1,n)$, $S = \binom{0,1, \ldots,n-1,n}{1,2, \ldots ,n}$, $\chi = \St$, the Steinberg character, or

$\bullet$ $(k,l) = (n,n-1)$, and $S = \binom{0,1, \ldots,n-2,n}{1,2, \ldots ,n-1}$; denote this unipotent character by $\chi_2$.

\smallskip
(iii) We have shown that, up to tensoring with $\sgn$ when $2 \nmid q$, $\chi_0=1_G$, $\St$, $\chi_1$, and $\chi_2$ are the only four characters in $\Irr(G)$ that are nonzero at both $x$ and $y$, respectively at $x$ and $y'$ when $2 \nmid q$. It is clear that
\begin{equation}\label{deg32}
  \chi_0(x)\chi_0(y) = \chi_0(x)\chi_0(y') = 1,~~|\St(x)\St(y)| = |\St(x)\St(y')| = 1.
\end{equation}
To bound $|\chi_1(x)\chi_1(y)|$ and $|\chi_1(x)\chi_1(y')|$, we follow the proof of \cite[Proposition 3.4.1]{LST1} that relies on
the main result of \cite{Odd-O}. Recall that $\chi_1$ is labeled by $S = \binom{X}{Y}=\binom{0,n}{1}$. Let $Z_{1} = \{0,1,n\}$ be the
set of ``singles'' and $Z_{2} = X \cap Y = \emptyset$. Then
the family ${\mathcal F}(\chi_1)$
consists of all irreducible characters $\psi_{S'}$ of the Weyl group ${\mathsf {W}}_{n}$ labeled by
symbols $S' = \binom{X'}{Y'}$ of defect $1$ which contain the same entries
(with the same multiplicities) as $\Lambda$ does, cf. \cite[Cor. (5.9)]{Odd-O}. For the given $S=\binom{0,n}{1}$ (or in fact for all
symbols of odd defect with the same set $Z_1 = \{0,1,n\}$ of ``singles''), we have the following possibilities for $S'$ and
the corresponding pair $(\lambda',\mu')$ of (possibly empty) partitions:
$$\left\{ \begin{array}{ll}S'=\binom{1,n}{0}, & (\lambda',\mu') = \bigl((1,n-1),(\emptyset)\bigr),\\
     S'=\binom{0,n}{1}, & (\lambda',\mu') = \bigl((n-1),(1)\bigr),\\
     S'=\binom{0,1}{n}, & (\lambda',\mu') = \bigl((\emptyset),(n)\bigr).\\\end{array} \right.$$
Let $w, w' \in {\mathsf {W}}_n$ correspond to $x$, respectively to $y$ and $y'$. Recalling the construction of $\psi_{S'}$
\cite[(3.2.1)]{LST1}, we find that
$$\psi_{S'}(w) = -1,~0,~-1,~\psi_{S'}(w') = 0,~ -1,~-1,$$
respectively.
It follows from \cite[Cor. (5.9)]{Odd-O} that
\begin{equation}\label{deg33}
 |\chi_1(x)| \leq 1, ~|\chi_1(y)|=|\chi_1(y')| \leq 1.
\end{equation}

To bound the character values for $\chi_2$, we use the Alvis-Curtis
duality functor $D_G$ which sends any irreducible character of $G$ to an
irreducible character of $G$ up to a sign, cf. \cite[Corollary 8.15]{DM}.
Using Theorems 1.1 and 1.2 of \cite{N}, we see that $\chi_1$ is the unique unipotent characters of its degree, and so,
by inspecting \cite[Table 1]{ST}, $\chi_1$ is a constituent of the rank $3$ permutation action of $G$ on singular $1$-spaces of
its natural module; also, $\chi_1$ is irreducible over $[G,G]$.
Hence $\chi_1$ is also a constituent of the permutation character $1^G_B$, where $B$ is a Borel subgroup of $G$, and the same is true for $1_G$ and $\St$. For each
irreducible constituent $\varphi$ of $1^G_B$, there is a polynomial
$d_\varphi(X) \in \Q[X]$ in variable $X$ (the so-called {\it generic degree}, cf.
\cite[\S13.5]{Ca}, which depends only on the Weyl group of $G$ but not on $q$) such that $\varphi(1) = d_\varphi(q)$.
According to Theorem  (1.7) and Proposition (1.6) of \cite{Cur}, $D_G$ permutes
the irreducible constituents of $1^G_B$. Moreover, there is an integer $N$ such that
\begin{equation}\label{d-power}
  d_{D_G(\varphi)}(X) = X^N d_\varphi(X^{-1}).
\end{equation}
It is well known, see e.g. Corollary 8.14 and
Definition 9.1 of \cite{DM}, that $D_G$ interchanges $1_G$ and $\St$. Since
$\St(1) = q^{n^2}$, (\ref{d-power}) applied to $\varphi = 1_G$ yields that
$N = n^2$. Applying (\ref{d-power}) to $\varphi = \chi_1$, we now obtain that
\begin{equation}\label{deg34}
  D_G(\chi_1)(1) = q^{n^2-2n}\chi_1(1).
\end{equation}

Furthermore, in the case of a rational torus $T$,
$D_T(\lambda) = \lambda$ for all $\lambda \in \Irr(T)$, see
\cite[Definition 8.8]{DM}. Applying this and \cite[Corollary 8.16]{DM} to
$T = \mathbf{C}_G(x)$, we now see that
$$D_G(\chi)(x) = \pm (D_T \circ \mathrm{Res}^G_T)(\chi)(x)
  = \pm \chi(x).$$
Similarly, $D_G(\chi)(y) = \pm \chi_1(y)$ and $D_G(\chi)(y') = \pm \chi(y')$.
It follows that, if $\chi_2$ is nonzero at both $x,y$ (respectively at $x,y'$), then so is
$D_G(\chi_2)$. It follows that either $\chi_2(x)\chi_2(y)\neq 0$, in which case $\chi_2=D_G(\chi_1)$ and \eqref{deg33} yields
\begin{equation}\label{deg35}
 |\chi_2(x)\chi_2(y)|=|\chi_2(x)\chi_2(y')| \leq 1,
\end{equation}
or $\chi_2(x)\chi_2(y)=0$, in which case \eqref{deg35} is automatic.

\smallskip
(iv) Now, if $g \in [G,G]$ is semisimple, then $g \in x^G \cdot y^G$ by \cite[Lemma 5.1]{GT1}. Suppose $g \in [G,G]$ is not semisimple.
Then $\St(g)=0$. If $2 \nmid q$, then $\sgn(g)=\sgn(x)=\sgn(y) = 1$. This shows that $\chi$ and $\chi\cdot\sgn$ take the same values
at $x$, $y$, and $g$. Since the index of any proper subgroup in $[G,G]$ is
$> q^{2n-1}$ (see \cite[\S9]{TZ1}), it follows that $|\chi(g)| \leq |G|^{1/2}q^{1/2-n}$, and so, choosing $n$ large enough, we obtain from
\eqref{deg34} and \eqref{deg35} that
$$\frac{|\chi_2(x)\chi_2(y)\chi_2(g)|}{\chi_2(1)} < 0.01.$$
Using Gluck's bound $|\psi_1(g)|/\psi(1) \leq 0.95$ for any $\psi \in \Irr([G,G])$, we obtain
$$\frac{1}{\gcd(2,q-1)}\biggl{|}\sum_{\chi \in \Irr(G)}\frac{\chi(x)\chi(y)\overline\chi(g)}{\chi(1)} \biggr{|} > 1-0.95-0.01 = 0.04,$$
and so $g \in x^G \cdot y^G$.

Finally, consider the case $2 \nmid q$ and $g \in G \smallsetminus [G,G]$. Then $\sgn(x)=1$ and $\sgn(g)=\sgn(y') = -1$. Again by choosing
$n$ large enough we obtain from \eqref{deg34} and \eqref{deg35} that
$$\frac{|\chi(x)\chi(y')\chi(g)|}{\chi(1)} < 0.001$$
for $\chi = \chi_2$, $\chi_2\cdot\sgn$, $\St$, $\St\cdot\sgn$.
Next, \cite[Lemma 2.19]{GT-12} together with Gluck's bound imply that
$$|\psi(g)|/\psi(1) \leq (3 +0.95)/4 = 0.9875$$
for any $\psi \in \Irr(G)$ that is irreducible over $[G,G]$. Hence,
$$\frac{1}{2}\biggl{|}\sum_{\chi \in \Irr(G)}\frac{\chi(x)\chi(y')\overline\chi(g)}{\chi(1)} \biggr{|} > 1-0.9875-0.002 > 0.01,$$
and so $g \in x^G \cdot (y')^G$, as stated.
\end{proof}

\begin{prop}\label{B-product3}
There is an explicit constant $C \geq 5$ such that the following statements hold for $G = \SO_{2n+1}(q)$ with $2 \nmid n \geq C$.
There exist regular semisimple elements $x \in T^+_n \cap [G,G]$, $y \in T^{-,-}_{n-1,1} \cap [G,G]$ such that
$x^G \cdot y^G = [G,G] \smallsetminus \{e\}$. If $2 \nmid q$, then there is a regular semisimple element
$y'  \in T^{-,-}_{n-1,1} \smallsetminus [G,G]$ such that
$x^G \cdot (y')^G = G \smallsetminus [G,G]$.
\end{prop}

\begin{proof}
(i) As $2 \nmid n \geq 5$, by \cite{Zs} we can find a primitive prime divisor $\ell_{2n-2}$ of $q^{2n-2}-1$ and a primitive prime divisor $\ell_{n}$ of $q^{n}-1$. It is straightforward to check that $T^+_n$ contains a regular semisimple element $x \in [G,G]$ of order $\ell_{n}$, and
likewise $T^{-,-}_{n-1,1}$ contains a regular semisimple element $y \in [G,G] \cap \Omega^+_{2n}(q)$ of order divisible by $\ell_{2n-2}$ (with
the projection onto $T^-_1 \cong \SO^-_2(q)$ having order $q+1$, which is possible by Lemma \ref{spin-tori}). If $2 \nmid q$, then by
changing $y$ to have the first projection onto $\SO^-_{2n-2}(q)$ of order $\ell_{2n-2}$, we obtain a regular semisimple element
$y' \in T^{-,-}_{n-1,1} \smallsetminus [G,G]$.

\smallskip
(ii) Suppose $\chi \in \Irr(G)$ is such that $\chi(x)\chi(y) \neq 0$ or $\chi(x)\chi(y') \neq 0$ if $2 \nmid q$. By Proposition \ref{B-product1}(i), the
pair of tori in question is centrally orthogonal, hence either $\chi$ or $\chi\cdot\sgn$ is unipotent by Proposition \ref{disjoint}.
Without loss we may assume that $\chi$ is unipotent, labeled by a minimal symbol
$$S = (X,Y), ~~X = (x_1 < x_2 < \ldots <x_k),~Y = (y_1 < y_2< \ldots <y_l),$$
where $k,l \in \Z_{\geq 0}$ and $2 \nmid (k-l)$.
Now, if the denominator
of the degree formula \eqref{deg11} is not divisible by $\ell_{n}$ then $\chi$ has $\ell_{n}$-defect $0$ and so $\chi(x) = 0$. Similarly,
if the denominator of \eqref{deg11} is not divisible by $\ell_{2n-2}$ then $\chi$ has $\ell_{2n-2}$-defect $0$ and $\chi(y) = 0$, as
well as $\chi(y')=0$ when $2 \nmid q$. Thus the denominator
in \eqref{deg11} is divisible by both $\ell_{n}$ and $\ell_{2n-2}$.

As mentioned in the proof of Proposition \ref{B-product2}, we always have that $x_i \leq n$ and $y_j \leq n$. Hence, the condition that the denominator of \eqref{deg11} is divisible by
$\ell_{n}$ implies that there is an $n$-hook $n$, whence we may assume that $n=x_k \in X$ and $0 \notin X$.
This implies $x_1 \geq 1$, whence \eqref{deg31a} holds and $k \geq 1$.
Now, if $l=0$, then $k=1$ and $\chi=1_G$. If $l=1$, then $k=2$, $S = \binom{1,n}{0}$, and $\chi(1)=(q^n+1)(q^n-q)/2(q-1)$; denote this
unipotent character by $\chi_1$.

Assume $l \geq 2$.  Since $2 \nmid n$, we must also have an $(n-1)$-cohook $c$ with $0 \leq c-(n-1) \leq 1$. Here,
$0,1 \in Y$ by \eqref{deg31a}, so $c \notin X$, i.e. $c \in Y$ and $c-(n-1) \notin X$. Also by \eqref{deg31a}, $l-1 \geq c$ and so $l \geq n$.
Hence $k \geq l-1 > 2$, whence $1 \in X$, implying $c-(n-1)=0$, and $c=n-1 \in Y$. Furthermore,
$k-1 \leq n-1$, hence $k \leq n$, and thus $l \leq n+1$. By \eqref{deg31}, we have

$\bullet$ either $(k,l)=(n,n+1)$, $S = \binom{1,2, \ldots,n-1,n}{0,1, \ldots ,n}$, $\chi = \St$, the Steinberg character, or

$\bullet$ $(k,l) = (n-1,n)$, and $S = \binom{1,2, \ldots,n-2,n}{0,1, \ldots ,n-1}$; denote this unipotent character by $\chi_2$.

\smallskip
(iii) We have shown that, up to tensoring with $\sgn$ when $2 \nmid q$, $\chi_0=1_G$, $\St$,  $\chi_1$, and $\chi_2$ are the only four characters of $\Irr(G)$ that are nonzero at both $x$ and $y$, respectively at $x$ and $y'$ when $2 \nmid q$. It is clear that \eqref{deg32} holds.
To bound $|\chi_1(x)\chi_1(y)|$, let $w, w' \in {\mathsf {W}}_n$ correspond to $x$, respectively to $y$ and $y'$. Repeating
the arguments in the proof of Proposition \ref{B-product2}, we come up with three possibilities for $S'$ and
$$\psi_{S'}(w) = -1,~0,~1,~\psi_{S'}(w') = 0,~ -1,~1,$$
respectively.
It follows from \cite[Cor. (5.9)]{Odd-O} that \eqref{deg33} holds in this case.

Using Theorems 1.1 and 1.2 of \cite{N}, we see that $\chi_1$ is the unique unipotent characters of its degree, and so,
by inspecting \cite[Table 1]{ST}, $\chi_1$ is a constituent of the rank $3$ permutation action of $G$ on singular $1$-spaces of
its natural module; also, $\chi_1$ is irreducible over $[G,G]$.
Hence $\chi_1$ is also a constituent of the permutation character $1^G_B$, where $B$ is a Borel subgroup of $G$.
Now, to bound the character values for $\chi_2$, we again follow the proof of Proposition \ref{B-product2} using the Alvis-Curtis
duality functor $D_G$. This shows that \eqref{deg35} holds in this case as well.
To finish the proof, we just repeat part (iv) of the proof of Proposition \ref{B-product2} verbatim.
\end{proof}

\begin{prop}\label{B-product4}
There exists an explicit constant $C > 0$ such that the following statements hold for $G = \Omega_{2n+1}(q)$ with $n \geq 8$.
If $2|n$, let $T=T^{-,-}_{n-2,2}$ and $T' = T^{+,+}_{n-3,3}$ be maximal tori in $H:=\SO_{2n+1}(q)$.
If $2 \nmid n$, let $T=T^{-}_{n}$ and $T' = T^{+,-}_{n-2,2}$ maximal tori in $\SO_{2n+1}(q)$.
Then there exist regular semisimple elements $x \in T \cap G$ and $y \in T' \cap G$, such that if $g \in G$ has $\supp(g) \geq C$ then
$g \in x^H \cdot y^H$.
\end{prop}

\begin{proof}
Using Lemma \ref{spin-tori}, we can see that regular semisimple elements $x \in T \cap G$ and $y \in T' \cap G$ exist.
Suppose $\chi \in \Irr(H)$ is such that $\chi(x)\chi(y) \neq 0$. By Proposition \ref{B-product1}(ii), (iii) the pair of tori in question is
weakly orthogonal when $2|q$ and centrally orthogonal when $2 \nmid q$. Hence, either $\chi$ is unipotent, or
$2 \nmid q$ and $\chi\cdot\sgn$ is unipotent. In the case $2 \nmid q$, note that $\sgn(x)=\sgn(y)=\sgn(g) = 1$ for all $g \in G$.
By Proposition \ref{bounded}, the number of such characters $\chi$ is at most $C_1$, and
for any such character, $|\chi(x)\chi(y)| \leq C_2$ for some $C_2$. Now choosing $C = \bigl(481\log_2(C_1C_2)\bigr)^2$, for
any $g \in G$ with $\supp(g) \geq C$ we have by \cite[Theorem 1.2]{LST1} that
$$\frac{1}{\gcd(2,q-1)}\biggl{|} \sum_{\chi \in \Irr(H)}\frac{\chi(x)\chi(y)\overline\chi(g)}{\chi(1)} \biggr{|} > 1-\frac{C_1C_2}{q^{\sqrt{C}/481}} \ge 0,$$
and so $g \in x^H \cdot y^H$.
\end{proof}

\bigskip

\section{Applications to derangements}
\subsection{The main results on derangements}
Our results from previous sections have applications to permutation groups. The study of fixed-point-free permutations,
also called {\it derangements}, was initiated about 300 years ago. Around 150 years ago Jordan proved that every finite transitive 
permutation group $G \le \SSS_n$ ($n > 1$) contains a derangement. In \cite{CC} it is shown that the proportion $\delta(G)$ of derangements 
in such a group $G$ is at least $1/n$.
It turns out that, if $G$ is simple, the proportion of derangements is bounded away from zero. Indeed, we have the following
theorem by Fulman and Guralnick (see \cite[1.1]{FG3} and the references therein).

\begin{thm}\label{FG}
There exists an absolute constant $\e > 0$ such that, if $G$ is a finite simple transitive permutation group, and $\cD = \cD(G) \subset G$
is the set of derangements in $G$, then
\[
|\cD| \ge \e |G|.
\]
\end{thm}

This confirms a conjecture of Boston and Shalev.

In fact it is shown in \cite{FG3} that $\e = 0.016$ will do provided $|G| \gg 0$.

In this section, we prove Theorem B, which we restate:

\begin{thm}\label{square}
Let $G$ be a finite simple transitive permutation group of sufficiently large order. Then every element of $G$ is a product of two derangements.
\end{thm}

Clearly, Theorem \ref{square} holds in the case $G$ is a cyclic group of prime order $\geq 3$.
Its proof for non-abelian simple groups will occupy the rest of the section.

We remark that  Theorem~\ref{FG} and Corollary~\ref{ABC} give an immediate proof of the easier three derangement result:

\begin{prop}
For all sufficiently large transitive simple permutation groups $G$, every permutation in $G$ is a product of
three derangements.
\end{prop}

\subsection{Some reductions}
We first prove some
preliminary results which reduce the proof of Theorem \ref{square} to the case $G$ is a simple group of Lie type of unbounded
rank (over fields of bounded size).

\medskip
Let $G$ be as above, and let $H < G$ be a point stabilizer. Recall that $\cD(G,H)$ denotes the set of derangements of $G$ in its
action on the left cosets of $H$, and that $\cD(G,H) = G \smallsetminus \cup_{g \in G} H^g$.
Thus, if $M < G$ is a maximal subgroup containing $H$, then $\cD(G,M) \subseteq \cD(G,H)$. Hence
$\cD(G,M)^2 = G$ implies $\cD(G,H)^2 = G$. This reduces Theorem~\ref{square} to the primitive case,
where $H$ is a maximal subgroup of $G$.

Since $\cD(G,H)$ is a normal subset of $G$ and $\cD(G,H) = \cD(G,H)^{-1}$, Theorem \ref{FG}
implies the following.

\begin{cor}\label{fpf} Let $\schX$ be a family of finite simple groups for which Question \ref{Weak} with $S=T$ has an affirmative answer.
Then Theorem \ref{square} holds for $G \in \schX$.
\end{cor}

Combining Corollary \ref{fpf} with Theorems \ref{S equals T} and \ref{Q12bounded} we obtain the following.

\begin{cor}\label{fpf2}
Theorem~\ref{square} holds for alternating groups and for finite simple groups of Lie type of bounded rank.
\end{cor}

In fact, we will show in the next section, see Theorem \ref{an-main}, that the conclusion of Theorem \ref{square} holds for {\bf all} (simple) alternating groups.

Since almost simple sporadic groups have bounded order, it remains to deal with classical groups of unbounded rank.
We use \cite[Theorem~1.7]{FG3} (extending \cite{Sh1}) which states the following:

\begin{thm}
Let $\tilde G$ be a classical group of rank $r$ acting faithfully on its natural module $V$. Let $\schY(\tilde G)$ denote the union of
all irreducible subgroups of $\tilde G$ (if $q$ is even and $\tilde G = \Sp_{2r}(\F_q)$, we exclude the subgroups $\GO^{\pm}_{2r}(q)$ from $X(G)$).
Then
\[
\frac{|\schY(\tilde G)|}{|\tilde G|} \to 0 \; {\rm as} \; r \to \infty.
\]
\end{thm}

\begin{cor}\label{irred} Theorem~\ref{square} holds for groups $G \in \Cl_n(q)$ when $n \gg 0$, provided the point-stabilizer $H$
is irreducible and not $\GO^{\pm}_{n}(q)$ when $G = \Sp_{n}(\F_q)$ with $2|q$.
\end{cor}

\begin{proof}
By the above theorem we have
$$|\schY(G)|/|G| < 1/2$$
for $n \gg 0$.
Since $\cup_{g \in G} H^g \subseteq \schY(G)$, it follows that $|\cD(G,H)|> |G|/2$ and therefore $\cD(G,H)^2 = G$.
\end{proof}

\begin{thm}\label{bd-dim}
There are absolute constants $C_1, C_2$ such that the following holds.
Let $G \in \Cl_n(q)$ be a finite simple classical primitive permutation group with point-stabilizer $H$. If $q$ is even, assume $(G,H) \ne (\Sp_{n}(\F_q), \GO^{\pm}_{n}(\F_q))$.
Suppose $n \ge C_1$ and the action is not a subspace action on subspaces of dimension $k \le C_2$. Then $G$ satisfies
Theorem \ref{square}.
\end{thm}

\begin{proof} Relying on Corollary \ref{irred}, we may assume that $H$ is reducible, namely $G$ acts in subspace action,
say on subspaces (non-degenerate or totally singular for $G \ne \PSL_n(q)$) of dimension $k$, with $1 \le k \le n/2$.
Theorems 6.4, 9.4, 9.10, 9.17 and 9.30 of \cite{FG2} show that, as $k \to \infty$, the proportion of derangements
in $G$ tends to $1$. The result follows as before.
\end{proof}

\subsection{Completion of the proof of Theorem~\ref{square}}
The above reduction results allow us to assume that $G = \Cl(V)$ is a finite simple classical group defined over fields
of bounded size, and we need to establish Theorem \ref{square} for $G$ under the assumption that $\dim(V)$ is sufficiently large.
Let $\tilde G$ denote the central extension of $G$ for which $V$ is a faithful linear representation,
and let $\tilde H$ denote the inverse image in $\tilde G$ of a point stabilizer $H$ of $G$. Also let
$\Pi$ denote the transitive permutation representation with $H$ a point stabilizer.
We show that if $|G|$ is sufficiently large, equivalently, $\dim(V)$ is sufficiently large,
there exist elements $\tilde x,\tilde y\in \tilde G$ which are
derangements on $\tilde G/\tilde H$ and such that every element in $G\smallsetminus\{1\}$ is the
product of a conjugate of $x$ and a conjugate of $y$, where $x$ (resp. $y$) is the image of $\tilde x$ (resp. $\tilde y$) in $G$.  Since $x^{-1}$ is also a derangement, $e$ is also a product of two derangements.  We proceed by cases.

\subsubsection{The case $\tilde G = \SL_n(\F_q)$ with $n \geq 3$}
Here $\tilde H$ is the stabilizer of an $m$-dimensional subspace $V'$ of $V=\F_q^n$,
$1<m<n-1$.  Fixing an $\F_q$-basis of $\F_{q^n}$ we obtain an embedding of the norm-$1$ elements of $\F_{q^n}$
in $\SL_n(\F_q)$.  Let $\tilde x$ denote the image of a multiplicative generator of the group of norm-$1$ elements.  Let $\tilde y$ denote the image in $\SL_n(\F_q)> \GL_{n-1}(\F_q)$ of a
generator of $\F_{q^{n-1}}^\times$.
Thus $\tilde x$ and $\tilde y$ are regular elements of the tori $T=T_n$ and $T'=T_{n-1,1}$ of $\SL_n(\F_q)$ in \cite[Table 2.1]{MSW}.
As the characteristic polynomial of $\tilde x$ is irreducible over $\F_q$ and that of $\tilde y$ has an irreducible factor of degree $n-1$, it follows that
neither $\tilde x$ nor $\tilde y$ can fix an $\F_q$-subspace of $\F_q^n$ of dimension $m$, so $x$ and $y$ are indeed, derangements.  By \cite[Theorem~2.1]{MSW}, the product of the conjugacy classes of $x$ and $y$ covers all non-trivial elements of $G$.

Assume now that $m=1$ or $m=n-1$. Then we note that the elements $t$ and $t'$ constructed in Theorem \ref{A-product2}
are both derangements in $\Pi$, and so the statement follows from Theorem \ref{A-product2}.

\subsubsection{The case $\tilde G = \SU_n(\F_q)$ with $n \geq 5$}
Since $H$ is maximal, we have that $\tilde H$ is the stabilizer of an
$m$-dimensional subspace $V'$ of $V=\F_{q^2}^n$, $1 \leq m \leq n-1$, where $V'$ is either totally singular, or non-degenerate.
The existence of the Hermitian form allows us to assume that $1 \leq m \leq n/2$. Applying Theorem \ref{bd-dim} we may furthermore
assume that $m \leq c_2$ is bounded and that $m \leq n/2-1$.
Let $\tilde x$ and $\tilde y$ be elements of $\tilde G$ of order $\frac{q^n-(-1)^n}{q+1}$ and $q^{n-1}-(-1)^{n-1}$ respectively, so they are regular semisimple elements of tori $T=T_n$ and $T'=T_{n-1,1}$ respectively. Assume that $V'$ is not a non-degenerate $1$-space.
Then both $\tilde x$ and $\tilde y$ are derangements in $\Pi$. By \cite[Theorem~2.2]{MSW}, the product of the conjugacy classes of $x$ and $y$ covers all non-trivial elements of $G$, and the statement follows.

\smallskip
Suppose now that $V'$ is a non-degenerate $1$-space.
If $q>2$, then we again note that the elements $t$ and $t'$ constructed in Theorem
\ref{A-product2} are both derangements in $\Pi$, and so the statement follows from Theorem \ref{A-product2}. Assume now that
$q=2$. Consider the case $g \in \tilde G=\SU_n(2)$ is a transvection. Then we can put $g$ in a factor $A = \SU_4(2)$ of a standard subgroup
$$A \times B = \SU_4(2) \times \SU_{n-4}(2)$$
of $\tilde G$.
Direct calculation with \cite{GAP} shows that $g$ is a product $g = xy$ of two elements of order $5$ in $A$.
If $n$ is large, we choose $z \in B$ a regular semisimple element of type $T_{n-4}$, a maximal torus in $B$.
Now we note that $g = (xz)(yz^{-1})$, and both $xz$, $yz^{-1}$ are derangements.
We also note that any non-unipotent element of support $1$ in $\SU_n(2)$ is semisimple, hence by \cite[Lemma 5.1]{GT1}
it is a product of two regular semisimple elements of type $T_n$ which are derangements.
It remains to consider the case $\supp(g) \geq 2$, in which case the statement follows from Theorem \ref{A-product3}, since the elements $t$ and $t'$ constructed therein are derangements in $\Pi$.

\subsubsection{The case $\tilde G = \Omega_{2n+1}(q)$ or $\Sp_{2n}(q)$ with $n \geq 5$}
Let $\tilde x$ and $\tilde y$ be elements of order $q^n+1$ and $q^n-1$ generating tori of type $T=T^-_n$ and $T'=T^+_n$ respectively.  Thus the $\Frob_q$ orbit of any
eigenvalue of $\tilde x$ (resp. $\tilde y$) in the natural representation consists of a $2n$-cycle (resp. two $n$-cycles) together with an additional fixed point if $G$ is of type $B_n$. As in case (ii), we may assume that $\tilde H$ is the stabilizer of an $m$-dimensional
subspace $V'$ which is either totally singular, or non-degenerate, and has bounded dimension by Theorem \ref{bd-dim}.
For $C_n$, therefore, the theorem follows from \cite[Theorem~2.3]{MSW}, while for $B_n$ it holds by \cite[Theorem~2.4]{MSW} unless
$V'$ is a non-degenerate $1$-space. Likewise, we must still consider the cases $(\tilde G,\tilde H) = (\Sp_{2n}(q),\GO^{\pm}_{2n}(q))$
when $2|q$.

\smallskip
In both of the remaining actions, we can view $\tilde G = [\Gamma,\Gamma]$, where $\Gamma = \SO(V)$ and $V = \F_q^{2n+1}$ when
$2 \nmid q$, and $\Gamma = \Sp(V) \cong \SO_{2n+1}(q)$ and $V = \F_q^{2n}$ when $2|q$. Then $\Pi$ is the restriction to $\tilde G$ of
the transitive permutation action of $\Gamma$ with point stabilizer $\GO^{\e}_{2n}(q)$ for a fixed $\e = \pm$. First we consider the case
$\e1 = (-1)^n$. By Propositions \ref{B-product2} and \ref{B-product3}, if $n$ is large enough we can find in $\tilde G$
regular semisimple elements $x_1$ of type $T^{-\e}_n$ and
$y_1$ of type $T^{\e,-}_{n-1,1}$ such that $x_1^{\Gamma} \cdot y_1^\Gamma = \tilde G \smallsetminus \{e\}$. Since both $x_1$ and $y_1$ are derangements in $\Pi$, the statement follows in this case.

Assume now that $\e1 \neq (-1)^n$. By Proposition \ref{B-product4},
we can find in $\tilde G$ regular semisimple elements, $x_2$ of type $T^{-,-}_{n-2,2}$ and
$y_2$ of type $T^{+,+}_{n-3,3}$ when $2|n$, $x_2$ of type $T^-_{n}$ and
$y_2$ of type $T^{+,-}_{n-2,2}$ when $2 \nmid n$, such that $x_2^\Gamma \cdot y_2^\Gamma$ contains any element $g \in \tilde G$ of
large enough support, say $\supp(g) \geq B$. Since both $x_2$ and $y_2$ are derangements in $\Pi$, the statement again follows in this case.
Now we consider the case $\supp(g) < B < n-3$ and let $\lambda$ be the primary eigenvalue of $g$ on $V$,
cf. \cite[Proposition 4.1.2]{LST1}. By \cite[Lemma 6.3.4]{LST1}, we can decompose $V = U \perp W$ as an orthogonal sum of
$g$-invariant non-degenerate subspaces, with $\dim(U) =6$, $U$ has type $+$ if $2 \nmid q$, and $g|_U = \lambda \cdot 1_U$.
Define
$$\left\{\begin{array}{ll}I(W)=J(W) = \Sp(W) \cong \Sp_{2n-6}(q), & \mbox{when }2|q,\\
    J(W) = \Omega(W) \cong \Omega_{2n-5}(q),~I(W) = \SO(W) \cong \SO_{2n-5}(q), & \mbox{when }2 \nmid q.\end{array}\right.$$
Likewise, we define $J(U) = \Sp(U) \cong \Sp_6(q)$ when $2|q$, and $J(U) = \Omega(U) \cong \Omega^+_6(q)$ when $2 \nmid q$.
Since $\e1 = (-1)^{n-3}$, we can consider regular semisimple elements $x_3 \in T^{-\e}_{n-3} \cap J(W)$ and
$y_3 \in T^{\e,-}_{n-4,1} \cap J(W)$ constructed in
Propositions \ref{B-product2} and \ref{B-product3} for $J(W)$. If $2 \nmid q$, we will also consider the regular semisimple element
$y'_3 \in T^{\e,-}_{n-4,1} \smallsetminus J(W)$ constructed in Propositions \ref{B-product2} and \ref{B-product3} for
$I(W) \cong \SO_{2n-5}(q)$.
Also fix a regular semisimple element $z \in T^+_3$ of $J(U)$.

If $2|q$ or if $\lambda=1$, then we can write $g = \diag(1_U,h)$ with $h \in J(W)$.
By Propositions \ref{B-product2} and \ref{B-product3}, when $n$ is large enough
$h = x_3^u y_3^v$ for some $u,v \in I(W)$, whence $g = (zx_3)^u(z^{-1}y_3)^v$ is a product of two derangements.

Finally, assume that $2 \nmid q$ and $\lambda=-1$; write $g=\diag(-1_U,h)$ with $h \in I(W)$.
If $q \equiv 1 (\mod 4)$, then $1=(-1)^{3(q-1)/2}$, and so
$-1_U \in J(U) \cong \Omega^{+}_{6}(q)$ by \cite[Proposition 2.5.13]{KL}, whence $h \in J(W)$, and, as in the previous case,
$g = ((-1_U)zx_3)^u(z^{-1}y_3)^v$ is a product of two derangements. If $q \equiv 3 (\mod 4)$, then
$-1 = (-1)^{3(q-1)/2}$ and $-1_U \in I(U) \smallsetminus J(U)$. In this case, $h \in I(W) \smallsetminus J(W)$, and so by
Propositions \ref{B-product2} and \ref{B-product3} when $n$ is large enough we can write
$h = x_3^{u'}(y'_3)^{v'}$ for some $u',v' \in I(W)$. Now $g = ((-1_U)zx_3)^{u'}(z^{-1}y'_3)^{v'}$ is again a product of two derangements in $\Pi$.

\subsubsection{The case $\tilde G = \Omega^-_{2n}(q)$ with $n \geq 4$}
Here we choose, in accordance with Lemma \ref{spin-tori}, regular semisimple elements $\tilde x$ of type $T$ and $\tilde y$ of type
$T'$, where
$T=T^-_n$ is a maximal torus of order $q^n+1$, and $T'=T^{-,+}_{n-1,1}$ is a maximal torus of order $(q^{n-1}+1)(q-1)$. Then the characteristic polynomial of $\tilde x$ is irreducible, while that of $\tilde y$ factors into two linear factors and an irreducible factor of degree $2r-2$.
Again, $\tilde H$ is the stabilizer of an $m$-dimensional subspace $V'$, totally singular (with $m \leq n-1$ bounded by
Theorem \ref{bd-dim}), or non-degenerate.
Now \cite[Theorem~2.5]{MSW} implies the theorem, unless $\dim(V') = 1$ or $V'$ is a non-degenerate $2$-space of type $+$.

\smallskip
Consider the remaining three actions.
Assume first that $2|n$. Then note that the elements $x_1,y_1$ of types $T^-_n$ and $T^{+,-}_{n-1,1}$
constructed in the proof of Proposition \ref{D-product2}(i) are both derangements
in $\Pi$, whence the statement follows from Proposition \ref{D-product2}(i). Hence we may assume that $2 \nmid n \geq 13$. In
this case, note that the elements $x_2,y_2$ of types $T^{-,+}_{n-5,5}$ and $T^{-,+}_{n-6,6}$
constructed in the proof of Proposition \ref{D-product2}(ii) with $(a,\e) = (5,-)$ are both derangements
in $\Pi$. Hence, there exists some absolute constant $B$ such that if $\supp(g) \geq B$, then the statement follows from
Proposition \ref{D-product2}(ii). Now we consider the case $\supp(g) < B < n-3$ and let $\lambda$ be the primary eigenvalue of $g$ on $V$,
cf. \cite[Proposition 4.1.2]{LST1}. By \cite[Lemma 6.3.4]{LST1}, we can decompose $V = U \perp W$ as an orthogonal sum of
$g$-invariant subspaces, with $\dim(U) =6$, $U$ has type $+$, and $g|_U = \lambda \cdot 1_U$.
As $2|(n-3) \geq 10$, we can find regular semisimple elements $x_3 \in T^-_{n-3}$ and $y_3\in T^{-,+}_{n-4,1}$ constructed in the proof
of Proposition \ref{D-product2}(i) for $\Omega(W) \cong \Omega^-_{2n-6}(q)$. Also fix a regular semisimple element
$z \in T^+_3$ of $\Omega(U) \cong \Omega^+_6(q)$.
If $2|q$ or if $\lambda=1$, then we can write $g = \diag(1_U,h)$ with $h \in \Omega^-_{2n-6}(q)$. By Proposition \ref{D-product2}(i),
$h = x_3^u y_3^v$ for some $u,v \in \Omega(W)$, whence $g = (zx_3)^u(z^{-1}y_3)^v$ is a product of two derangements. Finally,
assume that $2 \nmid q$ and $\lambda=-1$. If $q \equiv 3 (\mod 4)$, then $-1=(-1)^{n(q-1)/2}$, and so
$-1_V \in \Omega^{-}_{2n}(q) = \tilde G$ by \cite[Proposition 2.5.13]{KL}, whence we can replace $g$ by $(-1_V)g$ and appeal to the
previous case. If $q \equiv 1 (\mod 4)$, then $1 = (-1)^{3(q-1)/2}$ and $-1_U \in \Omega(U) \cong \Omega^+_6(q)$. In this case, we can write
$g = \diag(-1_U,h)$ with $h \in \Omega^-_{2n-6}(q)$. Again by Proposition \ref{D-product2}(i),
$h = x_3^u y_3^v$ for some $u,v \in \Omega(W)$, whence $g = ((-1_U)zx_3)^u(z^{-1}y_3)^v$ is a product of two derangements in $\Pi$.

\subsubsection{The case $\tilde G = \Omega^+_{2n}(q)$ with $2 \nmid n \geq 5$}
We again choose regular semisimple elements $\tilde x$ and $\tilde y$ of type $T$ and $T'$, where the maximal tori $T=T^+_n$ and
$T'=T^{-,-}_{n-1,1}$ have order $q^n-1$ and $(q^{n-1}+1)(q+1)$, using Lemma \ref{spin-tori}.
Here, the characteristic polynomial of $\tilde x$ factors into two irreducibles of degree $n$ while the characteristic polynomial of $\tilde y$ factors into irreducibles of degree $2n-2$ and $2$. Now, Theorem \ref{bd-dim} and \cite[Theorem~2.6]{MSW} imply the theorem unless $\tilde H$ is the stabilizer of a
a non-degenerate $2$-space $V'$ of type $-$.
(Note that the case $V'$ is non-degenerate $1$-dimensional does not occur since we choose $\tilde y$ to have the second
irreducible factor of degree $2$ in its characteristic polynomial, cf. Lemma \ref{spin-tori}).

\smallskip
Consider the remaining action on non-degenerate $2$-spaces of type $-$, assuming $n \geq 9$.
Note that the elements $x_1,y_1$ of types $T^{-,-}_{n-3,3}$ and $T^{-,-}_{n-4,4}$ constructed in the proof of Proposition \ref{D-product2}(ii) with
$(a,\e) = (3,+)$ are both derangements in $\Pi$. Hence, there exists some absolute constant $B$ such that if $\supp(g) \geq B$, then the statement follows from Proposition \ref{D-product2}(ii). Now we consider the case $\supp(g) < B < n-3$ and let $\lambda$ be the primary eigenvalue of $g$ on $V$. Applying \cite[Lemma 6.3.4]{LST1}, we can decompose $V = U \perp W$ as an orthogonal sum of
$g$-invariant subspaces, with $\dim(U) =6$, $U$ has type $-$, and $g|_U = \lambda \cdot 1_U$.
As $2|(n-3) \geq 6$, we can find regular semisimple elements $x_2 \in T^-_{n-3}$ and $y_2\in T^{-,+}_{n-4,1}$ in
$\Omega(W) \cong \Omega^-_{2n-6}(q)$. Also fix a regular semisimple element
$z \in T^-_3$ of $\Omega(U) \cong \Omega^-_6(q)$.
If $2|q$ or if $\lambda=1$, then we can write $g = \diag(1_U,h)$ with $h \in \Omega^-_{2n-6}(q)$. By \cite[Theorem 2.5]{MSW},
$h = x_2^u y_2^v$ for some $u,v \in \Omega(W)$, whence $g = (zx_2)^u(z^{-1}y_2)^v$ is a product of two derangements. Finally,
assume that $2 \nmid q$ and $\lambda=-1$. If $q \equiv 1 (\mod 4)$, then $1=(-1)^{n(q-1)/2}$, and so
$-1_V \in \Omega^{+}_{2n}(q) = \tilde G$, whence we can replace $g$ by $(-1_V)g$ and return to the
previous case. If $q \equiv 3 (\mod 4)$, then $-1 = (-1)^{3(q-1)/2}$ and $-1_U \in \Omega(U) \cong \Omega^-_6(q)$. In this case, we can write
$g = \diag(-1_U,h)$ with $h \in \Omega^-_{2n-6)}(q)$. Again by \cite[Theorem 2.5]{MSW},
$h = x_3^u y_3^v$ for some $u,v \in \Omega(W)$, whence $g = ((-1_U)zx_3)^u(z^{-1}y_3)^v$ is a product of two derangements.

\subsubsection{The case $\tilde G = \Omega^+_{2n}(q)$ with $2|n \geq 6$}
Now we choose regular semisimple elements $\tilde x$ and $\tilde y$ of type $T$ and $T'$, where the maximal tori $T=T^{+,+}_{n-1,1}$
and $T'=T^{-,-}_{n-1,1}$ have order $(q^{n-1}-1)(q-1)$ and $(q^{n-1}+1)(q+1)$, again using Lemma \ref{spin-tori}. By \cite[Theorem 2.7]{GT2},
$\tilde x^{\tilde G} \cdot \tilde y^{\tilde G}$ contains all non-central elements of $\tilde G$. Hence the theorem follows, unless $\tilde H$ is the stabilizer of a subspace $V'$ of (bounded by Theorem \ref{bd-dim}) dimension $m$, and $V'$ is non-degenerate of dimension $m=1,2$ (with $m=1$ occurring only when $q \leq 3$), or totally singular of dimension $m = 1$.

\smallskip
If $V'$ is a non-degenerate $2$-space of type $-$, we then choose $\tilde y'$ regular semisimple of type $T'_2=T^{-,-}_{n-2,2}$, a maximal  torus of order $(q^{n-2}+1)(q^2+1)$ as in \cite[\S7.1]{LST1}. As $\tilde x$ and $\tilde y'$ are both derangements in $\Pi$,
the theorem now follows from \cite[\S7.2]{LST1} and \cite[Theorem 7.6]{GM}.

\smallskip
In the remaining cases,
note that, as shown in the proof of
\cite[Theorem 2.7]{MSW}, there is a regular semisimple elements $\tilde x'$ of type $T'_1$, a maximal torus of order
$(q^{n/2}+(-1)^{n/2})^2$, such that there are exactly three irreducible characters of $\tilde G$ that are nonzero at both $\tilde x'$
and $\tilde y$; namely $1_{\tilde G}$, $\St$, and one more character $\rho$: $|\St(\tilde x')\St(\tilde y)|=1$ and
$|\rho(\tilde x')\rho(\tilde y)| = 2$. The imposed condition on $V'$ ensures that $\tilde x'$ and $\tilde y$ are both derangements in $\Pi$.
Consider any $g \in \tilde G \smallsetminus \ZB(\tilde G)$. If $g$ is semisimple, then $g \in (\tilde x')^{\tilde G} \cdot (\tilde y)^{\tilde G}$ by
\cite[Lemma 5.1]{GT1}. The same conclusion holds if $g$ is non-semisimple but has large enough support $\supp(g) > B$
with $q^{\sqrt{B}} \geq 2^{481}$ -- indeed, in this case $|\rho(g)/\rho(1)| \leq q^{-\sqrt{\supp(g)}/481} < 1/2$ and so
$$\biggl{|}\sum_{\chi \in \Irr(g)}\frac{\chi(\tilde x')\chi(y)\overline\chi(g)}{\chi(1)}\biggr{|} >
    1-\biggl{|}\frac{\rho(\tilde x')\rho(\tilde y)\overline\rho(g)}{\rho(1)}\biggr{|} > 1-1 =0.$$
It therefore remains to consider the case $q$ is bounded and $\supp(g) \leq B$, in which case we may assume $n > B+6$, and so
$g$ acting on the natural module $\F_q^{2n}$ has a primary eigenvalue $\lambda = \pm 1$ by \cite[Proposition 4.1.2]{LST1}. In the case
$2 \nmid q$, the condition $2|n$ implies by \cite[Proposition 2.5.13]{KL} that $-1 \in \Omega^{+}_{2n}(q) = \tilde G$. Hence we
can multiply $g$ by a suitable central element of $\tilde G$ to ensure that $\lambda=1$. Now, using \cite[Lemma 6.3.4]{LST1} and
the assumption $n > B+6$, we can find a $g$-invariant decomposition $V = U \perp W$, where $\dim U = 10$, $g$ acts trivially on $U$
and $U$ is non-degenerate of type $+$, whence $W$ is non-degenerate of type $+$ of dimension $2n-10$.  By \cite[Theorem 2.6]{MSW},
we can find regular semisimple elements $\tilde u$ and $\tilde v$ of type a maximal torus of order $q^{n-5}-1$ and
a maximal torus of order $(q^{n-6}+1)(q+1)$ in $H:= \Omega^+_{2n-10}(q)$ such that the $W$-component $h$ of $g$ is
$\tilde u^{h_1} \cdot \tilde v^{h_2}$ for some $h_1,h_2 \in H$.
We also fix a regular semisimple element $\tilde z \in \Omega^+_{10}(q)$ of type a maximal
torus of order $(q^3+1)(q^2+1)$. Now it is clear that $g = (\tilde z\tilde u)^{h_1}(\tilde z^{-1}\tilde v)^{h_2}$, and both
$\tilde z \tilde u$ and $\tilde z^{-1}\tilde v$ are derangements in $\Pi$.

Thus we have completed the proof of Theorem \ref{square}, namely of Theorem B.

\medskip

We conclude this section with a probabilistic result on derangements.
Recall that, for a permutation group $G$ and an element $g \in G$, $\Pr_{\cD(G),\cD(G)}(g)$
denotes the probability that two independently chosen random derangements $s, t \in \cD(G)$
satisfy $st = g$.

\begin{prop}\label{mix} Let $G$ be a finite simple transitive permutation group. 
\begin{enumerate}[\rm(i)]
\item $\Pr_{\cD(G),\cD(G)}$ converges to the uniform distribution on $G$ in the $L^1$ norm as $|G| \to \infty$.
Hence the random walk on $G$ with respect to its derangements as a generating set has mixing time two.

\item If $G$ is a group of Lie type of bounded rank, then $\Pr_{\cD(G),\cD(G)}$ converges to the uniform distribution on $G$ 
in the $L^{\infty}$ norm as $|G| \to \infty$.
\end{enumerate}
\end{prop}

\begin{proof}
By \cite[Theorem 2.5]{Sh2}, if $G$ is a finite simple group, and $x, y \in G$ are randomly chosen, then almost surely $\Pr_{x^G,y^G}$ 
converges to the uniform distribution $\bfU_G$ in the $L^1$ norm as $|G| \to \infty$. Hence the same holds for randomly chosen $x, y \in T$, 
where $T$ is any normal subset of $G$ of proportion bounded away from $0$. By Theorem \ref{FG} of Fulman and Guralnick
we may apply this to $T = \cD(G)$. This implies part (i). 

Part (ii) follows from part (iv) of Theorem A.
\end{proof}

We note that, by Corollary 6.9 of \cite{LS2}, if $T \subseteq \AAA_n$ is a normal subset of size at least
$e^{-(1/2 - \delta)n}|\AAA_n|$ for some fixed $\delta > 0$, then, as $n \to \infty$, the mixing time of the random walk 
on $\AAA_n$ with respect to the generating set $T$ is two. This provides an alternative proof of part (i)
for alternating groups. 

We also note that part (ii) above does not hold for alternating groups; indeed this follows from Theorem \ref{Strong Alt}
and its proof.

\section{Products of derangements in alternating groups}
First we need the following technical result:

\begin{prop}\label{an-stab}
Let $n \geq 5$, $n \neq 6,8,9,10$, and let
$$\cL_n:= \{ \ell \in \Z \mid 2 \nmid \ell, \lfloor 3n/4 \rfloor \leq \ell \leq n\}.$$
Suppose $H$ is a proper subgroup of $\AAA_n$ that satisfies the following condition.
\begin{enumerate}[\rm(a)]
\item If $n \leq 16$ then $H$ contains an $\ell_i$-cycle for at least the two largest members $\ell_i$ of $\cL_n$.
\item If $n \geq 17$ then $H$ contains an $\ell_i$-cycle for at least the three largest members $\ell_i$ of $\cL_n$.
\end{enumerate}
Then $2|n$ and $H \cong \AAA_{n-1}$, a point stabilizer in the natural action of $\AAA_n$ on $\Delta:=\{1,2, \ldots, n\}$.
\end{prop}

\begin{proof}
We proceed by induction on $n$, with the induction base verifying the cases where $n \leq 13$.

\smallskip
(i) If $n=5$, then $15$ divides $|H|$, and so $H = \AAA_5$ by \cite{ATLAS}. Similarly, if $n=7$, then $35$ divides $|H|$, and so $H = \AAA_7$
by \cite{ATLAS}. Suppose $n = 11$. As $11$ divides $|H|$, using \cite{ATLAS} we see that $H$ is contained in a maximal
subgroup $X \cong \mathrm{M}_{11}$ of $\AAA_{11}$. But this is a contradiction, since $X$ contains no element of order $9$
whereas $H$ contains a $9$-cycle. Next assume that $n=12$. Then $H$ contains an $11$-cycle and a $9$-cycle. Using \cite{ATLAS} we again see that $H$ is contained in a maximal subgroup $Y$ of $\AAA_{12}$, with $Y \cong \mathrm{M}_{12}$ or $Y \cong \AAA_{11}$, a point stabilizer. The former case is ruled out since $\mathrm{M}_{12}$ contains no element of order $9$. In the latter case, we must have
$H=\AAA_{11}$ by the $n=11$ result. If $n=13$, then $11 \cdot 13$ divides $|H|$ and so $H=\AAA_{13}$ by \cite{ATLAS}.

\smallskip
(ii) For the induction step, assume $n \geq 14$. First we consider the case $H$ is intransitive on $\Delta$. If $2 \nmid n$, then
$H$ contains an $n$-cycle and so it is transitive on $\Delta$, a contradiction. Hence $2|n$. Then we may assume that $H$ contains
the $(n-1)$-cycle $g=(1,2, \ldots,n-1)$. It follows that $\{1,2, \ldots,n-1\}$ and $\{n\}$ are the two $H$-orbits on $\Delta$, and so
$H \leq \Stab_{\AAA_n}(n) \cong \AAA_{n-1}$.
If in addition $n \geq 18$, then $n-1,n-3,n-5$ are the three largest members of $\cL_n$, and at the same time they are also the
three largest members of $\cL_{n-1}$. Applying the induction hypothesis to $n-1$, we obtain that $H = \Stab_{\AAA_n}(n)$, as stated.
Suppose $n=16$. Then $H \leq \AAA_{15}$ and it contains a $15$-cycle and a $13$-cycle. It follows that $H$ is transitive on
$\Delta':=\{1,2, \ldots,15\}$, and in fact it acts primitively on $\Delta'$. Now using \cite{GAP} we can check that
$\AAA_{15}$ and $\SSS_{15}$ are the only primitive subgroups of $\SSS_{15}$ that have order divisible by $13$. It follows that
$H = \AAA_{15}$.

\smallskip
(iii) We may now assume that $H$ is transitive on $\Delta$. Suppose that $H$ is imprimitive: $H$ preserves a partition
$\Delta = \Delta_1 \sqcup \Delta_2 \sqcup \ldots \sqcup \Delta_b$ with $1 < |\Delta_i|=a = n/b < n$. If $2|n$, then we may assume that
$H$ contains the $(n-1)$-cycle $g=(1,2, \ldots,n-1)$ and that $n \in \Delta_b$. Then $g$ fixes $\Delta_b$ and so must fix the
set $\Delta_b \smallsetminus \{n\}$ of size $a-1 < n-1$, a contradiction. Next, consider the case $2 \nmid n$. Then we may assume
that $H$ contains the $(n-2)$-cycle $h=(1,2, \ldots,n-2)$ and that $n \in \Delta_b$. Note that $a > 1$ divides $n$ which is odd, hence
$n/3 \geq a \geq 3$. Now $h$ fixes $\Delta_b$ and so must fix the set $\Delta_b \smallsetminus \{n\}$ of size
$a-1$ with $2 \leq a-1 < n-2$, again a contradiction.

\smallskip
(iv) Now we consider the remaining case where $H$ is primitive on $\Delta$.

If $n=14$, then $11 \cdot 13$ divides $|H|$. Using \cite{GAP} we can check that $H = \AAA_n$. Similarly, if
$15 \leq n \leq 17$, then $\AAA_n$ is the only primitive subgroup of $\AAA_n$ that has order divisible by $13$, whence $H = \AAA_n$.

From now on we may assume $n \geq 18$ and let $H_1:= \Stab_H(1) \leq \AAA_{n-1}$. First we consider the case $2|n$. Then
$H$ contains an $(n-1)$-cycle $g$, an $(n-3)$-cycle $h$, and an $(n-5)$-cycle $k$. Since $H$ is transitive on $\Delta$, we may replace
$g$ by an $H$-conjugate so that $g(1)=1$, and similarly $h(1)=1$ and $k(1)=1$. Thus $H_1 \leq \AAA_{n-1}$ contains $g$, $h$, and $k$,
and $n-1,n-3,n-5$ are the first three members of $\cL_{n-1}$. By the induction hypothesis applied to $H_1$, we have $H_1 = \AAA_{n-1}$.
As $H$ is transitive on $\Delta$, it follows that $H=\AAA_n$.

\smallskip
(v) Now we may assume that $2 \nmid n \geq 19$. Arguing as above, we may assume that $H_1$ contains
an $(n-2)$-cycle $s=(3,4, \ldots,n)$. Assume in addition that $H_1$ is intransitive on $\{2,3, \ldots,n\}$. Since
$H_1 \ni s$, it follows that $\{1\}$, $\{2\}$, and $\{3,4, \ldots,n\}$ are the $3$ $H_1$-orbits on $\Delta$. Note that $H_2 := \Stab_H(2)$
now contains $H_1$ and $|H_2|=|H|/n=|H_1|$, whence $H_2 = H_1$. We claim that for any $i \in \Delta$, there is a unique
$i^\star \in \Delta \smallsetminus \{i\}$ such that
\begin{equation}\label{eq:an1}
  \Stab_H(i)=\Stab_H(i^\star).
\end{equation}
(Indeed, using transitivity of $H$, we can find $x \in H$ such that $i = x(1)$, whence \eqref{eq:an1} holds for $i^\star:=x(2)$. Conversely, if
$\Stab_H(i)=\Stab_H(j)$ for some $j \neq i$, then conjugating the equality by $x$, we see that $H_1=\Stab_H(1)$ fixes
$x^{-1}(j) \neq x^{-1}(i)=1$. The orbit structure of $H_1$ on $\Delta$ then shows that $x^{-1}(j)=2$, and so $j=x(2)=i^\star$, and the claim
follows.) We also note that the uniqueness of $i^\star$ and \eqref{eq:an1} imply that $(i^\star)^\star=i$. Hence, the set $\Delta$ is partitioned
into pairs $\{j_1,j_1^\star\}, \ldots, \{j_m,j_m^\star\}$, which is impossible since $2 \nmid n$.

We have shown that $H_1$ is transitive on $\{2,3, \ldots,n\}$, and so $H$ is doubly transitive on $\Delta$. In particular, $H$ has a unique
minimal normal subgroup $S$, which is either elementary abelian or a non-abelian simple group, see \cite[Proposition 5.2]{Cam}.
Suppose we are in the former case. Then one may identify $\Delta$ with the vector space $\F_p^d$ for some prime
$p$ with $p^d=n$, $S$ with the group of translations $t_v: u \mapsto u+v$ on $\F_p^d$, $1 \in \Delta$ with the zero vector in
$\F_p^d$, and $H_1$ with a subgroup of $\GL(\F_p^d)$. Since $2 \nmid n$, $p > 2$, and so $H_1$ is imprimitive on
$\F_p^d \smallsetminus \{0\}$ (indeed, it permutes the sets of nonzero vectors of $(p^d-1)/(p-1)$ $\F_p$-lines). On the other hand,
the presence of the $(n-2)$-cycle $s \in H_1$ shows (as in (iii)) that the transitive subgroup $H_1$ must be primitive on
$\F_p^d \smallsetminus \{0\}$, a contradiction.

We have shown that $S$ is simple, non-abelian. Now we can use the list of $(H,S,n)$ as given in \cite{Cam}. The possibility
$(H,S,n)=(\mathrm{M}_{23},\mathrm{M}_{23},23)$ is ruled out since $H$ must contain the element $s$ of order $21$. Next,
if $(S,n)=(\tw2 B_2(q),q^2+1)$ with $q=2^{2f+1} \geq 8$, then $S \lhd H \leq \Aut(S) = S \cdot C_{2f+1}$. This is impossible, since
$H$ contains the element $s$ of order $q^2-1$. Similarly, if $(S,n)=(\PSU_3(q),q^3+1)$ with $q=2^e \geq 4$, then
$S \lhd H \leq \Aut(S) = \PGU_3(q) \cdot C_{2e}$. This is again impossible, since
$H$ contains the element $s$ of order $q^3-1$. Next, if $(S,n)=(\SL_2(q),q+1)$ with $q=2^e \geq 8$, then
$S \lhd H \leq \Aut(S) = \SL_2(q) \cdot C_{e}$. This is again impossible, since
$H$ contains the element of order $n-4=q-3$.

The proper containment $H < \AAA_n$ now leaves only possibility that
$$(S,n) = (\PSL_d(q),(q^d-1)/(q-1))$$
with $d \geq 3$, and we may assume
that $S$ and $H$ act on the $(q^d-1)/(q-1)$ $\F_q$-lines of the vector space $\F_q^d = \langle e_1,e_2,\ldots,e_d \rangle_{\F_q}$.
Since $H$ is doubly transitive, we may assume that the two fixed points of the $(n-2)$-cycle $s$ are
$\langle e_1 \rangle_{\F_q}$ and $\langle e_2 \rangle_{\F_q}$. In this case, $s$ acts on the set of $q+1$ $\F_q$-lines of
$\langle e_1,e_2 \rangle_{\F_q}$, fixing two of them. This is again impossible, since $s$ permutes cyclically the other
$n-2$ $\F_q$-lines.
\end{proof}

\begin{thm}\label{an-main}
Let $G \leq \mathrm{Sym}(\Omega)$ be a finite transitive permutation group. Suppose that $G \cong \AAA_n$ for some $n \geq 5$. Then
every element in $G$ is a product of two derangements.
\end{thm}

\begin{proof}
(a) Fix a symbol $\alpha \in \Omega$ and consider the point stabilizer $H:=\Stab_G(\alpha)$. We also consider the natural permutation
action of $G$ on $\Delta:=\{1,2, \ldots,n\}$. The cases $5 \leq n \leq 10$ can be checked directly using \cite{GAP}, so we will assume
that $n \geq 11$.

In the notation of Proposition \ref{an-stab}, suppose first that there is some $\ell \in \cL_n$ such that $H$ does {\bf not} contain any
$\ell$-cycle. In other words, any $\ell$-cycle in $G = \AAA_n$ is a derangement on $\Omega$. By the main result of
\cite{B}, the choice of $\ell$ ensures that every element in $G$ is a product of two $\ell$-cycles, hence a product of two derangements
(on $\Omega$).

It remains to consider the case where $H$ contains an $\ell$-cycle for any $\ell \in \cL_n$. By Proposition \ref{an-stab}, this implies that
$2|n$ and $H=\Stab_G(1)$, and thus $\Omega=\Delta$. We will now show that every element $g \in G$ is a product of
two derangements on $\Delta$. (Presumably this also follows from \cite{Xu}, but, for the reader's convenience, we give a short direct
proof.)

\smallskip
(ii) We will again proceed by induction on $n$, with the induction base $5 \leq n \leq 10$ already checked.

\smallskip
(b1) For the induction step, suppose that $g$ fixes at least $2$ points in $\Delta$, say $g(i)=i$ for $i = 1,2$. Since $n \geq 11$, we have
$n-2 \geq \lfloor 3n/4 \rfloor$. Viewing $g \in \AAA_{n-2}$, by the main result of \cite{B} we have that $g=x_1x_2$ is a product of
two $(n-2)$-cycles $x_1,x_2 \in \SSS_{n-2}$. It follows that $g=\tilde{x}_1\tilde{x}_2$, with $\tilde{x}_1=x_1(1,2)$ and
$\tilde{x}_2=x_2(1,2)$ being derangements in $\AAA_n$.


\smallskip
(b2) Suppose now that $g=g_1g_2 \in \AAA_m \times \AAA_{n-m}$ with $5 \leq m \leq n/2$. By the induction hypothesis,
$g_i = y_iz_i$ with $y_1,z_1 \in \AAA_m$ and $y_2,z_2 \in \AAA_{n-m}$ being derangements. It follows that
$g=(y_1y_2)(z_1z_2)$ with $y_1y_2 \in \AAA_n$ and $z_1z_2 \in \AAA_n$ being derangements. In particular, we are done if,
in the decomposition of $g$ into disjoint cycles, $g$ contains a cycle of odd length $c$ where $5 \leq c \leq n-5$. We are also
done if $c=3$: indeed, if $g = (1,2,3)h$ with $h \in \AAA_{n-3}$ disjoint from $(1,2,3)$, then we can write
$h=h_1h_2$ with $h_i \in \AAA_{n-3}$ being derangements, and so $g=((1,3,2)h_1) \cdot ((1,3,2)h_2)$ is a product of
two derangements. Together with (b1), we are also done in the case $c=n-3$.

\smallskip
(b3) Suppose $g$ contains at least two cycles $t_1$, $t_2$ of even length $d_1$, $d_2$ in its disjoint cycle decomposition.
If $6 \leq d_1+d_2 \leq n-6$, we are done by the previous step (b2), by taking $g_1:=t_1t_2$.
We are also done if $d_1+d_2=4$: indeed, if $g = (1,2)(3,4)h$ with
$h \in \AAA_{n-4}$ disjoint from $(1,2)(3,4)$, then we can write  $h=h_1h_2$ with $h_i \in \AAA_{n-4}$ being derangements,
and so $g=((1,3)(2,4)h_1) \cdot ((1,4)(2,3)h_2)$ is a product of two derangements.

\smallskip
(b4) The above steps leave only the following two cases for the disjoint cycle decomposition of $g$ (up to conjugation).

$\bullet$ $g=g_1g_2$ where $g_1$ is an $a$-cycle, $g_2$ is an $(n-a)$-cycle, and $2|a$. Here, if $4 \leq a \leq n-4$, then
$g=g^2 \cdot g^{-1}$, with $g^2$ and $g^{-1}$ being derangements. In the remaining case, say $g = (1,2,\ldots,n-2)(n-1,n)$,
setting $h = (1,2,\ldots,n-3,n-1)(n-2,n)$, we see that $gh$ consists of two disjoint $n/2$-cycles and is therefore a derangement, while
$g = (gh)(h^{-1})$.

$\bullet$ $g = (1,2, \ldots,n-1)$. Setting $h = (1,n-3)(2,3, \ldots,n-4,n-2,n-1,n) \in \AAA_n$, we see that
$$gh = (1,n-2)(2,4,6, \ldots,n-4,n-1,n,3,5, \ldots,n-3)$$
is a derangement, while $g = (gh)(h^{-1})$.
\end{proof}


\begin{thebibliography}{EGA IV$_2$}

\bibitem[Atlas]{ATLAS}
  J.H. Conway, R.T. Curtis, S.P. Norton, R.A. Parker and R.A. Wilson,
`{\it ATLAS of Finite Groups}', Clarendon Press, Oxford, 1985.

\bibitem[BNP]{BNP} L. Babai, N. Nikolov and L. Pyber, Product growth and mixing in finite groups, (Extended abstract.)
In: `{\it Proc. 19th Ann. Symp. Discrete Algorithms (SODA'08), ACM--SIAM 2008}', pp. 248--257.

\bibitem[B]{B} E. Bertram, Even permutations as a product of two conjugate cycles,
{\it J. Comb. Theory Ser. A} {\bf 12} (1972), 368--380.

\bibitem[BGT]{BGT} E. Breuillard, B. Green and T. Tao, Approximate subgroups of linear groups,
{\it Geom. Funct. Anal.} {\bf 21} (2011), 774--819.

\bibitem[BG]{BG} T.C. Burness and M. Giudici, `{\it Classical Groups, Derangements and Primes}', 
Australian Mathematical Society Lecture Series, {\bf 25}, Cambridge University Press, Cambridge, 2016. xviii+346 pp.

\bibitem[Cam]{Cam} P. J. Cameron, Finite permutation groups and finite simple groups, {\it Bull. Lond. Math. Soc.} {\bf 13} (1981), 1--22.

\bibitem[CC]{CC} P.J. Cameron and A.M. Cohen,
On the number of fixed point free elements in a permutation group. A collection of contributions in honour of Jack van Lint,
{\it Discrete Math.} {\bf 106/107} (1992), 135--138.

\bibitem[Ca]{Ca}
  R. Carter, `{\it Finite Groups of Lie type: Conjugacy Classes
and Complex Characters}', Wiley, Chichester, 1985.

\bibitem[Cu]{Cur}
  C.W. Curtis, Truncation and duality in the character ring of a finite group of Lie type,
{\it J. Algebra} {\bf  62} (1980),  320--332.

\bibitem[SGA $4\frac12$]{SGA 4.5} P. Deligne,
`{\it Cohomologie \'Etale. S\'eminaire de G\'eom\'etrie Alg\'ebrique du Bois-Marie SGA $4\frac12$}',
Lecture Notes in Mathematics {\bf 569}, Springer-Verlag, Berlin, 1977.

\bibitem[De]{Weil II} P. Deligne, La conjecture de Weil. II,
\textit{Inst.\ Hautes \'Etudes Sci.\ Publ.\ Math.} {\bf 52} (1980), 137--252.

\bibitem[DL]{DL} P.~Deligne and G.~Lusztig, Representations of reductive
groups over finite fields,
\textit{Ann. of Math.} \textbf{103}  (1976), 103--161.

\bibitem[DM]{DM}
F. Digne and J. Michel, `{\it Representations of Finite Groups of
Lie Type}', London
Mathematical Society Student Texts {\bf 21}, Cambridge University Press, 1991.

\bibitem[D]{D} J.D. Dixon, Random sets which invariably generate the symmetric group, {\it Discrete Math.} {\bf 105} (1992),  25--39.


\bibitem[EFG]{EFG} S. Eberhard, K. Ford and B. Green, Invariable generation of the symmetric group,
{\it Duke Math. J.} {\bf 166} (2017), 1573--1590.

\bibitem[EG]{EG} E.W. Ellers and N. Gordeev, On conjectures of J. Thompson and O. Ore,
{\it Trans. Amer. Math. Soc.} {\bf 350} (1998), 3657--3671.

\bibitem[ET]{ET} P. Erd\H{o}s and P. Tur{\'a}n, On some problems of a statistical group-theory. II,
{\em Acta Math. Acad. Sci. Hungar.} \textbf{18} (1967), 151--163.

\bibitem[FKS]{FKS} B. Fein, W.M. Kantor and M. Schacher, Relative Brauer groups. II,
{\it J. Reine Angew. Math.} {\bf 328} (1981), 39--57.


\bibitem[FG1]{FG1} J. Fulman and R. Guralnick, Bounds on the number and sizes of conjugacy classes in finite Chevalley groups
with applications to derangements, {\it Trans. Amer. Math. Soc.} {\bf 364} (2012), 3023--3070.


\bibitem[FG2]{FG2} J. Fulman and R. Guralnick, Derangements in subspace actions of finite classical groups,
{\it Trans. Amer. Math. Soc.} {\bf 369} (2017), 2521--2572.


\bibitem[FG3]{FG3} J. Fulman and R. Guralnick, Derangements in finite classical groups for actions related to extension field and
imprimitive subgroups and the solution of the Boston-Shalev conjecture, {\it Trans. Amer. Math. Soc.} {\bf 370} (2018), 4601--4622.

\bibitem[GAP]{GAP}
  The GAP group, `{\it {\sf GAP} - groups, Algorithms, and Programming}', Version 4.8.7, 2017,\\
{\url http://www.gap-system.org}.

\bibitem[Gl]{Gluck} D. Gluck,
Character value estimates for nonsemisimple elements,
\textit{J.\ Algebra} \textbf{155} (1993), 221--237.

\bibitem[Go]{G} W.T. Gowers, Quasirandom groups, {\it Combin. Probab. Comput.} {\bf 17} (2008), 363--387.

\bibitem[EGA IV$_2$]{EGA IV 2}  A. Grothendieck,
 \'El\'ements de g\'eom\'etrie alg\'ebrique. IV. \'Etude locale des sch\'emas et des morphismes de sch\'emas. II,
{\it Inst. Hautes \'Etudes Sci. Publ. Math.} {\bf 24}, 1965.

\bibitem[EGA IV$_3$]{EGA IV 3}  A. Grothendieck,
\'El\'ements de g\'eom\'etrie alg\'ebrique. IV. \'Etude locale des sch\'emas et des morphismes de sch\'emas. III,
{\it Inst. Hautes \'Etudes Sci. Publ. Math.} {\bf 28}, 1966.



\bibitem[GLT1]{GLT1} R.M. Guralnick, M. Larsen and P.H. Tiep,
Character levels and character bounds, {\it Forum of Math. Pi} {\bf 8} (2020), e2, 81 pages.

\bibitem[GLT2]{GLT} R.M. Guralnick, M. Larsen and P.H. Tiep,
Character levels and character bounds. II.
arXiv:1904.08070v1.

\bibitem[GLBST]{GLOST} R.M. Guralnick, M.W. Liebeck, E.A. O'Brien, A. Shalev and P.H. Tiep,
Surjective word maps and Burnside's $p^aq^b$ theorem, {\it Invent. Math.} {\bf 213} (2018), 589--695.

\bibitem[GL]{GL} R.M. Guralnick and F. L\"ubeck, On $p$-singular elements in Chevalley groups in characteristic
$p$, in: `{\it Groups and Computation, III}' (Columbus, OH, 1999), Ohio State Univ. Math. Res.
Inst. Publ. {\bf 8}, de Gruyter, Berlin, 2001, pp. 169--182.

\bibitem[GM]{GM} R.M. Guralnick and G. Malle, Products of conjugacy classes and fixed point spaces,
{\it J. Amer. Math. Soc.} {\bf 25} (2012), 77--121.

\bibitem[GT1]{GT-12} R.M. Guralnick and P.H. Tiep,
A problem of Koll\'ar and Larsen on finite linear groups and crepant resolutions, {\it J. Europ. Math. Soc. } {\bf 14} (2012), 605--657.

\bibitem[GT2]{GT1} R.M. Guralnick and P.H. Tiep, Lifting in Frattini covers and a characterization of finite solvable groups,
{\it J. Reine Angew. Math.} {\bf 708} (2015), 49--72.

\bibitem[GT3]{GT2} R.M. Guralnick and P.H. Tiep, Effective results on the Waring problem for finite simple groups,
{\it Amer. J. Math.} {\bf 137} (2015), 1401--1430.


\bibitem[GW]{GW} R.M. Guralnick and D. Wan, Bounds for fixed point free elements in a transitive
group and applications to curves over finite fields, {\it Israel J. Math.} {\bf 101} (1997), 255--287.

\bibitem[He1]{H1} H.A. Helfgott, Growth and generation in $\SL_2(\Z/p\Z)$, {\it Ann. of Math.} {\bf 167} (2008), 
601--623.

\bibitem[He2]{H2} H.A. Helfgott, Growth in $\SL_3(\Z/p\Z)$, {\it J. Eur. Math. Soc. (JEMS)} {\bf 13} (2011), 761--851.

\bibitem[Hr]{Hr} E. Hrushovski, Stable group theory and approximate subgroups, {\it J. Amer. Math. Soc.}
{\bf 25} (2012), 189--243.

\bibitem[KLSh]{KLSh} W.M. Kantor, A. Lubotzky and A. Shalev, Invariable generation
and the Chebotarev invariant of a finite group, {\it J. Algebra} {\bf 348} (2011), 302--314.

\bibitem[KL]{KL}
P.B. Kleidman and M.W. Liebeck, `{\it The Subgroup Structure of the
Finite Classical Groups}', London Math. Soc. Lecture Note Ser. 
{\bf 129}, Cambridge University Press, $1990$.

\bibitem[La]{L} M. Larsen, Word maps have large image, {\it Israel J. Math.} {\bf 139} (2004), 149--156.

\bibitem[LS1]{LS1} M. Larsen and A. Shalev, Word maps and Waring type problems,
{\it J. Amer. Math. Soc.} {\bf 22} (2009), 437--466.

\bibitem[LS2]{LS2} M. Larsen and A. Shalev,
Characters of symmetric groups: sharp bounds and applications,
\textit{Invent.\ Math.}  \textbf{174} (2008), 645--687.


\bibitem[LST1]{LST1} M. Larsen, A. Shalev and P.H. Tiep,
The Waring problem for finite simple groups, {\it Ann.\ of Math.} {\bf 174} (2011), 1885--1950.

\bibitem[LST2]{LST2} M. Larsen, A. Shalev and P.H. Tiep,
Probabilistic Waring problems for finite simple groups,
\textit{Ann.\ of Math.} \textbf{190} (2019), 561--608.


\bibitem[LBST]{LOST} M.W. Liebeck, E.A. O'Brien, A. Shalev and P.H. Tiep,
The Ore conjecture, {\it J. Europ. Math. Soc.} {\bf 12} (2010),
939--1008.

\bibitem[LiSh1]{LiSh1} M.W. Liebeck and A. Shalev,
Fuchsian groups, finite simple groups, and representation varieties,
{\it Invent. Math.} {\bf 159} (2005), 317--367.

\bibitem[LiSh2]{LiSh2} M.W. Liebeck and A. Shalev,
Character degrees and random walks in finite groups of Lie type, {\it Proc.\ London Math.\ Soc.} {\bf 90} (2005), 61--86.

\bibitem[LSSh]{LSSh} M.W. Liebeck, G. Schul and A. Shalev,
Rapid growth in finite simple groups,
{\it Trans. Amer. Math. Soc.} {\bf 369} (2017), 8765--8779.

\bibitem[LM]{LM} F. L\"ubeck and G. Malle,
Murnaghan-Nakayama rule for values of unipotent characters in classical groups,
{\it Represent. Theory} {\bf 20} (2016), 139--161.

\bibitem[LP]{LP} T. \L{}uczak and L. Pyber, On random generation of the symmetric group,
{\it Combin. Probab. Comput.} {\bf 2} (1993) 505--512.

\bibitem[Lu1]{Odd-O} G. Lusztig,
Unipotent characters of the symplectic and odd orthogonal groups over a finite field,
\textit{Invent.\ Math.} \textbf{64} (1981), 263--296.

\bibitem[Lu2]{Lu} G. Lusztig, `{\it Characters of a Reductive Group over a Finite Field}',
Annals of Mathematics Studies {\bf 107}, Princeton University Press, 1984.

\bibitem[Ma1]{M1}
  G. Malle, Unipotente Grade imprimitiver komplexer
  Spiegelungsgruppen, {\em J. Algebra} 177 (1995), 768--826.

\bibitem[Ma2]{M}
  G. Malle, Almost irreducible tensor squares, {\it Comm. Algebra} {\bf 27} (1999),
1033--1051.

\bibitem[MM]{MM} G.~Malle and B.H.~Matzat,  `{\it Inverse Galois Theory}',
Springer Monographs in Mathematics. Springer-Verlag, Berlin, 1999.

\bibitem[MSW]{MSW} G. Malle, J. Saxl and T. Weigel,
Generation of classical groups.
{\it Geom.\ Dedicata} {\bf 49} (1994), 85--116.

\bibitem[Ng]{N}
  H.N. Nguyen, Low-dimensional complex characters of the symplectic and orthogonal
groups, {\it Comm. Algebra} {\bf 38} (2010), 1157--1197.

\bibitem[NP]{NP}  N. Nikolov and L. Pyber, Product decompositions of quasirandom groups and a Jordan type theorem,
{\it J. Europ. Math. Soc.} {\bf 13} (2011), 1063--1077.

\bibitem[Ol]{O}
  J.B. Olsson, Remarks on symbols, hooks and degrees of unipotent
  characters, {\em J. Combin. Theory Ser. A} 42 (1986), 223--238.

\bibitem[PPR]{PPR}
R.A. Pemantle, Y. Peres and I. Rivin, Four random permutations conjugated by an adversary generate $S_n$ with high
probability, {\it Random Structures Algorithms} {\bf 49} (2016), 409--428.

\bibitem[PS]{PS} L. Pyber and E. Szab{\'o},
Growth in finite simple groups of Lie type,
{\it J. Amer. Math. Soc.} {\bf 29} (2016), 95--146.

\bibitem[Se]{S} D. Segal, `{\it Words: Notes on Verbal Width in Groups}', London Math. Soc. Lecture Note
Series {\bf 361}, Cambridge University Press, Cambridge, 2009.

\bibitem[Sh1]{Sh1} A. Shalev, A theorem on random matrices and some applications, {\it J. Algebra}
{\bf 199} (1998), 124--141.

\bibitem[Sh2]{Sh2} A. Shalev, Mixing and generation in simple groups,
{\it J. Algebra} {\bf 319} (2008), 3075--3086.

\bibitem[Sh3]{Sh3} A. Shalev, Word maps, conjugacy classes, and a noncommutative Waring-type theorem,
{\it Ann. of Math.} {\bf 170} (2009), 1383--1416.

\bibitem[ST]{ST} P. Sin and P.H. Tiep, Rank 3 permutation modules for finite classical groups, {\it J. Algebra} {\bf 291} (2005), 551--606.

\bibitem[St]{Stanley} R.P. Stanley, `{\it Enumerative Combinatorics}', vol. {\bf 2},
Cambridge Studies in Advanced Mathematics, {\bf 62}, Cambridge University Press, Cambridge, 1999.

\bibitem[TZ1]{TZ1}
  P.H. Tiep and A.E. Zalesskii, Some characterizations of the Weil
representations of the symplectic and unitary groups, {\it J. Algebra} {\bf 192}
(1997), 130--165.

\bibitem[TZ2]{TZ2}
  P.H. Tiep and A.E. Zalesskii, Unipotent elements of finite groups of Lie type and realization fields of their complex representations,
{\it J. Algebra} {\bf 271} (2004), 327--390.

\bibitem[Va]{Varshavsky} Y. Varshavsky,
Lefschetz-Verdier trace formula and a generalization of a theorem of Fujiwara,
\textit{Geom.\ Funct.\ Anal.} \textbf{17} (2007), 271--319.

\bibitem[WW]{WW} E.T. Whittaker and G.N. Watson,
`{\it A Course of Modern Analysis. An Introduction to the General Theory of Infinite Processes and of Analytic Functions; with an Account of the Principal Transcendental Functions}, 4th edition,
Cambridge University Press, Cambridge, 1927.


\bibitem[Xu]{Xu} C.-H. Xu, The commutators of the alternating group, {\it Sci. Sinica} {\bf 14} (1965), 339--342.


\bibitem[Zs]{Zs}
  K. Zsigmondy, Zur Theorie der Potenzreste, {\it Monatsh. Math. Phys.} {\bf 3}
(1892), 265--284.

\end{thebibliography}
 \end{document}